\documentclass[11pt]{book}
\usepackage[a4paper]{geometry}
\usepackage{amsthm}
\usepackage[utf8]{inputenc}
\usepackage{amsmath}
\usepackage{amssymb}
\usepackage{mathtools}
\usepackage{cite}
\usepackage{subfigure}
\usepackage{tikz-cd}
\usepackage{MnSymbol}
\usepackage{stmaryrd}
\usetikzlibrary{arrows}
\usepackage{changepage}
\usepackage{bussproofs}
\usepackage{hyperref}
\usepackage{float}
\usepackage{bbm}
\usepackage{setspace}
\usepackage{tabularx}
\usepackage{tabulary}
\usepackage{array} 
\theoremstyle{plain}
\newtheorem*{lemma*}{Lemma}
\newtheorem{lemma}{Lemma}[section]
\newtheorem*{theorem*}{Theorem}
\newtheorem{theorem}{Theorem}[section]
\newtheorem*{proposition*}{Proposition}
\newtheorem{proposition}{Proposition}[section]
\newtheorem{corollary}{Corollary}[section]
\theoremstyle{remark}

\theoremstyle{definition}
\newtheorem{definition}{Definition}[section]

\DeclareSymbolFont{symbolsC}{U}{txsyc}{m}{n}
\SetSymbolFont{symbolsC}{bold}{U}{txsyc}{bx}{n}
\DeclareFontSubstitution{U}{txsyc}{m}{n}
\DeclareSymbolFont{AMSa}{U}{txsya}{m}{n}
\SetSymbolFont{AMSa}{bold}{U}{txsya}{bx}{n}
\DeclareFontSubstitution{U}{txsya}{m}{n}
\DeclareMathSymbol{\multimap}{\mathrel}{AMSa}{40}
\DeclareMathSymbol{\multimapinv}{\mathrel}{symbolsC}{18}

\newcommand{\vect}{\mathsf{vect}_{\mathbbm{k}}}
\newcommand{\vectc}{\mathsf{vect}_{\mathbbm{k}}}

\newcommand{\Vect}{\mathsf{Vect}_{\mathbbm{k}}}

\newcommand{\chu}{\mathsf{chu}}
\newcommand{\TVS}{\mathsf{TVS}}
\newcommand{\TVSw}{\mathsf{TVS}_w}
\newcommand{\TVSm}{\mathsf{TVS}_m}
\newcommand{\Vcat}{\mathcal{V}\mathsf{-Cat}}
\newcommand{\Vmod}{\mathcal{V}\mathsf{-Mod}}
\newcommand{\Prof}{\mathsf{Prof}}
\newcommand{\Set}{\mathsf{Set}}
\newcommand{\Cat}{\mathsf{Cat}}
\newcommand{\Gvect}{G\mathsf{-vect}_{\mathbbm{k}}}
\newcommand{\GVect}{G\mathsf{-Vect}_{\mathbbm{k}}}
\newcommand{\Gvectc}{G\mathsf{-vect}_{\mathbbm{k}}}
\newcommand{\GVectc}{G\mathsf{-Vect}_{\mathbbm{k}}}
\newcommand{\Rep}{\mathsf{-Rep}}
\newcommand{\Repfd}{\mathsf{-Rep}_{\mathsf{f.d.}}}
\newcommand{\Fun}{\mathsf{Fun}}
\newcommand{\End}{\mathsf{End}}
\newcommand{\unit}{\mathbbm{1}}
\newcommand{\Cob}{\mathsf{Cob}}
\newcommand{\ihoml}{\underline{\textnormal{Hom}}^l}
\newcommand{\ihomr}{\underline{\textnormal{Hom}}^r}
\newcommand{\Lex}{\mathcal{L}ex}
\newcommand{\Rex}{\mathcal{R}ex}
\newcommand{\modl}{\mathsf{-mod}}
\newcommand{\modr}{\mathsf{mod-}}
\newcommand{\sumf}{\textnormal{sum}_{<}}
\newcommand{\WQF}{\textnormal{WQF}(G, \mathbbm{k}^{\times})}
\newcommand{\WSQF}{\textnormal{WSQF}(G, \mathbbm{k}^{\times})}
\newcommand{\WRQF}{\textnormal{WRQF}(G, \mathbbm{k}^{\times})}
\newcommand{\QF}{\textnormal{QF}(G, \mathbbm{k}^{\times})}

\newcommand{\RMS}{\textnormal{RMS}(G, \mathbbm{k})}
\newcommand{\WRMS}{\textnormal{WRMS}(G, \mathbbm{k})}
\newcommand{\charac}{\textnormal{Hom}(G, \mathbbm{k}^{\times})}
\usepackage{color}
\definecolor{bluec}{HTML}{0033FF}
\definecolor{redc}{HTML}{FF1717}
\begin{document}

\newgeometry{centering}
\begin{titlepage}
\begin{center}
\vfill
\Huge
\setstretch{1.5}
\textbf{Generalised Duality Theory\\ for Monoidal Categories\\and\\ Applications\\}
\vfill
\includegraphics[scale=0.075]{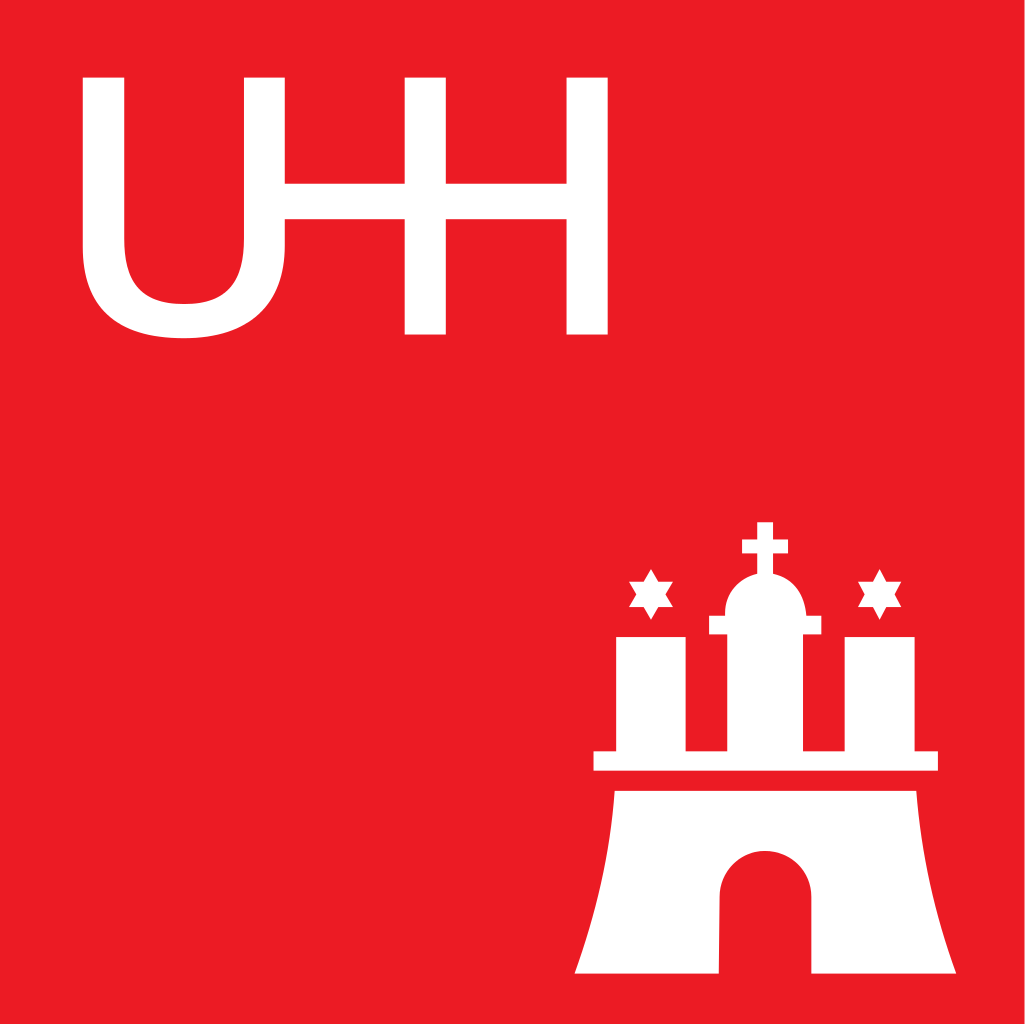}
\vfill
\LARGE
Stefan Jens Zetzsche\\
University of Hamburg

\vfill
\LARGE 
\setstretch{1.5}
A thesis submitted for the degree of \\
\textit{M.Sc. Mathematics}\\
September 2018




\end{center}
\end{titlepage}
\restoregeometry

\tableofcontents

\chapter{Introduction}

Since its introduction by Mac Lane and Eilenberg in the 1940s, category theory has proved to be universally important not despite but due to its simplicity. In particular, it is used in mathematical aspects of field theory (e.g. \cite{atiyah1988topological}, \cite{fuchs2002tft}), quantum computation (e.g. \cite{vicary2012higher}, \cite{abramsky2004categorical}), knot theory, logic and proof theory (e.g. \cite{dunn2016coherence}, \cite{mellies2009categorical}, \cite{andregeometry}).
The main subject of this thesis is the notion of \textit{duality} in monoidal categories and its applications. 

  One of the most prominent examples of a duality is arguably found in the \textit{linear dual} 
\begin{equation*}
	V^{\star} := \textnormal{Hom}_{\mathbbm{k}}(V, \mathbbm{k}) 
\end{equation*}
 of some $\mathbbm{k}$-vector space $V$. To this end, let us consider the category with $\mathbbm{k}$-vector spaces of arbitrary dimension as objects and $\mathbbm{k}$-linear maps as morphisms. This category enjoys a variety of prototypical properties. Foremost, it is \textit{monoidal} with the usual tensor product $\otimes$ of vector spaces. The unit of the tensor product is given by the field $\mathbbm{k}$, since for any $\mathbbm{k}$-vector space $V$, the underlying $\mathbbm{k}$-action witnesses $V \otimes \mathbbm{k} \cong \mathbbm{k} \otimes V \cong V$. Moreover, the monoidal structure is \textit{symmetric}, since for any pair of vector spaces $V,W$ there exist a natural isomorphism $V \otimes W \cong W \otimes V$. It is \textit{closed} in the sense that for any two spaces $V,W$ there exists a vector space $\underline{\textnormal{Hom}}(V,W)$, in this case the morphism set $\textnormal{Hom}_{\mathbbm{k}}(V,W)$ with its natural vector space structure, and a natural isomorphism 
\begin{equation}
\label{internalhomintro}
	\textnormal{Hom}_{\mathbbm{k}}(U \otimes V, W) \cong \textnormal{Hom}_{\mathbbm{k}}(U,\underline{\textnormal{Hom}}(V,W)).
\end{equation}
A vector space and its linear dual are related in the following ways. For every morphism $f: V \rightarrow W$ one can define a dual morphism $f^{\star}: W^{\star} \rightarrow V^{\star}$ as $f^{\star}(w')(v) = w'(f(v))$. Moreover, for a  vector space $V$ there exists an \textit{evaluation} map $\varepsilon: V \otimes V^{\star} \rightarrow \mathbbm{k}$ defined by $\varepsilon(v \otimes v') := v'(v)$, and for a finite-dimensional vector space $V$ a \textit{coevaluation} map $\eta: \mathbbm{k} \rightarrow V^{\star} \otimes V$ defined by $\eta(1) := \sum_i v'_i \otimes v_i$ for a finite basis $(v_i)_{i \in I}$ of $V$ and corresponding dual basis $(v'_i)_{i \in I}$ of $V^{\star}$ with $v'_i(v_j) := \delta_{ij}$. One can show that the definition of $\eta$ does not depend on the chosen basis $(v_i)_{i \in I}$.
Both maps satisfy the relations
\begin{align}
\label{snakeintro}
\begin{split}
	&( V \cong V \otimes \mathbbm{k} \overset{\textnormal{id} \otimes \eta}{\rightarrow} V \otimes (V^{\star} \otimes V) \cong (V \otimes V^{\star}) \otimes V \overset{\varepsilon \otimes \textnormal{id}}{\rightarrow} \mathbbm{k} \otimes V \cong V ) \quad  = \quad \textnormal{id}_V, \\
&( V^{\star} \cong \mathbbm{k} \otimes V^{\star} \overset{\eta \otimes \textnormal{id}}{\rightarrow} (V^{\star} \otimes V) \otimes V^{\star} \cong V^{\star} \otimes (V \otimes V^{\star}) \overset{\textnormal{id} \otimes \varepsilon}{\rightarrow} V^{\star} \otimes \mathbbm{k} \cong V^{\star} ) \quad =  \quad \textnormal{id}_{V^{\star}}
\end{split}
\end{align}
and a symmetric variant of it. As a consequence, if $V$ is finite-dimensional, the map $V \rightarrow V^{\star \star}$ given by $v \mapsto \textnormal{ev}_v$ with $\textnormal{ev}_v(v') := v'(v)$ for $v' \in V^{\star}$ witnesses an isomorphism between $V$ and its double dual, 
\begin{equation}
\label{doubledualvect}
	V \cong V^{\star \star} = \underline{\textnormal{Hom}}(\underline{\textnormal{Hom}}(V, \mathbbm{k}), \mathbbm{k}).
\end{equation}
Another observation shows that for finite-dimensional vector spaces $V,W$ there exists a natural isomorphism
\begin{equation}
\label{curryingintro}
		\textnormal{Hom}_{\mathbbm{k}}(V \otimes W, \mathbbm{k}) \cong \textnormal{Hom}_{\mathbbm{k}}(V, W^{\star}) \cong \textnormal{Hom}_{\mathbbm{k}}(W, V^{\star})
\end{equation}
witnessed by currying, i.e. by mapping a form $\kappa: V \otimes W \rightarrow \mathbbm{k}$ to a linear map $V \rightarrow W^{\star}$ given as $v \mapsto (w \mapsto \kappa(v \otimes w))$.

In category theory the notion of duality is classically captured intrinsically. Indeed, a symmetric monoidal category $(\mathcal{C}, \otimes, \unit)$ with tensor product $\otimes$ and tensor unit $\mathbbm{1}$ is \textit{rigid}, if for every $x \in \mathcal{C}$ there exists an object $x^{\star} \in \mathcal{C}$ and morphisms $\varepsilon: x\otimes x^{\star} \rightarrow \mathbbm{1}$, $\eta: \mathbbm{1} \rightarrow x^{\star} \otimes x$, satisfying \eqref{snakeintro}. This definition can be reformulated in a straightforward way for the non-symmetric case. Over the time rigidity has proved to be robust and appropriate for many situations. However, there exist situations in which it appears to be too restrictive. We will now discuss one of them.

For example, let $X \not= \emptyset$ be a non-empty set. The powerset $\mathcal{P}(X)$ induces a category with subsets $A \subseteq X$ as objects and a morphism $A \rightarrow B$ between objects $A,B \subseteq X$, if $A \subseteq B$. This category can be equipped with a symmetric monoidal structure $(\mathcal{P}(X), \cap, X)$ with tensor product given by the intersection of sets $\cap$ and tensor unit $X$. For every subset $A \subseteq X$ there exists a candidate for a dual given by the complement $A^{c} = X \setminus A$. Taking the complement is contravariant as $A \subseteq B$ if and only if $B^c \subseteq A^c$. The monoidal category $(\mathcal{P}(X), \cap, X)$ is closed with internal hom $\underline{\textnormal{Hom}}_{\cap}(A,B) = A^c  \cup B$ for $A,B \subseteq X$. Taking the complement can be expressed in terms of the internal hom as $A^c = \underline{\textnormal{Hom}}_{\cap}(A, \emptyset)$.  Thus, as every subset equals its double dual, 
\begin{equation}
\label{doubledualsubset}
 A = (A^c)^c = \underline{\textnormal{Hom}}_{\cap}(\underline{\textnormal{Hom}}_{\cap}(A,\emptyset),\emptyset).
\end{equation} But $(\mathcal{P}(X), \cap, X)$ can not be rigid, since there exists no coevaluation $X \rightarrow A^c \cap A$. Indeed, $X \subseteq A^c \cap A = \emptyset$ is false for every subset $A \subseteq X$. However, there do exist morphisms $V \otimes V^{\star} \rightarrow \emptyset$ and $\emptyset \rightarrow V^{\star} \otimes V$, related to evaluation and coevaluation by substituting the monoidal unit $X$ with its dual $X^c = \emptyset$. Similarly equation \eqref{doubledualsubset} is related to equation \eqref{doubledualvect}. In fact, this foreshadows that we could have equally chosen the category $(\mathcal{P}(X), \cup, \emptyset)$ induced by $\mathcal{P}(X)$ with tensor product given by union $\cup$ and tensor unit $\emptyset$.

A notion capturing the duality of the previous example is given in the definition of $\star$\textit{-autonomous categories}. The original definition in \cite{barrautonomcat} was given for symmetric monoidal categories, and in \cite{barr1995nonsymmetric} generalised to possibly non-symmetric monoidal categories. Before we state the definition, note that in any closed symmetric category $(\mathcal{C}, \otimes, \mathbbm{1})$ condition \eqref{internalhomintro} assures for any $x,k \in \mathcal{C}$ the existence of natural isomorphisms $\textnormal{Hom}(\underline{\textnormal{Hom}}(x,k), \underline{\textnormal{Hom}}(x,k)) \cong 
\textnormal{Hom}(\underline{\textnormal{Hom}}(x,k) \otimes x, k)$ and $\textnormal{Hom}(x \otimes \underline{\textnormal{Hom}}(x,k), k) \cong \textnormal{Hom}(x,  \underline{\textnormal{Hom}}(\underline{\textnormal{Hom}}(x,k), k))$. Using the symmetry one can combine both and identify the identity $\textnormal{id}_{\underline{\textnormal{Hom}}(x,k)}$ with an unique morphism 
\begin{equation}
\label{inducedmorphismintro}
	x \rightarrow \underline{\textnormal{Hom}}(\underline{\textnormal{Hom}}(x,k), k).
\end{equation}

\begin{definition}
\label{starautonomdef1}
A symmetric $\star$\textit{-autonomous category} is given by a closed symmetric monoidal category $(\mathcal{C}, \otimes, \mathbbm{1})$ with internal hom $\underline{\textnormal{Hom}}(-,-)$, and an object $k \in \mathcal{C}$, called \textit{dualizing object}, such that the induced morphisms \eqref{inducedmorphismintro} are isomorphisms
\[ x \cong \underline{\textnormal{Hom}}(\underline{\textnormal{Hom}}(x,k), k)  \] for all $x \in k$.	
\end{definition}
 Every symmetric rigid category with duality $(-)^{\star}$ and monoidal unit $\mathbbm{1}$ is closed with internal hom $	\underline{\textnormal{Hom}}(x,y) := y \otimes x^{\star}$, and $\star$-autonomous with dualizing object $k := \mathbbm{1}$. In particular, the rigid category of finite-dimensional $\mathbbm{k}$-vector spaces is $\star$-autonomous with dualizing object $k := \mathbbm{k}$. In that case the morphism \eqref{inducedmorphismintro} maps an element $v \in V$ to $\textnormal{ev}_v \in V^{\star \star}$ with $\textnormal{ev}_v(v') = v'(v)$ for $v' \in V^{\star}$. The monoidal category $(\mathcal{P}(X), \cap, X)$ with internal hom $\underline{\textnormal{Hom}}_{\cap}(A,B) = B^c \setminus A$ is not rigid, but $\star$-autonomous with dualizing object $k := \emptyset = X^c$. Since $(A^c)^c = A$ for any subset $A \subseteq X$, the morphism \eqref{inducedmorphismintro} is given by the identity.
 
 The concept of $\star$-autonomous categories finds application in a broad spectrum of fields.  For example, in \cite{seely1987linear} it was shown that such categories provide models of Girard's linear logic, similar to the relation between cartesian closed categories and $\lambda$-calculus \cite{lambek1980lambda}. Linear logic can be used to model the logic of resource use. In \cite{girard1987linear} Girard states that “\textit{a completely new approach to the whole area between constructive logics and computer science is initiated}”. In \cite{boyarchenko2011duality} $\star$-autonomous categories are studied under the name \textit{Grothendieck-Verdier categories}. The main goal of this thesis is to summarise the theory of $\star$-autonomous categories and give a variety of examples.
\vspace{1em}\\
The outline of the remaining part of the thesis is as follows.
In the preliminary Chapter 2 we recapitulate pertinent definitions and basic facts. Most of the material is covered in the literature, for example in \cite{etingof2015tensor}. We will introduce the string diagram notation for monoidal categories, give a definition of rigid categories and recall related notions motivated by knot theory, such as braidings and twists.

In Chapter 3 we will introduce possibly non-symmetric $\star$-autonomous categories.
In particular, in \autoref{starautonomequivalent} we will show the equivalence of multiple different definitions of $\star$-autonomous categories. For example, one would like to generalise \eqref{curryingintro} instead of \eqref{doubledualvect} by defining a $\star$-autonomous category as a symmetric monoidal category $(\mathcal{C}, \otimes, \mathbbm{1})$ together with an object $k \in \mathcal{C}$, an equivalence $(-)^{\star}: \mathcal{C} \rightarrow \mathcal{C}^{\textnormal{opp}(1)}$, and natural isomorphisms 
\[ \textnormal{Hom}_{\mathcal{C}}(x \otimes y, k) \cong\ \textnormal{Hom}_{\mathcal{C}}(x , y^{\star}). \]
As a matter of fact, both definitions are equivalent, as we will see in \autoref{starautonomequivalent}. 

Apart from that, we will emphasise the subtle differences between $\star$-autonomous categories and rigid categories. For example, for every rigid category $(\mathcal{C}, \otimes, \mathbbm{1})$, the duality $(-)^{\star}:(\mathcal{C}, \otimes, \mathbbm{1}) \rightarrow (\mathcal{C}, \otimes, \mathbbm{1})^{\textnormal{opp}(0,1)}$ is monoidal, where the latter denotes the category induced from $(\mathcal{C}, \otimes, \mathbbm{1})$ by reversing the morphisms and the tensor product. 
This is not necessarily the case for every $\star$-autonomous category (however, by a non-trivial proof the double dual is again necessary monoidal \cite[Prop. 4.2.]{boyarchenko2011duality}) and motivates the definition of a second tensor product $\otimes'$ on $\mathcal{C}$ by
 \begin{equation}
 \label{secondtensorproductintro}
 	x \otimes' y :=\ ^{\star}(y^{\star} \otimes x^{\star})
 \end{equation}
 for $x,y \in \mathcal{C}$ and $^{\star}(-): \mathcal{C}^{\textnormal{opp}(1)} \rightarrow \mathcal{C}$ the quasi-inverse of the duality functor $(-)^{\star}: \mathcal{C} \rightarrow \mathcal{C}^{\textnormal{opp}(1)}$. In the case of the $\star$-autonomous category $(\mathcal{P}(X), \cap, X)$ with tensor product given by the intersection of subsets, the second tensor product corresponds, as one might expect, to the union of subsets, $\cup$. Indeed, using the definition \eqref{secondtensorproductintro} and De Morgan's law one finds
 \[
 (B^c \cap A^c)^c = A \cup B,
 \]
 for $A,B \subseteq X$.
 This indicates that $\star$-autonomous categories come in symmetric pairs of monoidal categories related by a duality and natural isomorphisms; in this particular example $(\mathcal{P}(X), \cap, X)$ and $(\mathcal{P}(X), \cup, \emptyset)$ with $(-)^c$. Emphasising this perspective, it was shown in \cite{seely1992weakly} that $\star$-autonomous categories correspond to \textit{linearly distributive categories with duality}. We introduce this notion in \autoref{linearlydistributivecategory} and discuss it in detail in Section 3.2.
 
Aside from their broad range of applications, another interesting aspect of $\star$-auton-\\omous categories is their relation to \textit{Frobenius algebras}. A Frobenius algebra is a finite-dimensional $\mathbbm{k}$-algebra $A$ equipped with a non-degenerate form $\sigma: A \otimes A \rightarrow \mathbbm{k}$ such that $\sigma(ab \otimes c) = \sigma(a \otimes bc)$ for $a,b,c \in A$. They have been studied already in the early 20th century \cite{nakayama1939frobeniusean}, \cite{nakayama1941frobeniusean}, \cite{nakayama1939frobeniusean}, \cite{brauer1937regular}, and in recent times the interest has been renewed due to their connection to $2$-dimensional quantum field theories \cite{abrams1996two}. \textit{Frobenius pseudomonoids} generalise Frobenius algebras to bicategories. A special class of Frobenius pseudomonoids are  those of which multiplication and unit morphisms have right adjoints, usually called $\dag$\textit{-Frobenius pseudmonoids}. A \textit{profunctor} $\mathcal{C} \nrightarrow \mathcal{D}$ is defined to be a functor $\mathcal{D}^{\textnormal{opp}} \times \mathcal{C} \rightarrow \Set$. Due to a short remark of \cite{street2004frobenius}, $\star$-autonomous categories correspond to $\dag$-Frobenius pseudomonoids in the bicategory of profunctors. In \autoref{frobeniusinducesstarautonom} and \autoref{autonominducesfrobenius} we work out a complete proof of this statement.

Chapter 4 of this thesis is dedicated to a variety of examples. 
Based on an observation in \cite{fuchs2016eilenberg}, we discuss a $\star$-autonomous structure on the category of (left or right) exact endofunctors of a finite linear category in  Section 4.1. To this end we repeat the notion of a generalised Nakayama functor and present a range of results related to end, respectively coend, calculus. Our own main results of this section are \autoref{gammamonoidal} and \autoref{rigidleftadjoint}. In particular, we show that the second tensor product is again given by the composition of endofunctors.

In Section 4.2 we generalise the classic duality theory known for graded vector spaces. A finite-dimensional (f.d.) $G$-graded vector space $V$ over some field $\mathbbm{k}$ for a finite group $G$ is given as a direct sum $V = \bigoplus_{g \in G} V_g$ of f.d. $\mathbbm{k}$-vector spaces indexed by $G$. The direct sum of two f.d. $G$-graded vector spaces is defined component wise. Every f.d. $G$-graded vector space is semisimple in the sense that it can be expressed as direct sum of simple f.d. $G$-graded vector spaces $\mathbbm{k}_g$ defined by $(\mathbbm{k}_g)_h := \delta_{gh} \mathbbm{k}$. It is natural to introduce the graded dual space $V^{\star}$ of a f.d. $G$-graded vector space $V$ as 
\begin{equation}
\label{dualityintro}
	(V^{\star})_g := (V_{g^{-1}})^{\star}.
\end{equation}
Together with the tensor product $V \otimes W := \bigoplus_{g \in G, hk=g} V_h \otimes W_k$ this defines the rigid monoidal category $\Gvect$ of f.d. $G$-graded vector spaces. By an result of Eilenberg and MacLane (cf. \cite{eilenberg1950cohomology}), if $G$ is abelian, the braided monoidal structures 
(i.e. associator, braiding) of $\Gvect$ correspond to the set $\textnormal{QF}(G, \mathbbm{k}^{\times})$ of quadratic forms on $G$ with values in $\mathbbm{k}^{\times}$. Our own main result \autoref{theoremweakribbon} of this section gives a similar statement for ribbon structures with respect to an arbitrary $\star$-autonomous structure on $\Gvect$. To this end we define the notion of weak quadratic forms in \autoref{weakquadraticformdef} and show that they can be uniquely decomposed into a product of a quadratic form and a character on $G$ in \autoref{weakquadraticformrepr}. 

Finally, Section 4.3 concludes with an investigation of the duality structure of topological vector spaces based on work of Barr \cite{barr2000autonomous}. The category $\chu$ consists of objects $(U,V,\langle -, - \rangle)$, with $U,V$ vector spaces over some field $\mathbbm{k}$ and $\langle -,- \rangle: U \otimes V \rightarrow \mathbbm{k}$ a linear pairing, together with appropriate morphisms. Barr showed that $\chu$ is $\star$-autonomous with trivial duality 
\[ (U,V,\langle -, - \rangle)^{\star} := (V,U,\langle -, - \rangle \circ \tau_{V,U}), \]
where $\tau_{V,U}: V \otimes U \rightarrow U \otimes V$ denotes the canonical braiding swapping arguments.
Using an equivalence between $\chu$ and the categories of weak, respectively Mackey topologized vector spaces, $\TVSw$ and $\TVSm$, one can induce a $\star$-autonomous structure on the latter categories. In \autoref{tvstensoruniversal} we show that one of the tensor products possesses a similar, but weaker universal property as the so called projective tensor product. This motivates some generalisation of the original duality structure proposed by Barr.
\vspace{1em}\\
To sum up, our own contribution to this thesis includes \autoref{dualsuniquebicat}, \autoref{lemmafrobpseudo}
, the proofs of \autoref{frobeniusinducesstarautonom} and \autoref{autonominducesfrobenius}, \autoref{gammamonoidal}, \autoref{symmetricfrobcor}, \autoref{rigidleftadjoint}, \autoref{internalhomgradedlemma}, \autoref{gradeddualeq}, \autoref{generaltensorlemma}, 
\autoref{weakeq}, \autoref{weakeq2}, \autoref{weakquadraticformrepr}, \autoref{wqfcharaclemma}, \autoref{symmetrycorollary}, \autoref{weakhochk}, \autoref{samebeta}, \autoref{wrqfwsqf}, \autoref{theoremweakribbon}, \autoref{internalhomtvs}, \autoref{dualitytvs}, \autoref{tensor1tvs}, \autoref{tensor2tvs}, \autoref{tvstensoruniversal}.


\vspace{2em}
\noindent \LARGE{\textbf{Acknowledgements}}
\vspace{1em}\\
\normalsize First of all, I would like to thank my supervisor Prof. Dr. Christoph Schweigert, who suggested the topic of this thesis to me. I am very grateful for his encouragement, patient supervising and support of my work. 
I also want to thank my co-supervisor Jun.-Prof. Dr. Simon Lentner, in particular for the helpful discussions about Section 4.2.
I have gained greatly from conversations with Vincent Koppen. 
Moreover, I'm thankful to Jamie Vicary, in particular for \textit{Globular}, the proof assistant I used to produce the string diagrams in this thesis. 
Lastly, I want to thank my parents for their continuous love, support and faith in me. \nocite{bar2016globular}

\chapter{Classic duality theory}

\section{Monoidal categories}

 A \textit{monoid} is a set $M$ together with an associative product $M \times M \rightarrow M$ and an unit element $\textnormal{1}_M$. Category theory is deeply related to the notion of a monoid.
For instance, in every category $\mathcal{C}$ the endomorphism set $\textnormal{Hom}_{\mathcal{C}}(x,x)$ of an arbitrary object $x \in \mathcal{C}$ forms a monoid with the composition and identity morphism $\textnormal{id}_x: x \rightarrow x$. Conversely, every monoid induces a category with one object. Apart from that, most of the categories we will encounter in this thesis are \textit{themself} monoidal -- they come with an additional structure, turning them into a categorified monoid. The definition of a monoidal category goes back to the work of Mac Lane, cf. \cite{maclane1963}, \cite{maclane1998}.

\begin{definition}[Monoidal category]
	A \textit{monoidal category} $(\mathcal{C}, \otimes, \mathbbm{1}) := (\mathcal{C}, \otimes, \mathbbm{1}, \alpha, \lambda, \varrho)$ is given by a category $\mathcal{C}$, a bifunctor $\otimes: \mathcal{C} \times \mathcal{C} \rightarrow \mathcal{C}$ called the \textit{tensor product}, an \textit{unit object} $\mathbbm{1} \in \mathcal{C}$ and three natural isomorphisms subject to a pentagon and a triangle axiom (cf. \cite{etingof2015tensor}): the \textit{associator} $\alpha: \otimes \circ (\otimes \times \textnormal{id}_{\mathcal{C}}) \Rightarrow \otimes \circ (\textnormal{id}_{\mathcal{C}} \times \otimes)$, the \textit{left unitor} $ \lambda: \otimes \circ (\mathbbm{1} \times \textnormal{id}_{\mathcal{C}}) \Rightarrow \textnormal{id}_{\mathcal{C}}$, and the \textit{right unitor} $\varrho: \otimes \circ (\textnormal{id}_{\mathcal{C}} \times \mathbbm{1}) \Rightarrow \textnormal{id}_{\mathcal{C}}$. A monoidal category $(\mathcal{C}, \otimes, \mathbbm{1}, \alpha, \lambda, \varrho)$ is \textit{strict}, if $\alpha, \lambda$ and $\varrho$ are identities.
\end{definition}

The associator guarantees that any two tensor products of finitely many objects are isomorphic. The pentagon axiom guarantees that there exists precisely \textit{one} such isomorphism. The latter result is known as \textit{Mac Lane's coherence theorem}.

\begin{definition}[Monoidal functor]
Let $(\mathcal{C}, \otimes_{\mathcal{C}}, \mathbbm{1}_{\mathcal{C}})$ and  $(\mathcal{D}, \otimes_{\mathcal{D}}, \mathbbm{1}_{\mathcal{D}})$ be monoidal categories.
	A (\textit{lax}) \textit{monoidal functor} $(F, \Phi, \phi): (\mathcal{C}, \otimes_{\mathcal{C}}, \mathbbm{1}_{\mathcal{C}}) \rightarrow (\mathcal{D}, \otimes_{\mathcal{D}}, \mathbbm{1}_{\mathcal{D}})$ is given by a functor $F: \mathcal{C} \rightarrow \mathcal{D}$, a natural transformation $\Phi:\otimes_{\mathcal{D}} \circ (F \otimes_{\mathcal{C}} F) \Rightarrow F \circ \otimes_{\mathcal{C}}$ and a morphism $\phi: \mathbbm{1}_{\mathcal{D}} \rightarrow F(\mathbbm{1}_{\mathcal{C}})$ subject to three commutativity constraints (cf. \cite{etingof2015tensor}). A monoidal functor $(F, \Phi, \phi)$ is \textit{strong}, if $\Phi$ and $\phi$ are invertible; it is a \textit{strict}, if $\Phi$ and $\phi$ are identities.
\end{definition}

Examples of monoidal categories include the category $\Set$ with the cartesian product of sets, the category $\Vect$ of arbitrary dimensional $\mathbbm{k}$-vector spaces with the usual tensor product, the category $\End(\mathcal{C}) := \Fun(\mathcal{C}, \mathcal{C})$ of endofunctors on some category $\mathcal{C}$ with the composition of functors, and the representation category $A\Rep$ of some arbitrary dimensional $\mathbbm{k}$-bialgebra $A$. 

For any category $\mathcal{C}$ there exists a \textit{dual category} $\mathcal{C}^{\textnormal{opp}(1)}$ which equals $\mathcal{C}$ but has reversed morphisms. For every functor $\mathcal{C} \rightarrow \mathcal{D}$ the \textit{dual functor} $F^{\textnormal{opp}(1)}: \mathcal{C}^{\textnormal{opp}(1)} \rightarrow \mathcal{D}^{\textnormal{opp}(1)}$ is defined as $F^{\textnormal{opp}(1)}(\overline{x}) := \overline{F(x)}$ and $F^{\textnormal{opp}(1)}(\overline{f}: \overline{x} \rightarrow \overline{y}) := (\overline{F(f)}: \overline{F(x)} \rightarrow \overline{F(y)})$. Not to be confused with the dual category is the next definition. 

\begin{definition}[Opposite category]
\label{oppositecategory}
	Let $\mathcal{C} = (\mathcal{C}, \otimes, \mathbbm{1}, \alpha, \lambda, \varrho)$ be a monoidal category and $\tau: \mathcal{C} \times \mathcal{C} \rightarrow \mathcal{C} \times \mathcal{C}$ the functor swapping the two arguments. 
	Then the \textit{opposite category}  $\mathcal{C}^{\textnormal{opp}(0)}$ is the monoidal category
	$\mathcal{C}^{\textnormal{opp}(0)} := (\mathcal{C}, \otimes \circ \tau, \mathbbm{1}, (\alpha^{-1}_{Z,Y,X})_{X,Y,Z \in \mathcal{C}}, \varrho, \lambda).$
\end{definition} 
	In particular, we will abbreviate $\mathcal{C}^{\textnormal{opp}(0,1)}$ for the monoidal category $(\mathcal{C}^{\textnormal{opp}(0)})^{\textnormal{opp}(1)} = (\mathcal{C}^{\textnormal{opp}(1)})^{\textnormal{opp}(0)}$.

\textit{Mac Lane's strictness theorem} \cite{mac2013categories} guarantees that every monoidal category is monoidally equivalent to a strict monoidal category. 
If not otherwise stated we will thus assume that we are working in a strict monoidal category. 

In particular, we will use the graphical calculus of string diagrams for strict monoidal categories as introduced by Street, Joyal et al. (e.g. \cite{joyal1991geometry}, \cite{andregeometry}). Diagrams are read from bottom to top. Objects correspond to strings, and morphisms correspond to points. The composition of morphism is given by vertically joining strings. Multiple strings or multiple objects in horizontal juxtaposition with each other should be understood as the tensor product of each other.  The tensor unit is usually omitted. Possible graphical representations of morphisms in some strict monoidal category $(\mathcal{C}, \otimes, \mathbbm{1})$ are for example \[
\includegraphics[height=0.5cm]{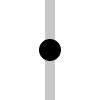} : x \rightarrow y, \quad
\includegraphics[height=0.5cm]{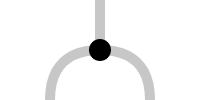}: x \otimes y \rightarrow z, \quad
\includegraphics[height=0.5cm]{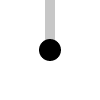}: \mathbbm{1} \rightarrow x.
\]

For the purpose of graphical calculus a particular identity proves to be very useful. The \textit{interchanger law} for morphisms $f: a \rightarrow b,\ f': b\rightarrow c$ and $g: x \rightarrow y,\ g': y \rightarrow z$ states that
$
(f' \circ f) \otimes (g' \circ g) = (f' \otimes g') \circ (f \otimes g).
$
In particular, after substituting with appropriate identities, one obtains
\[
(\textnormal{id}_b \otimes g) \circ (f \otimes \textnormal{id}_c) =
f \otimes g
=  (f \otimes \textnormal{id}_d) \circ (\textnormal{id}_a \otimes g),
\]
which corresponds in the graphical representation to the identity
\begin{equation}
\label{intertwiner}
\begin{tabular}{m{2em} m{0em} m{1.5em}}
	\includegraphics[height=1cm]{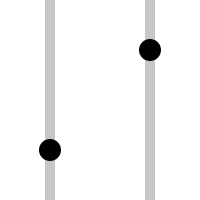} &= &\includegraphics[height=1cm]{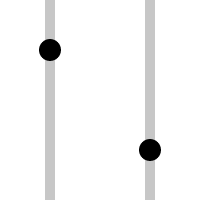}.
\end{tabular}
\end{equation}

\section{Rigid monoidal categories}

A $\mathbbm{k}$-vector space $V$ and its linear dual $V^{\star} := \textnormal{Hom}(V, \mathbbm{k})$ enjoy some remarkable relations. For this section the following observations are central. First, there exists a map $\varepsilon: V \otimes V^{\star} \rightarrow \mathbbm{k}$ defined by $\varepsilon(v \otimes v') := v'(v)$. Secondly, if $V$ is finite-dimensional, there exists a map $\eta: \mathbbm{k} \rightarrow V^{\star} \otimes V$ defined by $\eta(1) := \sum_i v'_i \otimes v_i$ for a finite basis $(v_i)_{i \in I}$ of $V$ and corresponding dual basis $(v'_i)_{i \in I}$ of $V^{\star}$ with $v'_i(v_j) := \delta_{ij}$. One can show that the definition of $\eta$ does not depend on the chosen basis $(v_i)_{i \in I}$.
Thirdly, if $V$ is finite-dimensional, both maps satisfy the relations
\begin{align}
\label{snake}
\begin{split}
	&( V \cong V \otimes \mathbbm{k} \overset{\textnormal{id} \otimes \eta}{\rightarrow} V \otimes (V^{\star} \otimes V) \cong (V \otimes V^{\star}) \otimes V \overset{\varepsilon \otimes \textnormal{id}}{\rightarrow} \mathbbm{k} \otimes V \cong V ) \quad  = \quad \textnormal{id}_V, \\
&( V^{\star} \cong \mathbbm{k} \otimes V^{\star} \overset{\eta \otimes \textnormal{id}}{\rightarrow} (V^{\star} \otimes V) \otimes V^{\star} \cong V^{\star} \otimes (V \otimes V^{\star}) \overset{\textnormal{id} \otimes \varepsilon}{\rightarrow} V^{\star} \otimes \mathbbm{k} \cong V^{\star} ) \quad =  \quad \textnormal{id}_{V^{\star}}
\end{split}
\end{align}
and a symmetric variant of it. The next definition captures the previous observations in categorical terms. 
\begin{definition}[Rigid category]
\label{rigid}
	Let $(\mathcal{C}, \otimes, \mathbbm{1}, \alpha, \lambda, \varrho)$ be a monoidal category. A \textit{left dual} to $x \in \mathcal{C}$ is an object $^{\star}x \in \mathcal{C}$ with morphisms $\varepsilon_x:\ ^{\star}x \otimes x \rightarrow \unit$ and $\eta_x: \unit \rightarrow x \otimes\ ^{\star}x$, called (left) \textit{evaluation} and (left) \textit{coevaluation} respectively, satisfying so called \textit{zig-zag} (or \textit{snake}) identities
\begin{equation}
	\label{rigidequation1}
	\varrho_x	\circ (\textnormal{id}_x \otimes \varepsilon_x) \circ \alpha_{x, ^{\star}x, x} \circ (\eta_x \otimes \textnormal{id}_x) \circ \lambda_x^{-1} \quad = \quad \textnormal{id}_x,
	\end{equation}
	\begin{equation}
		\label{rigidequation2}
			\lambda_{^{\star} x} \circ (\varepsilon_x \otimes \textnormal{id}_{^{\star}x}) \circ \alpha^{-1}_{^{\star}x, x, ^{\star}x} \circ (\textnormal{id}_{^{\star}x} \otimes \eta_{x}) \circ \varrho^{-1}_{^{\star}x} \quad = \quad \textnormal{id}_{^{\star}x}.
	\end{equation}
	A \textit{right dual} to $x \in \mathcal{C}$ is an object $x^{\star} \in \mathcal{C}$ with a (right) evaluation $\widetilde{\varepsilon}_x: x \otimes x^{\star} \rightarrow \unit$ and a (right) coevaluation $\widetilde{\eta}_x: \unit \rightarrow\  x^{\star} \otimes x$, satisfying analogous zig-zag identities. If for every object $x \in \mathcal{C}$ there exist left and right duals $^{\star}x,\ x^{\star}$, we call $(\mathcal{C}, \otimes, \mathbbm{1})$ \textit{rigid} with dualities $(-)^{\star}$ and $^{\star}(-)$.
\end{definition}
In the graphical calculus for strict monoidal categories left evaluation and left coevaluation are usually denoted as follows:
\[ \varepsilon_x \equiv \includegraphics[height=0.5cm]{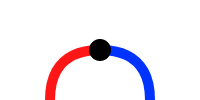}:\ ^{\star} x \otimes x \rightarrow \unit,\qquad \eta_x \equiv \includegraphics[height=0.5cm]{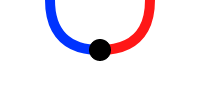}: \unit \rightarrow x \otimes\ ^{\star}x. \] 
To distinguish, we represent the object $x \in \mathcal{C}$ in blue, and its left dual $^{\star}x \in \mathcal{C}$ in red.
The snake identities \eqref{rigidequation1}, \eqref{rigidequation2} then can be formulated as
	 	\begin{equation}
	 	\label{zigzagidentities}
	 		\begin{tabular}{m{3.5em} m{0em} m{2em} m{2em} m{3.5em} m{0em} m{2em}}
	\includegraphics[height=1cm]{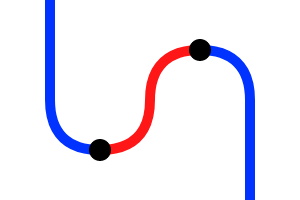} 
	&= 
	&\includegraphics[height=1cm, width=0.5cm]{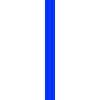} 
	& \textnormal{and}  
	&\includegraphics[height=1cm]{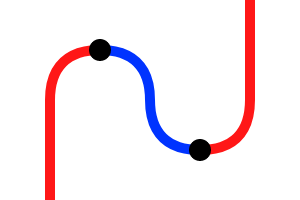} 
	&=  
	&\includegraphics[height=1cm, width=0.5cm]{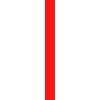}.
	\end{tabular}	
	 	\end{equation}
	 	 	
Analogously one can represent the right evaluation, the right coevaluation and the corresponding snake identities.

By construction the category $\vect$ of finite-dimensional vector spaces is rigid. Another, non-trivial example is given by the representation category $H\Repfd$ of finite-dimensional modules over some finite-dimensional $\mathbbm{k}$-Hopf algebra $H$ with antipode $S: H \rightarrow H$ and skew-antipode $T: H \rightarrow H$. In that case the left and right dual of some $H$-module $V$ is given by the vector space $V^{\star}$ with respective actions 
\[ (h.v')(v) := v'(S(h).v) \quad \textnormal{and} \quad (h.v')(v) := v'(T(h).v)) \] 
for $v' \in V^{\star}, v \in V, h \in H$. In particular, the representation category $\mathbbm{k}\lbrack G \rbrack\Repfd$ of the group algebra $\mathbbm{k}\lbrack G \rbrack$ of some finite group $G$ is rigid. Other examples include the cobordism categories $\Cob(n)$, the category of profunctors $\Prof$, and the category $\Gvect$ of $G$-graded finite-dimensional vectorspaces, if $G$ is a group.

One can extend the notion of duality from objects to morphisms in the following way. For a morphism $f: x \rightarrow y$ in a rigid category $(\mathcal{C}, \otimes, \unit, \alpha, \lambda, \varrho)$ with left evaluation $\varepsilon_y:\ ^{\star}y \otimes y \rightarrow \unit$ and left coevaluation $\eta_x: \unit \rightarrow x \otimes\ ^{\star} x$, define a left dual morphism $^{\star} f:\ ^{\star}y \rightarrow\ ^{\star}x$ as
\[
^{\star} f := \lambda_{^{\star}x} \circ (\varepsilon_y \otimes \textnormal{id}_{^{\star} x}) \circ  \alpha^{-1}_{^{\star}y, y, ^{\star}x} \circ (\textnormal{id}_{^{\star}y} \otimes (f \otimes \textnormal{id}_{^{\star} x})) \circ (\textnormal{id}_{^{\star}y} \otimes \eta_x) \circ \varrho^{-1}_{^{\star}y},
\]
and analogously a right dual morphism $f^{\star}:\ y^{\star} \rightarrow\ x^{\star}$. This construction is functorial, contravariant and invertible up to isomorphism, as the following result (cf. \cite[Sec. 2.10.]{etingof2015tensor}) shows.

\begin{lemma}
\label{rigidmonoidalequivalence}
Let $(\mathcal{C}, \otimes, \unit)$ be a rigid category. The duality $(-)^{\star}$ can be equipped with the structure of a monoidal equivalence
\[
	(-)^{\star}:\quad (\mathcal{C}, \otimes, \unit) \rightarrow (\mathcal{C}, \otimes, \unit)^{\textnormal{opp}(0,1)}
\]
with quasi-inverse $^{\star}(-)$.
\end{lemma}

\section{Internal hom}

\begin{definition}[Module category]
	Let $(\mathcal{C}, \otimes, \unit)$ be a monoidal category and $\mathcal{M}$ an arbitrary category. An \textit{action} of $\mathcal{C}$ on $\mathcal{M}$ is a strong monoidal functor $F: (\mathcal{C}, \otimes, \unit) \rightarrow (\End(\mathcal{M}), \circ , \textnormal{id}_{\mathcal{M}})$. If such an action exists, we call $_{\mathcal{C}} \mathcal{M} := (\mathcal{M}, F)$ a (left) \textit{module category} over $\mathcal{C}$.
\end{definition}

Giving an action of $(\mathcal{C}, \otimes, \unit)$ on $\mathcal{M}$ is equivalent to giving a functor $\rhd: \mathcal{C} \times \mathcal{M} \rightarrow \mathcal{M}$ and natural isomorphisms 
$(\alpha_{x,y,m}: (x \otimes y) \rhd m \cong x \rhd (y \rhd m))_{x,y \in \mathcal{C}, m \in \mathcal{M}}$ and $(\lambda_m: \unit \rhd m \cong m)_{m \in \mathcal{M}}$, such that one pentagon and one triangle diagram commute, cf. \cite[Prop. 7.1.3]{etingof2015tensor}.

\begin{definition}[Internal hom]
\label{internalhomdef}
	Let $(\mathcal{C}, \otimes, \unit)$ be monoidal category and let $_{\mathcal{C}} \mathcal{M}$ be a left $\mathcal{C}$-module category. An \textit{internal hom} of $\mathcal{M}$ in $\mathcal{C}$ is a functor
	\begin{equation*}
		\underline{\textnormal{Hom}}(-,-): \mathcal{M}^{\textnormal{opp}} \times \mathcal{M} \rightarrow \mathcal{C}
	\end{equation*}
	such that for every $m \in \mathcal{M}$ the functor $\underline{\textnormal{Hom}}(m,-) : \mathcal{M} \rightarrow \mathcal{C}$ is right adjoint to the functor $- \rhd m: \mathcal{C} \rightarrow \mathcal{M}$, i.e. there are natural isomorphisms
	\begin{equation*}
		\textnormal{Hom}_{\mathcal{M}}(x \rhd m,n) \cong  \textnormal{Hom}_{\mathcal{C}}(x,\underline{\textnormal{Hom}}(m,n))
	\end{equation*}
	for all $x \in \mathcal{C}$ and $m,n \in \mathcal{M}$.
\end{definition}

\begin{definition}[Closed category]
	Let $(\mathcal{C}, \otimes, \unit)$ be a monoidal category. We call $(\mathcal{C}, \otimes, \unit)$ \textit{left closed}, if there exists an internal hom $\underline{\textnormal{Hom}}^l(-,-)$ for $\mathcal{C}$ as left $\mathcal{C}$-module category by tensoring from the left. We call $(\mathcal{C}, \otimes, \unit)$ \textit{right closed}, if there exists an internal hom $\underline{\textnormal{Hom}}^r(-,-)$ for $\mathcal{C}$ as left $\mathcal{C}^{\textnormal{opp}(0)}$-module category by tensoring from the right. A monoidal category that is both left and right closed is called \textit{biclosed}.
\end{definition}
In other words, a monoidal category $(\mathcal{C}, \otimes, \unit)$ is left closed with internal hom $\underline{\textnormal{Hom}}^l(-,-)$, if there exist natural isomorphisms
\begin{equation}
\label{ihomldefeq}
	\textnormal{Hom}_{\mathcal{C}}(x \otimes y, z) \cong \textnormal{Hom}_{\mathcal{C}}(x, \underline{\textnormal{Hom}}^l(y,z)),
\end{equation}
for all $x,y,z \in \mathcal{C}$, and it is right closed with internal hom $\underline{\textnormal{Hom}}^r(-,-)$, if there exist natural isomorphisms
\begin{equation}
\label{ihomrdefeq}
	\textnormal{Hom}_{\mathcal{C}}(x \otimes y, z) \cong \textnormal{Hom}_{\mathcal{C}}(y, \underline{\textnormal{Hom}}^r(x,z))
\end{equation}
for all $x,y,z \in \mathcal{C}$.

The Yoneda Lemma implies that left, respectively right, internal homs are unique up to unique natural isomorphism. In particular, if a biclosed category $(\mathcal{C}, \otimes, \unit, \underline{\textnormal{Hom}}^l, \underline{\textnormal{Hom}}^r)$ is symmetric, the Yoneda Lemma implies that left and right internal hom are isomorphic, $\underline{\textnormal{Hom}}^l \cong \underline{\textnormal{Hom}}^r$. For this reason a symmetric biclosed category is sometimes simply called \textit{closed}.

The following result (cf. \cite[Prop. 2.10.8]{etingof2015tensor} shows that every rigid category is biclosed and vice versa.

\begin{lemma}
\label{rigidimpliesclosed}
	Let $(\mathcal{C}, \otimes, \unit)$ be a rigid category with duality $(-)^{\star}$ and quasi-inverse $^{\star}(-)$. Then $(\mathcal{C}, \otimes, \unit)$ is biclosed by defining
	\begin{equation}
	\label{internalhomrigid1}
		\underline{\textnormal{Hom}}^l(x,y) := y \otimes\  ^{\star}x \quad \textnormal{and} \quad \underline{\textnormal{Hom}}^r(x,y) :=  x^{\star} \otimes y
	\end{equation}
	on $x,y \in \mathcal{C}$. Conversely, every biclosed category $(\mathcal{C}, \otimes, \unit)$ with internal homs $\underline{\textnormal{Hom}}^l(-,-)$ and $\underline{\textnormal{Hom}}^r(-,-)$ is rigid by defining
	\begin{equation}
		\label{internalhomrigid2}
			^{\star}x := \underline{\textnormal{Hom}}^l(x,\mathbbm{1}) \quad \textnormal{and} \quad  x^{\star} := \underline{\textnormal{Hom}}^r(x,\mathbbm{1}).
	\end{equation}
	In fact, by uniqueness, in every rigid biclosed category duality and internal hom are related by \eqref{internalhomrigid1}, \eqref{internalhomrigid2} up to natural isomorphism.
\end{lemma}

In the rigid biclosed symmetric monoidal category $(\vect, \otimes, \mathbbm{k})$ there exist isomorphisms $ V \cong V^{\star \star} = \underline{\textnormal{Hom}}(\underline{\textnormal{Hom}}(V, \mathbbm{k}), \mathbbm{k})$ and $\textnormal{Hom}(V \otimes W, \mathbbm{k}) \cong \textnormal{Hom}(V, W^{\star}) \cong \textnormal{Hom}(W, V^{\star})$. \autoref{rigidimpliesclosed} yields, in particular, non-symmetric generalisations of those isomorphisms for arbitrary rigid categories. In fact, in any rigid biclosed category $(\mathcal{C}, \otimes, \unit, (-)^{\star},\ ^{\star}(-))$ with internal homs $ \underline{\textnormal{Hom}}^l(-,-)$ and $\underline{\textnormal{Hom}}^r(-,-)$ we find isomorphisms
\begin{equation*}
	\underline{\textnormal{Hom}}^r(\underline{\textnormal{Hom}}^l(x,\mathbbm{1}),\mathbbm{1}) = (^{\star}x)^{\star} \cong x \cong\ ^{\star}(x^{\star}) = \underline{\textnormal{Hom}}^l(\underline{\textnormal{Hom}}^r(x,\mathbbm{1}),\mathbbm{1}),
\end{equation*}
\begin{equation*}
	\textnormal{Hom}_{\mathcal{C}}(x,\ ^{\star}y) \cong	\textnormal{Hom}_{\mathcal{C}}(x \otimes y, \mathbbm{1}) \cong \textnormal{Hom}_{\mathcal{C}}(y , x^{\star}).
\end{equation*}

\section{Ribbon monoidal categories}

Similarly to how a monoidal category categorifies the notion of a monoid, the following definition categorifies the notion of a \textit{commutative} monoid. Following the philosophy of category theory, the \textit{property} $x \otimes y = y \otimes x$ is replaced by a \textit{structure}, namely an isomorphism $x \otimes y \cong y \otimes x$.

\begin{definition}[Braided category]
\label{braidingdef}
	A \textit{braiding} on a monoidal category $(\mathcal{C}, \otimes, \mathbbm{1})$ is a natural isomorphism 
	\[ \gamma_{x,y}: x \otimes y \rightarrow y \otimes x \]
	 such that two 	hexagonal diagrams commute (cf. \cite[Sec. 8.1]{etingof2015tensor}). A \textit{braided monoidal category} is a pair consisting of a monoidal category and a braiding.
\end{definition}

For any braiding $\gamma$ the \textit{reverse braiding} $\gamma'$ is defined by $\gamma'_{x,y} := \gamma_{y,x}^{-1}$. One can check that $\gamma'$ is indeed a braiding, i.e. satisfies the two hexagonal diagrams. Braiding and reverse braiding are commonly pictured as
\begin{equation*}
	\gamma_{x,y} \equiv \includegraphics[height=0.5cm]{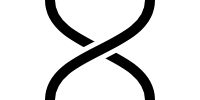}, \qquad
	\gamma'_{x,y} \equiv \includegraphics[height=0.5cm]{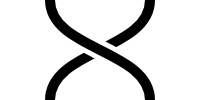}.
\end{equation*}
The equality $\gamma'_{y,x} \circ \gamma_{x,y} = \textnormal{id}_{x\otimes y}$ then corresponds to the topological reasonable statement
\begin{equation}
\label{braidingreverse}
\begin{tabular}{m{2em} m{0em} m{1.5em}}
	\includegraphics[height=1cm]{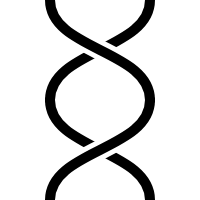} 
	&= 
	&\includegraphics[height=1cm, width=1cm]{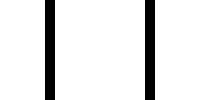}.
\end{tabular}
\end{equation}

\begin{definition}[Symmetric category]
	A braided monoidal category $(\mathcal{C}, \otimes, \mathbbm{1}, \gamma)$ is called \textit{symmetric}, if 
	\begin{equation}
	\label{symmetriceq}
		\gamma_{y,x} \circ \gamma_{x,y} = \textnormal{id}_{x \otimes y}
	\end{equation}
	for all $x,y \in \mathcal{C}$.
\end{definition}

Examples of symmetric braided monoidal categories include $\Set$, $\Vect$, $A\Rep$ for $A$ some arbitrary-dimensional $\mathbbm{k}$-bialgebra $A$, and the category $\GVect$ of $G$-graded arbitrary-dimensional vector spaces   for an abelian group $G$. In all cases the braiding is given by the transposition of factors.

Note that \eqref{symmetriceq} corresponds to the non-trivial (in contrast to \eqref{braidingreverse}) topological condition
\begin{equation*}
\begin{tabular}{m{2em} m{0em} m{1.5em}}
	\includegraphics[height=1cm]{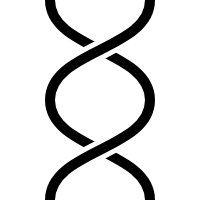} 
	&= 
	&\includegraphics[height=1cm, width=1cm]{pics/braiding/id}.
\end{tabular}
\end{equation*}

\begin{definition}[Ribbon category]
\label{ribbondef}
	A \textit{twist} on a braided rigid monoidal category $(\mathcal{C}, \otimes, \mathbbm{1}, \gamma, (-)^{\star})$ is a natural automorphism $\theta: \textnormal{id}_{\mathcal{C}} \Rightarrow \textnormal{id}_{\mathcal{C}}$ such that
	\begin{equation*}
		\theta_{x \otimes y} = (\theta_x \otimes \theta_y) \circ \gamma_{y,x} \circ \gamma_{x,y}
	\end{equation*}
	for all $x,y \in \mathcal{C}$. A twist is called \textit{ribbon structure}, if 
	\begin{equation*}
		(\theta_x)^{\star} = \theta_{x^{\star}}
	\end{equation*}
	for all $x \in \mathcal{C}$. A \textit{ribbon category} is a braided rigid monoidal category equipped with a ribbon structure.
\end{definition}

\chapter{Generalised duality theory}

In this Chapter we will introduce the main subject of this thesis, $\star$-\textit{autonomous categories}.

\section{$\star$-autonomous categories}
Let $(\mathcal{C}, \otimes, \unit)$ be a biclosed monoidal category with internal homs $\ihoml$ and $\ihomr$. Abbreviating $_{\mathcal{C}}\langle -,- \rangle := \textnormal{Hom}_{\mathcal{C}}(-,-)$, for any object $x,k \in \mathcal{C}$ there exist natural isomorphisms
	\[
	_{\mathcal{C}}\langle \ihoml(x,k), \ihoml(x,k) \rangle \cong\ _{\mathcal{C}} \langle \ihoml(x,k) \otimes x, k \rangle \cong\ _{\mathcal{C}} \langle x, \ihomr(\ihoml(x,k),k) \rangle,
	\]
		\[
	_{\mathcal{C}}\langle \ihomr(x,k), \ihomr(x,k) \rangle \cong\ _{\mathcal{C}} \langle x \otimes \ihomr(x,k) , k \rangle \cong\ _{\mathcal{C}} \langle x, \ihoml(\ihomr(x,k),k) \rangle.
	\]
This implies, in particular, that we can identify $\textnormal{id}_{\ihoml(x,k)}$ and $\textnormal{id}_{\ihomr(x,k)}$ with unique morphisms
\begin{equation}
\label{induceddualizingobject}
	x \rightarrow \ihomr(\ihoml(x,k),k) \quad \textnormal{and} \quad x \rightarrow \ihoml(\ihomr(x,k),k).
\end{equation} 

\begin{definition}[Dualizing object]
	Let $(\mathcal{C}, \otimes, \unit)$ be a biclosed monoidal category with internal homs $\ihoml$ and $\ihomr$. We call an object $k \in \mathcal{C}$ \textit{dualizing object}, if the induced morphisms \eqref{induceddualizingobject} are isomorphisms for all $x \in \mathcal{C}$.
\end{definition}

In the following Proposition, and its proof, we summarise results found in  \cite{barrautonomcat}, \cite{barr1995nonsymmetric} and \cite{boyarchenko2011duality}. For clarity we will use the letter $D$ instead of the symbol $^{\star}$ for the duality functor.

\begin{proposition}
\label{starautonomequivalent}
Let $(\mathcal{C}, \otimes, \unit)$ be a monoidal category. The following are equivalent:
\begin{enumerate}
	\item There exist left and right internal homs on $\mathcal{C}$ and a dualizing object $k \in \mathcal{C}$.
	\item There exists an equivalence $D: \mathcal{C} \rightarrow \mathcal{C}^{\textnormal{opp}(1)}$ and natural isomorphisms
	\begin{equation}
	\label{starautonomfrobeniuseq}
		\textnormal{Hom}_{\mathcal{C}}(x \otimes y, D^{\textnormal{opp}(1)}(\overline{z})) \cong \textnormal{Hom}_{\mathcal{C}}(x, D^{\textnormal{opp}(1)}(\overline{y \otimes z})).
	\end{equation}
	\item There exists an equivalence $D: \mathcal{C} \rightarrow \mathcal{C}^{\textnormal{opp}(1)}$, an object $k \in \mathcal{C}$, and natural isomorphisms
	\begin{equation}
		\label{grothendieck}
		\textnormal{Hom}_{\mathcal{C}}(x \otimes y, k) \cong \textnormal{Hom}_{\mathcal{C}}(x, D^{\textnormal{opp}(1)}(\overline{y})).
	\end{equation}
	\item There exists an equivalence $D: \mathcal{C} \rightarrow \mathcal{C}^{\textnormal{opp}(1)}$, an object $k \in \mathcal{C}$, and natural isomorphisms
	\begin{equation}
		\label{grothendieck2}
	\textnormal{Hom}_{\mathcal{C}}(x \otimes y, k) \cong \textnormal{Hom}_{\mathcal{C}}(y, D^{-1}(\overline{x})).
	\end{equation}
\end{enumerate}
\end{proposition}
\begin{proof}
	$(1) \Rightarrow (2)$: Define $D := \ihoml(-,k)^{\textnormal{opp}(1)}: \mathcal{C} \rightarrow \mathcal{C}^{\textnormal{opp}(1)}$ and $D^{-1} := \ihomr(-,k): \mathcal{C}^{\textnormal{opp}(1)} \rightarrow \mathcal{C}$. The isomorphisms \eqref{induceddualizingobject} witness that $D$ is indeed an equivalence with quasi-inverse $D^{-1}$. Moreover, using the associator, we find
	\begin{align*}
		&\textnormal{Hom}_{\mathcal{C}}(x \otimes y, D^{\textnormal{opp}(1)}(\overline{z})) 
		= \textnormal{Hom}_{\mathcal{C}}(x \otimes y,  \ihoml(z,k)) 
		\cong \textnormal{Hom}_{\mathcal{C}}(x \otimes y \otimes z, k) \\
		&\cong \textnormal{Hom}_{\mathcal{C}}(x,  \ihoml(y \otimes z,k)) 
		= \textnormal{Hom}_{\mathcal{C}}(x, D^{\textnormal{opp}(1)}(\overline{y \otimes z})).
	\end{align*}
	
\noindent	$(2) \Rightarrow (1)$: 
	The calculations
\begin{align*}
_{\mathcal{C}} \langle x \otimes y, z \rangle 
\cong\ _{\mathcal{C}} \langle x \otimes y, D^{\textnormal{opp}(1)}(\overline{D^{-1}(\overline{z})}) \rangle	
\overset{\eqref{starautonomfrobeniuseq}}{\cong}\ _{\mathcal{C}} \langle x, D^{\textnormal{opp}(1)}(\overline{y \otimes D^{-1}(\overline{z}) }) \rangle,
\end{align*}
\begin{align*}
	&_{\mathcal{C}} \langle x \otimes y, z \rangle 
	\overset{D^{\textnormal{opp}(1)}}{\cong}\ _{\mathcal{C}} \langle D^{\textnormal{opp}(1)}(\overline{z}),  D^{\textnormal{opp}(1)}(\overline{x \otimes y}) \rangle
	\overset{\eqref{starautonomfrobeniuseq}}{\cong}\  _{\mathcal{C}} \langle \unit, D^{\textnormal{opp}(1)}(\overline{D^{\textnormal{opp}(1)}(\overline{z}) \otimes x \otimes y}) \rangle \\
	&\overset{\eqref{starautonomfrobeniuseq}}{\cong}\ _{\mathcal{C}} \langle D^{\textnormal{opp}(1)}(\overline{z}) \otimes x, D^{\textnormal{opp}(1)}(\overline{y}) \rangle
	\overset{D^{-1}}{\cong}\ _{\mathcal{C}} \langle y, D^{-1}(\overline{D^{\textnormal{opp}(1)}(\overline{z} \otimes x}) \rangle
\end{align*}
	show that $\ihoml(\overline{x}, y) := D^{\textnormal{opp}(1)}(\overline{x \otimes D^{-1}(\overline{y})})$ and $\ihomr(\overline{x}, y) :=$ $D^{-1}(\overline{D^{\textnormal{opp}(1)}(\overline{y}) \otimes x})$ define left and right inner homs. Moreover, the Yoneda Lemma and 
	\begin{align}
	\label{ddinversek}
	_{\mathcal{C}} \langle x, 	D^{\textnormal{opp}(1)}(\overline{\unit}) \rangle 
	\overset{\eqref{starautonomfrobeniuseq}}{\cong}\ _{\mathcal{C}} \langle \unit, 	D^{\textnormal{opp}(1)}(\overline{x \otimes \unit}) \rangle 
	\overset{D^{-1}}{\cong}\ _{\mathcal{C}} \langle x, D^{-1}(\overline{\unit}) \rangle
	\end{align}
		show that $D^{\textnormal{opp}(1)}(\overline{\unit}) \cong D^{-1}(\overline{\unit}) =: k$. Since $D$ and $D^{-1}$ are quasi-inverse, and  
		\[ \ihoml(\overline{x}, k) =  D^{\textnormal{opp}(1)}(\overline{x \otimes D^{-1}(\overline{k})}) \cong D^{\textnormal{opp}(1)}(\overline{x \otimes D^{-1}(\overline{D^{\textnormal{opp}(1)}(\overline{\unit})})}) \cong D^{\textnormal{opp}(1)}(\overline{x}),
		\]
		\[
		\ihomr(\overline{x}, k) = D^{-1}(\overline{D^{\textnormal{opp}(1)}(\overline{D^{-1}(\overline{\unit})}) \otimes x}) \cong D^{-1}(\overline{x}),
		\]		
		it follows that $k$ is dualizing.	

\noindent $(2) \Rightarrow (3)$: Define $k :=  D^{\textnormal{opp}(1)}(\overline{\unit}) \overset{\eqref{ddinversek}}{\cong} D^{-1}(\overline{\unit})$. Then
\begin{align*}
_{\mathcal{C}} \langle x \otimes y, k \rangle
\cong\	_{\mathcal{C}} \langle x \otimes y, D^{\textnormal{opp}(1)}(\overline{\unit}) \rangle
\overset{\eqref{starautonomfrobeniuseq}}{\cong}\ _{\mathcal{C}} \langle x, D^{\textnormal{opp}(1)}(\overline{y \otimes \unit}) \rangle
\cong\ _{\mathcal{C}} \langle x, D^{\textnormal{opp}(1)}(\overline{y}) \rangle.
\end{align*}

\noindent $(3) \Rightarrow (2)$: Using the associator, one shows that
\begin{align*}
	_{\mathcal{C}} \langle x \otimes y, D^{\textnormal{opp}(1)}(\overline{z}) \rangle
	\overset{\eqref{grothendieck}}{\cong}\ _{\mathcal{C}} \langle (x \otimes y) \otimes z, k \rangle
	\cong\ \langle x \otimes (y \otimes z), k \rangle
	\overset{\eqref{grothendieck}}{\cong}\ _{\mathcal{C}} \langle x, D^{\textnormal{opp}(1)}(\overline{y \otimes z}) \rangle.
\end{align*}
		
\noindent $(3) \Leftrightarrow (4)$: This follows immediately from applying $D$ and $D^{-1}$, respectively.			
	\end{proof}

\begin{definition}[$\star$-autonomous category]
	A monoidal category with one of the structures of \autoref{starautonomequivalent} is called $\star$\textit{-autonomous}. In particular, a monoidal category with a structure as in $(3)$ or $(4)$, is also called \textit{Grothendieck-Verdier} category (cf. \cite{boyarchenko2011duality}).
	\end{definition}
	
Every rigid monoidal category is biclosed by \autoref{rigidimpliesclosed} and $\star$-autonomous with dualizing object given by the monoidal unit. For this reason, $\star$-autonomous categories with dualizing object given by the monoidal unit are called \textit{r-categories} in \cite{boyarchenko2011duality}. It follows that every rigid category is an r-category. However, the converse is false (cf. \cite[Example 0.9]{boyarchenko2011duality}).

In \cite[Prop. 1.3]{boyarchenko2011duality} it is shown that a dualizing object is unique only up to tensoring with invertible objects, i.e. if $k \in \mathcal{C}$ is a dualizing object and $x \in \mathcal{C}$ is an invertible object, then $x \otimes k$ and $k \otimes x$ are also dualizing objects.

As far as we understand, at the present the literature does not include any commonly accepted notion of morphisms of $\star$-autonomous categories. In fact, it would be appropriate to introduce several notions of morphisms, namely for all structures introduced in \autoref{starautonomequivalent}, and then extend \autoref{starautonomequivalent} to an equivalence of categories.

For example, a candidate for a morphism between two $\star$-autonomous categories $(\mathcal{C}, \otimes_{\mathcal{C}}, \unit_{\mathcal{C}}, k_{\mathcal{C}}, D_{\mathcal{C}}, D^{-1}_{\mathcal{C}})$ and  $(\mathcal{D}, \otimes_{\mathcal{D}}, \unit_{\mathcal{D}}, k_{\mathcal{D}}, D_{\mathcal{D}}, D^{-1}_{\mathcal{D}})$ in the sense of $3.$ of \\\autoref{starautonomequivalent}, could be a lax monoidal functor 
\begin{equation}
		\label{monoidalfuncstar}
		(F, \Phi, \phi): (\mathcal{C}, \otimes_{\mathcal{C}}, \unit_{\mathcal{C}}) \rightarrow (\mathcal{D}, \otimes_{\mathcal{D}}, \unit_{\mathcal{D}})
	\end{equation}
	together with a morphism 
	\begin{equation}
	\label{kf}
		k_F: F(k_{\mathcal{C}}) \rightarrow k_{\mathcal{D}}.
	\end{equation}
	and a lax natural transformation
	\begin{equation}
	\label{laxnattransf}
		F(D^{\textnormal{opp}(1)}_{\mathcal{C}}(\overline{x})) \rightarrow D^{\textnormal{opp}(1)}_{\mathcal{D}}(\overline{F(x)}),
	\end{equation}
	such that the following coherence diagram commutes
	\begin{equation}
	\label{laxnatcomm}
	\begin{tikzcd}[row sep=large, column sep = large]
	_{\mathcal{C}} \langle x \otimes y, k_{\mathcal{C}} \rangle \arrow{d}[left]{F} \arrow{r}{\eqref{grothendieck}} & _{\mathcal{C}} \langle x, D^{\textnormal{opp}(1)}_{\mathcal{C}}(\overline{y}) \rangle \arrow{d}{F}\\
	_{\mathcal{D}} \langle F(x \otimes y), F(k_{\mathcal{C}}) \rangle \arrow{d}[left]{\Phi, k_F} & _{\mathcal{D}} \langle F(x), F(D^{\textnormal{opp}(1)}_{\mathcal{C}}(\overline{y})) \rangle \arrow{d}{\eqref{laxnattransf}} \\
	_{\mathcal{D}} \langle F(x) \otimes F(y), k_{\mathcal{D}} \rangle \arrow{r}{\eqref{grothendieck}}&
	_{\mathcal{D}} \langle F(x), D^{\textnormal{opp}(1)}_{\mathcal{D}}(\overline{F(y)}) \rangle.
	\end{tikzcd}
\end{equation}
In fact, as a special case of \cite[Remark, p.27]{mellies2016dialogue} shows, a lax natural transformation  \eqref{laxnattransf} can be deduced from the data \eqref{monoidalfuncstar}, \eqref{kf} in the following way. First, for any element $x \in \mathcal{C}$ we can identify the identity morphism $\textnormal{id}_{D^{\textnormal{opp}(1)}(\overline{x})}$ via the natural isomorphism $\eqref{grothendieck}$ of $\mathcal{C}$ with an unique morphism 
\begin{equation}
\label{evaluationmorphism}
D^{\textnormal{opp}(1)}_{\mathcal{C}}(\overline{x}) \otimes_{\mathcal{C}} x \rightarrow k_{\mathcal{C}}.
\end{equation}
Thus we can construct a morphism
\[ F(D^{\textnormal{opp}(1)}_{\mathcal{C}}(\overline{x})) \otimes_{\mathcal{D}} F(x) \overset{\Phi}{\rightarrow} F(D^{\textnormal{opp}(1)}_{\mathcal{C}}(\overline{x}) \otimes_{\mathcal{C}} x) \overset{\eqref{evaluationmorphism}}{\rightarrow} F(k_{\mathcal{C}}) \overset{\eqref{kf}}{\rightarrow} k_{\mathcal{D}},
\]
and identify it via the natural isomorphism $\eqref{grothendieck}$ of $\mathcal{D}$ with an unique morphism 
\[ F(D^{\textnormal{opp}(1)}_{\mathcal{C}}(\overline{x})) \rightarrow D^{\textnormal{opp}(1)}_{\mathcal{D}}(\overline{F(x)}).
\]
One checks that this construction provides a natural transformation which satisfies \eqref{laxnatcomm}.

In view of the previous remarks we give the following definition.

\begin{definition}
	A (\textit{lax}) $\star$\textit{-autonomous functor} between $\star$-autonomous categories \\$(\mathcal{C}, \otimes_{\mathcal{C}}, \unit_{\mathcal{C}}, k_{\mathcal{C}})$ and $(\mathcal{D}, \otimes_{\mathcal{D}}, \unit_{\mathcal{D}}, k_{\mathcal{D}})$ is given by a lax monoidal functor
	\begin{equation}
		F: (\mathcal{C}, \otimes_{\mathcal{C}}, \unit_{\mathcal{C}}) \rightarrow (\mathcal{D}, \otimes_{\mathcal{D}}, \unit_{\mathcal{D}})
	\end{equation}
	together with a morphism 
	\begin{equation}
		k_F: F(k_{\mathcal{C}}) \rightarrow k_{\mathcal{D}}.
	\end{equation}
\end{definition}

We now turn to a different aspect. As a generalisation of \autoref{braidingdef} and \autoref{ribbondef} we adopt the following definitions from \cite{boyarchenko2011duality}.

\begin{definition}
\label{weakribbondef}
	A $\star$-autonomous category is \textit{braided}, if its underlying monoidal structure is equipped with a braiding. Similarly, a \textit{twist} on a braided $\star$-autonomous category is a twist on the underlying braided monoidal structure. A \textit{(weak) ribbon structure} of a braided $\star$-autonomous category $\mathcal{C}$ with duality $D$ is a twist $\theta$, such that 
	\[
	D(\theta_x) = \theta_{D(x)}
	\]
	for all $x \in \mathcal{C}$.
	\end{definition}

We will encounter weak ribbon structures of $\star$-autonomous categories again in \autoref{weakribbonmonoidal}.

\section{Linearly distributive categories}

One of the important differences between $\star$-autonomous categories and rigid categories lies in the monoidal structure of the duality functor $(-)^{\star}$. In \autoref{rigidmonoidalequivalence} we have seen that for every rigid category $(\mathcal{C}, \otimes, \unit)$ the induced duality $(-)^{\star}: \mathcal{C} \rightarrow \mathcal{C}^{\textnormal{opp}(1)}$ can be equipped with the structure of a	monoidal equivalence
\begin{equation}
\label{rigiddualityeq}
	(-)^{\star}: (\mathcal{C}, \otimes, \unit) \rightarrow (\mathcal{C}, \otimes, \unit)^{\textnormal{opp}(0,1)}.
\end{equation}
 This is not necessarily the case for every $\star$-autonomous category $(\mathcal{C}, \otimes, \unit, (-)^{\star})$ (however, by a non-trivial proof the double dual is again necessary monoidal in the above sense \cite[Prop. 4.2.]{boyarchenko2011duality}) and motivates the definition of a second tensor product $\otimes'$ on $\mathcal{C}$ by
 \begin{equation}
 \label{secondtensorproductintro}
 	x \otimes' y :=\ ^{\star}(y^{\star} \otimes x^{\star})
 \end{equation}
 for $x,y \in \mathcal{C}$, and $^{\star}(-): \mathcal{C}^{\textnormal{opp}(1)} \rightarrow \mathcal{C}$ the quasi-inverse of the duality $(-)^{\star}: \mathcal{C} \rightarrow \mathcal{C}^{\textnormal{opp}(1)}$.

For example, let $X$ be a set and let $(\mathcal{P}(X), \cap, X, (-)^c)$ be the $\star$-autonomous category we defined in the introduction -- the category with tensor product given by the intersection of sets $\cap$, tensor unit $X$, and self-inverse duality $(-)^c$ given by the set theoretic complement. The second tensor product on $\mathcal{P}(X)$ induced by $(-)^c$ is given on $A,B \subseteq X$ as $(A^c \cap B^c)^c = A \cup B$, using De Morgan's laws. In other words, in contrast to $\eqref{rigiddualityeq}$, the duality $(-)^c$ induces a monoidal equivalence 
\[ (-)^c: (\mathcal{P}(X), \cap, X) \rightarrow (\mathcal{P}(X), \cup, \emptyset)^{\textnormal{opp}(0,1)}. \]

Surprisingly, to characterise $\star$-autonomous categories it is sufficient to focus on the relations between the tensor products $\otimes, \otimes'$ and the duality $(-)^{\star}$. To this end, we introduce the following definitions. However, note that historically the notion of \textit{linearly distributive categories}, originally called \textit{weakly distributive categories},  was defined by Cockett and Seely in \cite{cockettseely1997}, and only later related to $\star$-autonomous categories.

\begin{definition}[Linearly distributive category]
\label{linearlydistributivecategory}
	A \textit{linearly distributive category} is a category $\mathcal{C}$ with two monoidal structures $(\otimes_1, \unit_1, \alpha_1, \lambda_1, \varrho_1)$, $(\otimes_2,\unit_2, \alpha_2, \lambda_2, \varrho_2)$ and, not necessarily invertible, natural transformations
\begin{align*}
	(\delta^L)_{x,y,z}&: x \otimes_1 ( y \otimes_2 z ) \rightarrow (x \otimes_1 y) \otimes_2 z, \\
	(\delta^R)_{x,y,z}&: (x \otimes_2 y) \otimes_1 z \rightarrow x \otimes_2 (y \otimes_1 z),
	\end{align*}
called \textit{distributors}, subject to the six pentagon and four triangle constraints, describing how the distributors interact with the associators (cf. e.g. \cite{mellies2009categorical}, \cite{cockettseely1997}). A linearly distributive category is \textit{symmetric}, if both monoidal structures are symmetric.
\end{definition}

If the distributors are invertible it is possible to define a second linearly distributive structure by reversing the order of the monoidal structures and defining $\delta_L' := (\delta^R)^{-1}$, respectively $\delta_R' := (\delta^L)^{-1}$.

Every monoidal category $(\mathcal{C}, \otimes, \mathbbm{1}, \alpha, \lambda, \varrho)$ induces a linearly distributive category by choosing both monoidal structures equally as $(\otimes, \mathbbm{1}, \alpha, \lambda, \varrho)$ and the distributors as $\delta^L =  \alpha^{-1}$, respectively $\delta^R = \alpha$. Note that in that case the above construction does not yield a new linearly distributive category.

\begin{definition}[Linearly distributive category with duality]
	Let $\mathcal{C}$ be a linearly distributive category $\mathcal{C}$ with monoidal structures $(\otimes_1, \unit_1, \alpha_1, \lambda_1, \varrho_1),\ (\otimes_2, \unit_2, \alpha_2, \lambda_2, \varrho_2)$. A \textit{linear left dual} to $x \in \mathcal{C}$ is an object $^{\star}x \in \mathcal{C}$ with morphisms $\varepsilon_x:\ ^{\star}x \otimes_1 x \rightarrow \unit_2$ and $\eta_x: \unit_1 \rightarrow x \otimes_2\ ^{\star}x$, called (left) \textit{evaluation} and (left) \textit{coevaluation} respectively, satisfying the so called zig-zag (or snake) identities
\begin{equation}
	\varrho_2	\circ (\textnormal{id}_x \otimes_2 \varepsilon_x) \circ \delta^R \circ (\eta_x \otimes_1 \textnormal{id}_x) \circ \lambda^{-1}_1 \quad = \quad \textnormal{id}_x,
	\end{equation}
	\begin{equation}
			\lambda_2 \circ (\varepsilon_x \otimes_2 \textnormal{id}_{^{\star}x}) \circ \delta^L \circ (\textnormal{id}_{^{\star}x} \otimes_1 \eta_{x}) \circ \varrho^{-1}_1 \quad = \quad \textnormal{id}_{^{\star}x}.
	\end{equation}
	A \textit{right linear dual} to $x \in \mathcal{C}$ is an object $x^{\star} \in \mathcal{C}$ with a (right) evaluation $\widetilde{\varepsilon}_x: x \otimes_1 x^{\star} \rightarrow \unit_2$ and a (right) coevaluation $\widetilde{\eta}_x: \unit_1 \rightarrow\  x^{\star} \otimes_2 x$, satisfying analogous zig-zag identities. If for every object $x \in \mathcal{C}$ there exist left and right linear duals $^{\star}x,\ x^{\star}$, we say $\mathcal{C}$ comes \textit{with duality}.
\end{definition}

Note that  due to examples in logic, Cockett and Seely originally used the term \textit{negation} instead of \textit{duality}. Moreover, sometimes one speaks of a \textit{linear adjunction} and writes $x \dashv y$, if $y$ is a right dual of $x$, and $x$ a left dual of $y$, respectively.

One can extend the notion of duality in a linearly distributive category from objects to morphisms in the following way. For a morphism $f: x \rightarrow y$, left evaluation $\varepsilon_y:\ ^{\star}y \otimes_1 y \rightarrow \unit_2$ and left coevaluation $\eta_x: \unit_1 \rightarrow x \otimes_2\ ^{\star} x$, define a left dual morphism $^{\star} f:\ ^{\star}y \rightarrow\ ^{\star}x$ as
\begin{align*}
	^{\star} f &:= \lambda_2 \circ (\varepsilon_y \otimes_2 \textnormal{id}_{^{\star} x}) \circ \delta^L \circ (\textnormal{id}_{^{\star}y} \otimes_1 (f \otimes_2 \textnormal{id}_{^{\star} x})) \circ (\textnormal{id}_{^{\star}y} \otimes_1 \eta_x) \circ \varrho^{-1}_1 \\
	&= \lambda_2 \circ (\varepsilon_y \otimes_2 \textnormal{id}_{^{\star} x}) \circ ((\textnormal{id}_{^{\star}y} \otimes_1 f) \otimes_2 \textnormal{id}_{^{\star} x}) \circ \delta^L \circ (\textnormal{id}_{^{\star}y} \otimes_1 \eta_x) \circ \varrho^{-1}_1
	,
\end{align*}
using the naturality of $\delta^L$. Analogously one can define a right dual morphism $f^{\star}:\ y^{\star} \rightarrow\ x^{\star}$. This construction is functorial, contravariant and invertible up to natural isomorphism, as the following result (cf. \cite[Prop. 9]{mellies2017micrological}) shows.

\begin{lemma}
\label{linearlydistributivemonoidalfunctor}
Let $\mathcal{C}$ be a linearly distributive category with duality $(-)^{\star}$ and monoidal structures $(\otimes_1, \unit_1, \alpha_1, \lambda_1, \varrho_1)$,$ (\otimes_2, \unit_2, \alpha_2, \lambda_2, \varrho_2)$. The duality $(-)^{\star}$ can be equipped with the structure of a monoidal equivalence
\[
	(-)^{\star}:\quad (\mathcal{C}, \otimes_1, \unit_1) \rightarrow (\mathcal{C}, \otimes_2, \unit_2)^{\textnormal{opp}(0,1)}
\]
with quasi-inverse $^{\star}(-)$.
\end{lemma}

In the same way as every monoidal category induces a linearly distributive category, a \textit{rigid} monoidal category induces a linearly distributive category \textit{with duality}, and \autoref{linearlydistributivemonoidalfunctor} can be seen as the appropriate generalisation of \autoref{rigidmonoidalequivalence}. The converse is not true, as, for example, the category $(\mathcal{P}(X), \cap, X, (-)^c)$ witnesses, once again.

In \autoref{rigidimpliesclosed} we've seen that every rigid category can be equipped with the structure of a biclosed category. The following result (cf. \cite[Prop 4.10.4, Prop. 4.10.5]{mellies2009categorical}) gives a similar construction for linearly distributive categories with duality.
\begin{lemma}
\label{starautonomimpliesclosed}
	Let $\mathcal{C}$ be a linearly distributive category with duality $(-)^{\star}$, quasi-inverse $^{\star}(-)$, and monoidal structures $(\otimes_1, \unit_1, \alpha_1, \lambda_1, \varrho_1)$,$ (\otimes_2, \unit_2, \alpha_2, \lambda_2, \varrho_2)$. Then $(\mathcal{C}, \otimes_1, \unit_1)$ is biclosed by defining
	\begin{equation*}
		\underline{\textnormal{Hom}}^l(x,y) := y \otimes_2\ ^{\star}x \quad \textnormal{and} \quad \underline{\textnormal{Hom}}^r(x,y) :=\  x^{\star} \otimes_2 y
	\end{equation*}
\end{lemma}

\begin{corollary}
\label{linearlydistributiveinducedmorphisms}
	Let $\mathcal{C}$ be a linearly distributive category with monoidal structures $(\otimes_1, \unit_1, \alpha_1, \lambda_1, \varrho_1)$,$ (\otimes_2, \unit_2, \alpha_2, \lambda_2, \varrho_2)$. Let $x,y \in \mathcal{C}$ be objects with duals. Then there exist natural isomorphisms
		\[
	_{\mathcal{C}} \langle x \otimes_1 y, \unit_2 \rangle \cong\ _{\mathcal{C}}\langle y, x^{\star} \rangle \cong\ _{\mathcal{C}}\langle x,\ ^{\star} y \rangle,
	\] 
	\[
	_{\mathcal{C}} \langle \unit_1, y \otimes_2 x \rangle \cong\ _{\mathcal{C}} \langle x^{\star}, y \rangle \cong\  _{\mathcal{C}} \langle ^{\star}y, x \rangle.
	\]
	Moreover, any pair of forms $x \otimes_1 y \rightarrow \unit_2$ and $\unit_1 \rightarrow y \otimes_2 x$ provides a linear adjunction $x \dashv y$ if and only if the associated morphisms $y \rightarrow x^{\star}$ and $x^{\star} \rightarrow y$ are mutually inverse.
\end{corollary}

Apart of that, \autoref{starautonomimpliesclosed} implies, in particular, that $\underline{\textnormal{Hom}}^l(x,\unit_2) \cong\ ^{\star}x$ and  $\underline{\textnormal{Hom}}^r(x,\unit_2) \cong x^{\star}$ for all $x \in \mathcal{C}$. Thus, since \autoref{linearlydistributivemonoidalfunctor} shows that $(-)^{\star}$ and $^{\star}(-)$ are quasi-inverse to each other, we find
\[
\underline{\textnormal{Hom}}^l(\underline{\textnormal{Hom}}^r(x,\unit_2),\unit_2) \cong x \cong \underline{\textnormal{Hom}}^r(\underline{\textnormal{Hom}}^l(x,\unit_2),\unit_2).
\] 
One can show that both isomorphism equal those morphisms obtained via \eqref{induceddualizingobject}.
Hence, it follows that $\unit_2$ is a dualizing object for the biclosed monoidal category $(\mathcal{C}, \otimes_1 ,\unit_1)$, i.e. $\mathcal{C}$ is $\star$-autonomous.

A result of Seely and Cockett (cf. \cite[Theorem 4.5]{cockettseely1997}) shows that in the symmetric case the contrary is also true: a symmetric $\star$-autonomous category induces a symmetric linearly distributive category with duality. The original statement states that “\textit{the notions of symmetric weakly distributive categories with negation and (symmetric)} $\star$-\textit{autonomous categories coincide}”. However, as far as we understand, the proof does merely give an existence statement, and does not use any notion of isomorphism. Thus, we give the following variant.
\begin{proposition}
Every symmetric linearly distributive category with duality induces a symmetric $\star$-autonomous category and vice versa.
\end{proposition}

In fact, the statement can be generalised to a non-symmetric version. If $(\mathcal{C}, \otimes_1, \unit)$ is a non-symmetric $\star$-autonomous category with dualities $(-)^{\star},\ ^{\star}(-)$, dualizing object $k$ and internal homs $\ihoml, \ihomr$, define a second tensor product $\otimes_2$ on $\mathcal{C}$ as in \eqref{secondtensorproductintro}, i.e. on $x,y \in \mathcal{C}$ let
\[
x \otimes_2 y :=\ ^{\star}(y^{\star} \otimes x^{\star}).
\]
 Since internal homs internalise themself, we can identify 
 \begin{align}
\label{secondtensorrepres}
\begin{split}
	x \otimes_2 y 
	&=\ ^{\star}(y^{\star} \otimes x^{\star}) 
	\cong \ihomr(\ihoml(y,k) \otimes \ihoml(x,k),k) \\
	&\cong \ihomr(\ihoml(x,k), \ihomr(\ihoml(y,k),k))
	\cong \ihomr(\ihoml(x,k),y) \\
	&\cong \ihomr(x^{\star},y).
\end{split}	
\end{align}
Then one can deduce, for example, a left distributor $\delta^L_{x,y,z}: x \otimes_1 (y \otimes_2 z) \rightarrow (x \otimes_1 y) \otimes_2 z$ for a linearly distributive category $\mathcal{C}$ with tensor products $\otimes_1, \otimes_2$ via

\begin{prooftree}
\AxiomC{$(x \otimes_1 y)^{\star} \rightarrow (x \otimes_1 y)^{\star}$}
\LeftLabel{$\eqref{starautonomfrobeniuseq}$}
\UnaryInfC{$(x \otimes_1 y)^{\star} \otimes_1 x \rightarrow y^{\star}$}
\AxiomC{$\ihomr(y^{\star}, z) \rightarrow \ihomr(y^{\star}, z)$}
\RightLabel{$\eqref{ihomrdefeq}$}
\UnaryInfC{$y^{\star} \otimes_1 \ihomr(y^{\star}, z) \rightarrow z$}
\RightLabel{$\eqref{ihomldefeq}$}
\UnaryInfC{$y^{\star} \rightarrow \ihoml(\ihomr(y^{\star}, z), z)$}
\LeftLabel{$\circ$}
\BinaryInfC{$(x \otimes_1 y)^{\star} \otimes_1 x \rightarrow \ihoml(\ihomr(y^{\star}, z), z)$}
\LeftLabel{$\eqref{ihomldefeq}$}
\UnaryInfC{$(x \otimes_1 y)^{\star} \otimes_1 x \otimes_1 \ihomr(y^{\star}, z) \rightarrow z$}
\LeftLabel{$\eqref{ihomrdefeq}$}
\UnaryInfC{$x \otimes_1 \ihomr(y^{\star}, z) \rightarrow \ihomr( (x \otimes_1 y)^{\star}, z)$}
\LeftLabel{$\eqref{secondtensorrepres}$}
\UnaryInfC{$x \otimes_1 (y \otimes_2 z) \rightarrow (x \otimes_1 y) \otimes_2 z$}.
\end{prooftree}

\section{$\dag$-Frobenius pseudomonoids}

The main goal of this section is to relate $\star$-autonomous categories to $\dag$-Frobenius pseudomonoids in the bicategory of profunctors.

A \textit{Frobenius algebra} is a finite-dimensional $\mathbbm{k}$-algebra $A$ equipped with a non-degener-\\ate \textit{Frobenius form} $\sigma: A \otimes A \rightarrow \mathbbm{k}$ such that $\sigma(ab \otimes c) = \sigma(a \otimes bc)$ for $a,b,c \in A$. They have been studied already in the early 20th century (cf. \cite{nakayama1939frobeniusean}, \cite{nakayama1941frobeniusean},\cite{brauer1937regular}) and in recent times the interest has been revived due to their connection to $2$-dimensional quantum field theories (cf. \cite{abrams1996two}).

Examples of Frobenius algebras include the $\mathbbm{k}$-algebra $\textnormal{Mat}_n(\mathbbm{k})$ of $n \times n$-matrices with entries in some field $\mathbbm{k}$ and the trace; the group algebra $\mathbbm{k}\lbrack G \rbrack$ of a finite group $G$ with Frobenius form given by the coefficient of the unit element $1_G$ in the product of two arguments; and, by a result of Larson and Sweedler, every finite-dimensional Hopf algebra (cf. \cite{larson1969associative}).

Our next goal is to adapt existing definitions of Frobenius algebras in monoidal categories (cf. \cite{fuchs2009frobenius}) to monoidal bicategories and prove the equivalence of them.
The main approach is to replace \textit{properties}, i.e. relations between $1$-morphisms, by \textit{structures}, i.e. $2$-morphisms. We begin with an adaption of of rigidity (cf. \autoref{rigid}) to monoidal bicategories.
\begin{definition}[Biexact pairing]
	 	Let $(\mathcal{C}, \otimes, \unit)$ be a strict monoidal bicategory. A \textit{biexact pairing} $x \dashv y$, between objects $x,y \in \mathcal{C}$ is given by a $1$-morphism $\varepsilon_1 \equiv \includegraphics[height=0.5cm]{pics/proof1/vareps1}: x \otimes y \rightarrow \unit$, a $1$-morphism $\eta_1 \equiv \includegraphics[height=0.5cm]{pics/proof1/eta1}: \unit \rightarrow y \otimes x$, as well as invertible $2$-morphisms
	 	\begin{equation}
	 	\label{zigzagidentities2}
	 		\varepsilon_2: \includegraphics[height=1cm]{pics/proof1/vareps2} \cong \includegraphics[height=1cm, width=0.5cm]{pics/proof1/idB} 
	 		\quad \textnormal{and} \quad
	 		 \eta_2: \quad \includegraphics[height=1cm, width=0.5cm]{pics/proof1/idA} \cong \includegraphics[height=1cm]{pics/proof1/eta2}.
	 	\end{equation}
If there exists a biexact pairing $x \dashv y$, we call $x$ \textit{left dual} of $y$, and $y$ \textit{right dual} of $x$.
	 	 \end{definition} 	
Note that in \cite[Def. 8, Def. 13]{mellies2013dialogue}, in addition to our definition, the $2$-morphisms are also  subject to some constraints.	 
	 	 
In the following Lemma we show the uniqueness of right duals in monoidal bicategories. We postpone the proof to the Appendix. The statement can be restated for the uniqueness of left duals. 

\begin{lemma}
\label{dualsuniquebicat}
	Let $(\mathcal{C}, \otimes, \unit)$ be a strict monoidal bicategory and $A,B,B' \in \mathcal{C}$. Assume $A \dashv B$ and $A \dashv B'$ witnessed by $\varepsilon_1, \eta_1$ and $\varepsilon_1', \eta_1'$. Then there exist $1$-morphisms $\alpha: B \rightarrow B'$ and $\beta: B' \rightarrow B$ as well as invertible $2$-morphisms $\beta \circ \alpha \cong \textnormal{id}_{B}$ and $\alpha \circ \beta \cong \textnormal{id}_{B'}$. The $1$-morphism $\alpha$ preserves evaluation and coevaluation morphisms in the sense that there exist invertible $2$-morphisms $\varepsilon_1' \circ (\textnormal{id}_{A} \otimes \alpha) \cong \varepsilon_1$ and $(\alpha \otimes \textnormal{id}_A) \circ \eta_1 \cong \eta_1'$. Furthermore, $\alpha$ is unique in the sense that for every other $1$-morphism $\tilde{\alpha}: B \rightarrow B'$ that preserves evaluation and coevaluation morphisms in this way, there exists an invertible $2$-morphism $\alpha \cong \tilde{\alpha}$. 
\end{lemma}

The following \autoref{biexactpairinginvertible} is a generalisation of 
\autoref{linearlydistributiveinducedmorphisms} to bicategories. It can  be obtained in a straightforward way by using the the zig-zag $2$-morphisms \eqref{zigzagidentities2}.

\begin{lemma}
\label{biexactpairinginvertible}
	Let $(\mathcal{C}, \otimes, \unit)$ be a strict monoidal bicategory with $x,y, x^{\star},\ ^{\star}y \in \mathcal{C}$ and biexact pairings $x \dashv x^{\star}$,$^{\star}y \dashv y$. Then there exist equivalences of categories		
	\[
	_{\mathcal{C}} \langle x \otimes y, \unit_2 \rangle \cong\ _{\mathcal{C}}\langle y, x^{\star} \rangle \cong\ _{\mathcal{C}}\langle x,\ ^{\star} y \rangle,
	\] 
	\[
	_{\mathcal{C}} \langle \unit, y \otimes_2 x \rangle \cong\ _{\mathcal{C}} \langle x^{\star}, y \rangle \cong\  _{\mathcal{C}} \langle ^{\star}y, x \rangle.
	\]
	Moreover, any pair of $1$-morphisms $x \otimes y \rightarrow \unit$ and $\unit \rightarrow y \otimes x$ provides a biexact pairing $x \dashv y$ if and only if the associated $1$-morphisms $y \rightarrow x^{\star}$ and $x^{\star} \rightarrow y$ are mutually inverse up to $2$-morphisms.
\end{lemma}

Central to the understanding of algebras, and in particular Frobenius algebras, is the notion of monoids. We've already encountered the notion of monoidal categories. Surprisingly, inside of monoidal categories it is possible to define another categorification of a monoid -- a \textit{monoid object} (cf. \cite{mac2013categories}). This can be seen as an example of the \textit{micro-macrocosmos principle}. Let us give some examples of monoidal objects. A monoid object in $\Set$ is a classical monoid. A monoid object in the endofunctor category $\End(\mathcal{C}, \mathcal{C})$ of some category $\mathcal{C}$ is a monad on $\mathcal{C}$.
But most importantly, a monoid object in the category $\Vect$ is a $\mathbbm{k}$-algebra.

The notion of monoid objects is generalised to monoidal \textit{bi}categories by the following definition.
\begin{definition}[Pseudomonoid]
	A \textit{pseudomonoid} $(A, \mu, \eta) := (A, \mu, \eta, \alpha, \lambda, \varrho)$ in a monoidal bicategory $(\mathcal{C}, \otimes, \unit)$ is an object $A \in \mathcal{C}$ endowed with $1$-morphisms $\mu \equiv  \includegraphics[height=0.5cm]{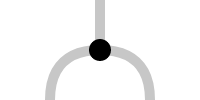}: A \otimes A \rightarrow A$ and $\eta \equiv \includegraphics[height=0.5cm]{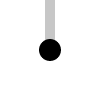}: \unit \rightarrow A$ called \textit{multiplication} and \textit{unit} of $A$, respectively, and invertible $2$-cells
		\begin{equation}
		\label{pseudomonoassoc}
			\includegraphics[height=1cm]{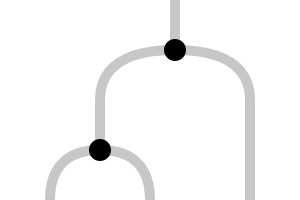} \overset{\alpha}{\cong} \includegraphics[height=1cm]{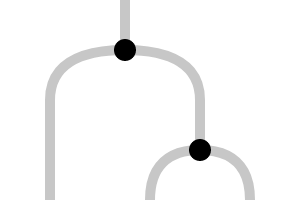},
		\end{equation}
		called \textit{associator}, and
		\begin{equation}
		\label{pseudomonounit}
 	\includegraphics[height=1cm]{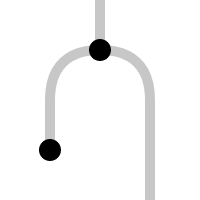} \overset{\lambda}{\cong}\includegraphics[height=1cm]{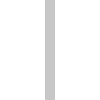} \overset{\varrho}{\cong} \includegraphics[height=1cm]{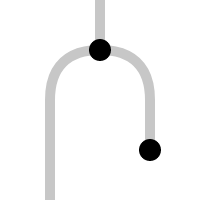},
 	\end{equation}
 	called \textit{left unitor} and \textit{right unitor},
	subject to the following commutativity constraints
	\begin{equation}
	\begin{tikzcd}[row sep=large, column sep = large]
		\includegraphics[height=1.5cm]{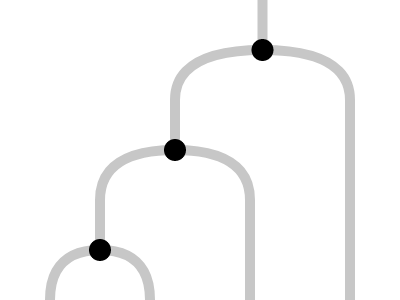} \arrow{r}{\alpha} \arrow{d}[left]{\alpha}  & \includegraphics[height=1.5cm]{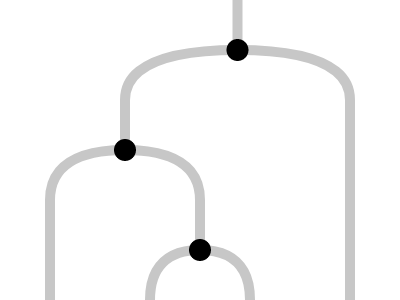} \arrow{r}{\alpha} & \includegraphics[height=1.5cm]{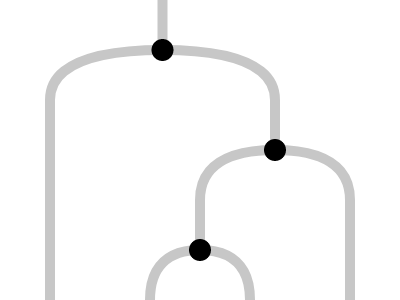} \arrow{d}{\alpha}	\\
		\includegraphics[height=1.5cm]{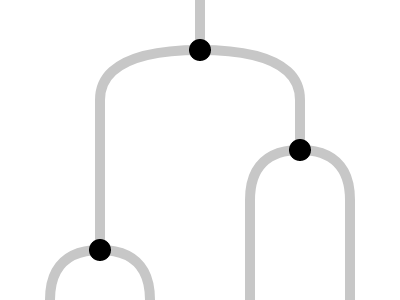} \arrow{r}[below]{\overset{\eqref{intertwiner}}{\cong}}
		& \includegraphics[height=1.5cm]{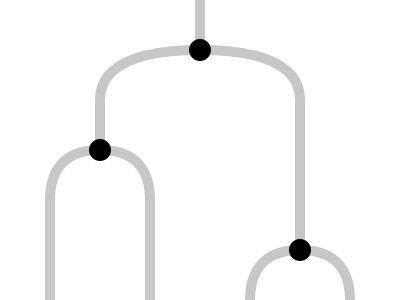} \arrow{r}[below]{\alpha}
		& \includegraphics[height=1.5cm]{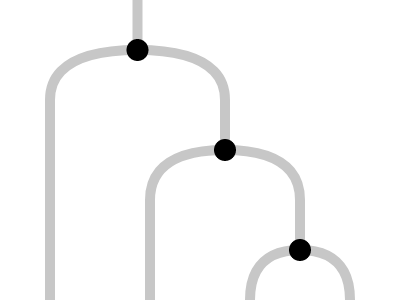},
	\end{tikzcd}
\end{equation}
\begin{equation}
	\begin{tikzcd}[row sep=large, column sep = large]
		&\includegraphics[height=0.5cm]{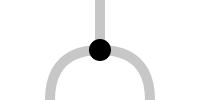} \\
		\includegraphics[height=1.5cm]{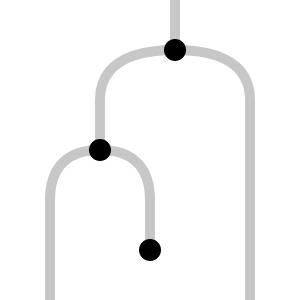} \arrow{ur}{\varrho} \arrow{rr}[below]{\alpha} && \includegraphics[height=1.5cm]{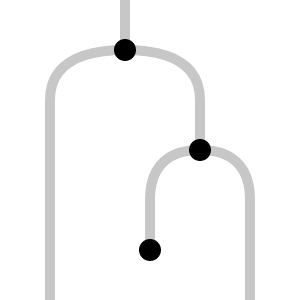} \arrow{ul}[right]{\lambda}.
	\end{tikzcd}
\end{equation}
		The invertible $2$-cells $\alpha, \lambda, \varrho,$ are part of the data of a pseudomonoid, however, in the following it will be convenient to supress them in the notation, if not explicitly needed.
		
	A \textit{pseudocomonoid} $(A, \delta, \varepsilon) := (A, \delta, \varepsilon, \alpha, \lambda, \varrho)$ in $(\mathcal{C}, \otimes, \unit)$ is a pseudomonoid in the bicategory $(\mathcal{C}, \otimes, \unit)^{\textnormal{opp}(1)}$.
	
	\end{definition}	
For example, a small monoidal category is a pseudomonoid in the monoidal bicategory $\Cat$ of small categories.

 The following result shows -- analogous to classical Frobenius algebras and Frobenius algebras in monoidal categories -- that there exist multiple equivalent definitions of what will be called Frobenius pseudomonoids in monoidal bicategories. For the sake of clarity we will postpone the proof, which extensively uses the graphical calculus, to the Appendix. 

\begin{lemma}
\label{lemmafrobpseudo}
	Let $(A,\ \mu \equiv \includegraphics[height=0.5cm]{pics/proof1/mu}: A \otimes A \rightarrow A,\ \eta \equiv \includegraphics[height=0.5cm]{pics/proof1/eta}: \unit \rightarrow A)$ be a pseudomonoid in a monoidal bicategory $(\mathcal{C}, \otimes, \unit)$. The following are equivalent:
	\begin{enumerate}
		\item There exists $\varepsilon \equiv \includegraphics[height=0.5cm]{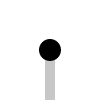}:A \rightarrow \unit$, such that$ \includegraphics[height=1cm]{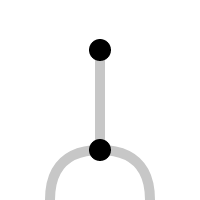}: A \otimes A \rightarrow \unit$ is an evaluation for a biexact pairing $A \dashv A$.
		\item There exists $\includegraphics[height=0.5cm]{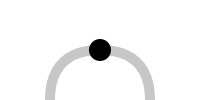}: A \otimes A \rightarrow \unit$, that is an evaluation for a biexact pairing $A \dashv A$ and an invertible $2$-cell
		\begin{equation}
		\label{form}
			\includegraphics[height=1cm]{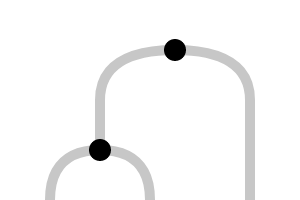} \cong \includegraphics[height=1cm]{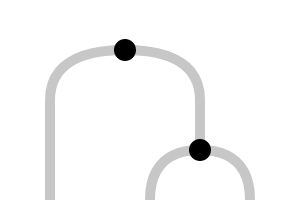}
		\end{equation}
		\item There exists a pseudocomonoid structure $(A,\ \delta \equiv\includegraphics[height=0.5cm]{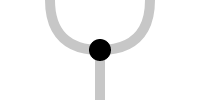}: A \rightarrow A \otimes A,\ \varepsilon \equiv\includegraphics[height=0.5cm]{pics/proof1/vareps}: A \rightarrow \unit)$ on $A$ and two invertible $2$-cells
			\begin{equation}
			\label{pseudocomonoid2cell}
				\includegraphics[height=1cm]{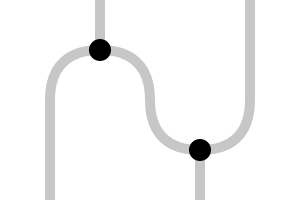} \cong \includegraphics[height=1cm]{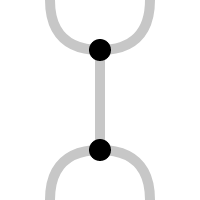} \cong \includegraphics[height=1cm]{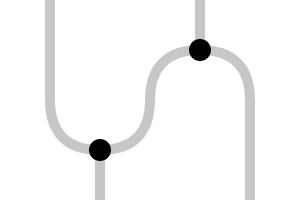}
			\end{equation}
	\end{enumerate}
\end{lemma}	

Note that we merely give an existence statement and do not attempt to prove an equivalence of categories spanned by structures on pseudomonoids, which would be beyond the scope of this thesis.

\begin{definition}[Frobenius pseudomonoid]
	A  pseudomonoid with one of the structures of \autoref{lemmafrobpseudo} is called \textit{Frobenius}.
\end{definition}
	Be aware that in \cite{mellies2013dialogue} it is in addition required that a variant of Mac Lane's pentagon diagram commutes. 
	
	Next we will define the bicategory in which to consider Frobenius pseudomonoids: the monoidal bicategory of profunctors. 
	
	For an introduction to enriched categories see e.g. \cite{borceux2000short}.
	
	\begin{definition}[$\Vcat$]
	\label{vcat}
	Let $(\mathcal{V}, \otimes, \unit)$ be a symmetric closed monoidal category. Let $\Vcat$ be the symmetric closed bicategory of small $\mathcal{V}$-enriched categories, $\mathcal{V}$-functors and $\mathcal{V}$-natural transformations, defined in the following way.
	The tensor product $\mathcal{C} \otimes \mathcal{D}$ of $\mathcal{V}$-enriched categories $\mathcal{C}, \mathcal{D}$ is the $\mathcal{V}$-enriched category with objects
$
\textnormal{ob}(\mathcal{C} \otimes \mathcal{D}) = \textnormal{ob}(\mathcal{C}) \times \textnormal{ob}(\mathcal{D})
$ and morphisms
	\[ \textnormal{Hom}_{\mathcal{C} \otimes \mathcal{D}}((x,y),(x',y')) = \textnormal{Hom}_{\mathcal{C}}(x,x') \otimes \textnormal{Hom}_{\mathcal{D}}(y,y') \in \mathcal{V}.
	\]
	 for $x,x' \in \mathcal{C},\ y,y' \in \mathcal{D}$. The monoidal unit with respect to this tensor product is given by the category with one object $\star$, such that $\textnormal{Hom}(\star, \star) = \unit \in \mathcal{V}$. A symmetric braiding of $\Vcat$ is induced by the symmetric braiding of $\mathcal{V}$. The internal hom $\underline{\textnormal{Hom}}(\mathcal{C}, \mathcal{D})$ of two $\mathcal{V}$-enriched categories is given by the $\mathcal{V}$-enriched category that has as objects $\mathcal{V}$-enriched functors $F: \mathcal{C} \rightarrow \mathcal{D}$, and as morphisms the end (cf. \autoref{defend})
	 \begin{equation}
	 \label{endmorphisms}
	 	\textnormal{Hom}_{\underline{\textnormal{Hom}}(\mathcal{C}, \mathcal{D})}(F, G) = \int_{x \in \mathcal{C}} \textnormal{Hom}_{\mathcal{D}}(F(\overline{x}),G(x)) \in \mathcal{V}.
	 \end{equation}
\end{definition}

This can be seen as a generalisation of the $\Set$-enriched category $\Cat$ of small categories (cf. \eqref{endmorphisms} with \eqref{natcoend}).

\begin{definition}[$\Vmod$] 
\label{vmod}
Assume that $\mathcal{V}$ is cocomplete. The closed symmetric monoidal bicategory $\Vmod$ has the same objects and  monoidal structure as $\Vcat$, but its morphisms are given as
	 \[
	 \textnormal{Hom}_{\mathcal{V}\textnormal{-Mod}}(\mathcal{C}, \mathcal{D}) :=  \textnormal{Hom}_{\mathcal{V}\textnormal{-Cat}}(\mathcal{D}^{\textnormal{opp}(1)} \otimes \mathcal{C}, \mathcal{V}).
	 \]
	A morphism $F \in  \textnormal{Hom}_{\mathcal{V}\textnormal{-Mod}}(\mathcal{C}, \mathcal{D})$ is called \textit{profunctor} (or \textit{module}, \textit{distributor}) from $\mathcal{C}$ \textit{to} $\mathcal{D}$  and is commonly denoted by $F: \mathcal{C} \nrightarrow \mathcal{D}$. The composition of profunctors $F: \mathcal{A} \nrightarrow \mathcal{B}$ and $G: \mathcal{B} \nrightarrow \mathcal{C}$ is given as the coend (cf. \autoref{defend}),
	 \begin{equation}
	 \label{profunctorscompo}
	 	(G \circ F)(\overline{c},a) = \int^{b \in \mathcal{B}} F(\overline{b},a) \otimes G(\overline{c},b), \quad c \in \mathcal{C}, a \in \mathcal{A},
	 \end{equation}
	 and is associative up to isomorphism. The coend exists since the objects of $\Vmod$ are small categories, and $\mathcal{V}$ is cocomplete (cf. \cite[Cor. 5.2]{mac2013categories}). In the case of $\mathcal{V} = (\Set, \times, \lbrace \star \rbrace)$, it is common to write $\Prof := \Vmod$ and speak of the category of profunctors.
\end{definition}

The term \textit{modules} can be motivated by the following observation. Consider $\mathcal{V}= \Vect$, then the morphism spaces of categories in $\Vmod$ are $\mathbbm{k}$-vector spaces, such that the composition is bilinear. In particular, categories in $\Vmod$ with one object $\star$ are in bijection to $\mathbbm{k}$-algebras with composition as multiplication. A module $F: \mathcal{A} \nrightarrow \mathcal{B}$ between such categories $\mathcal{A}$ and $\mathcal{B}$ then can be seen as a $(A,B)$-bimodule $M := F(\star, \star)$ with actions $\rho_A, \rho_B$ of algebras $A := \textnormal{Hom}_{\mathcal{A}}(\star, \star)$ and  $B := \textnormal{Hom}_{\mathcal{B}}(\star, \star)$ given by $\rho_A(a) := F(\textnormal{id}_{\star}, a)$ and $\rho_B(b) := F(b, \textnormal{id}_{\star})$.

	We will now define inclusions of $\Vcat$ into $\Vmod$. To this end note that every functor $F: \mathcal{C} \rightarrow \mathcal{D}$ yields modules $F_{\star}: \mathcal{C} \nrightarrow \mathcal{D}$ and $F^{\star}: \mathcal{D} \nrightarrow \mathcal{C}$ defined by 
	\begin{equation}
	\label{defstar}
		 F_{\star}(\overline{d},c) = \textnormal{Hom}_{\mathcal{D}}(d, F(c)) \quad \textnormal{and} \quad F^{\star}(\overline{c},d) = \textnormal{Hom}_{\mathcal{D}}(F(c), d)
	\end{equation} 
	  for  $c \in \mathcal{C}, d \in \mathcal{D}$.
	 Moreover, every $\mathcal{V}$-natural transformation $\eta: F \Rightarrow G$ of $\mathcal{V}$-functors $F,G: \mathcal{C} \rightarrow \mathcal{D}$ yields $\mathcal{V}$-natural transformations $\eta_{\star}: F_{\star} \Rightarrow G_{\star}$ and $\eta^{\star}: G^{\star} \Rightarrow F^{\star}$ defined by
	 $
	 (\eta_{\star})_{(\overline{d},c)} =  \textnormal{Hom}_{\mathcal{D}}(\textnormal{id}_d, \eta_c)$ and $(\eta^{\star})_{(\overline{d},c)} = \textnormal{Hom}_{\mathcal{D}}(\eta_c, \textnormal{id}_d)$.
\begin{lemma}
	Mapping $\mathcal{V}$-categories to itself, $\mathcal{V}$-functors $F$ to $F_{\star}$ and $F^{\star}$, respectively, and $\mathcal{V}$-natural transformations $\eta$ to $\eta_{\star}$ and $\eta^{\star}$, respectively, yields monoidal 2-functors
	 \begin{equation}
	 	\label{starmonoidalfunctor}
(-)_{\star}: \Vcat \rightarrow \Vmod \quad \textnormal{and} \quad (-)^{\star}: \Vcat \rightarrow (\Vmod)^{\textnormal{opp}(1,2)}.
	 \end{equation}
	 Here $(-)^{\textnormal{opp}(1,2)}$ means reversing the $1$- and $2$-cells.
\end{lemma}

We say $G:  \mathcal{C} \nrightarrow \mathcal{D}$ is \textit{representable}, if there exists a functor $F:\mathcal{C} \rightarrow \mathcal{D}$, such that $G = F_{\star}$. The following result (cf. \cite[Prop. 7.9.1]{borceux50handbook}) shows, in particular, that in the case of profunctors, having a right adjoint implies being representable.

\begin{lemma}
\label{rightadjointrep}
For any functor $F: \mathcal{C} \rightarrow \mathcal{D}$ between locally small categories $\mathcal{C}, \mathcal{D}$, the profunctor $F^{\star}$ is right adjoint to $F_{\star}$ in $\Vmod$, 
\begin{equation}
	F_{\star} \dashv F^{\star}.
\end{equation}
Conversely, every profunctor $G: \mathcal{C} \nrightarrow \mathcal{D}$ with a right adjoint is representable,  $G = F_{\star}$ for some $F: \mathcal{C} \rightarrow \mathcal{D}$.  
\end{lemma}

We will not state the full proof, but note that the coevaluation $\textnormal{id}_{\mathcal{C}} \Rightarrow F^{\star} \circ F_{\star}$ is given by the action of $F$ on the homspace, 
\[
\textnormal{id}_{\mathcal{C}}(\overline{c}, c') \overset{F}{\rightarrow} (F^{\star} \circ F_{\star})(\overline{c}, c') \cong \int^{d \in \mathcal{D}}\ _{\mathcal{D}} \langle F(c), d \rangle \otimes\  _{\mathcal{D}} \langle d, F(c') \rangle \cong\ _{\mathcal{D}} \langle F(c), F(c') \rangle.
\]

All profunctors are representable, if one imposes the mild condition of Cauchy completeness (cf. \cite[Theor. 7.9.3]{borceux50handbook}). However, here we will instead concentrate on the condition of having a right adjoint.

\begin{definition}[$\dag$-Frobenius pseudomonoid]
	Let $(\mathcal{C}, \otimes, \unit)$ be monoidal bicategory. A Frobenius pseudomonoid $(A, \mu: A \otimes A \rightarrow A, \eta: \unit \rightarrow A, \sigma: A \otimes A \rightarrow \unit)$ in $\mathcal{C}$, such that $\mu$ and $\eta$ have right adjoints, is called $\dag$\textit{-Frobenius pseudomonoid}. 
\end{definition}

The main result of this section is given as follows.

\begin{theorem}
	A $\dag$-Frobenius pseudomonoid in $\Prof$ induces a $\star$-autonomous category and vice versa.
\end{theorem}

The statement is based on \cite{street2004frobenius} and well-established. An extension to an equivalence of categories would be valuable, but needs some more thought. Here we work out a complete proof, split in two parts. 

We will extensively use (cf. \cite[Cor. 5]{fuchs2016coends}) that for any category $\mathcal{C}$, and any pair of objects $u,v \in \mathcal{C}$,  the coend (cf. \autoref{defend}) of the functor $_{\mathcal{C}} \langle u,- \rangle \times\ _{\mathcal{C}} \langle -, v \rangle: \mathcal{C}^{\textnormal{opp}(1)} \times \mathcal{C} \rightarrow \Set$ exists, and is given as
\begin{equation}
\label{coendrechnung}
	\int^{x \in \mathcal{C}}\ _{\mathcal{C}} \langle u,x \rangle \times\ _{\mathcal{C}} \langle x, v \rangle =\ _{\mathcal{C}} \langle u,v \rangle.
\end{equation}

\begin{theorem}
\label{frobeniusinducesstarautonom}
	A $\dag$-Frobenius pseudomonoid in $\Prof$ induces a $\star$-autonomous category.
\end{theorem}
\begin{proof}
\begin{enumerate}
	\item Let $(\mathcal{C}^{\textnormal{opp}(1)}, \mu, \eta, \alpha, \lambda, \varrho, \sigma)$ be a $\dag$-Frobenius pseudomonoid in the monoidal autonomous bicategory $\Prof$. Since by definition $\mu$ and $\eta$ have right adjoints, they are representable by \autoref{rightadjointrep}, i.e. there exist functors $\otimes_{\mathcal{C}}: \mathcal{C} \times \mathcal{C} \rightarrow \mathcal{C}$ and $\eta_{\mathcal{C}}: \unit_{\Cat} \rightarrow \mathcal{C}$ such that 
	\begin{equation}
	\label{muetaconstruction}
		\mu = (\otimes_{\mathcal{C}}^{\textnormal{opp}(1)})_{\star} \quad \textnormal{and} \quad \eta = (\eta_{\mathcal{C}}^{\textnormal{opp}(1)})_{\star}.
	\end{equation}
We will now construct an monoidal structure on $\mathcal{C}$ with tensor product $\otimes_{\mathcal{C}}$ and unit $\eta_{\mathcal{C}}(\star)$.
	First of all, by the very definition of $(-)_{\star}$ and the previous representation of $\mu$, it holds for $x,y,z \in \mathcal{C}$
	\begin{equation}
	\label{murepr}
		\mu(x,\overline{y},\overline{z}) \cong\  _{\mathcal{C}} \langle y \otimes_{\mathcal{C}} z, x \rangle.
	\end{equation}
	The associator $\alpha: \mu \circ (\mu \times \textnormal{id}_{\mathcal{C}^{\textnormal{opp}(1)}}) \overset{\simeq}{\rightarrow} \mu \circ (\textnormal{id}_{\mathcal{C}^{\textnormal{opp}(1)}} \times \mu)$ in $\Prof$ yields an associator $\alpha_{\mathcal{C}}: \otimes_{\mathcal{C}} \circ (\otimes_{\mathcal{C}} \times \textnormal{id}_{\mathcal{C}}) \overset{\simeq}{\rightarrow} \otimes_{\mathcal{C}} \circ (\textnormal{id}_{\mathcal{C}} \times \otimes_{\mathcal{C}})$ in $\Cat$ in the following way. For $a,b,c,d \in \mathcal{C}$ we find
	\begin{align*}
		&\mu \circ (\mu \times \textnormal{id}_{\mathcal{C}^{\textnormal{opp}(1)}})(a,\overline{b}, \overline{c}, \overline{d})
		= \int^{(\overline{x}, \overline{y}) \in (\mathcal{C}^{\textnormal{opp}(1)})^2} (\mu \times \textnormal{id}_{\mathcal{C}^{\textnormal{opp}(1)}})(x,y,\overline{b}, \overline{c}, \overline{d}) \times \mu(a,\overline{x}, \overline{y}) \\
		&= \int^{(\overline{x}, \overline{y}) \in (\mathcal{C}^{\textnormal{opp}(1)})^2} \mu(x, \overline{b}, \overline{c}) \times \textnormal{id}_{\mathcal{C}^{\textnormal{opp}(1)}}(y, \overline{d}) \times \mu(a,\overline{x}, \overline{y}) \\
		&= \int^{(\overline{x}, \overline{y}) \in (\mathcal{C}^{\textnormal{opp}(1)})^2}\ _{\mathcal{C}} \langle b \otimes_{\mathcal{C}} c, x \rangle \times\ _{\mathcal{C}} \langle d, y \rangle \times\ _{\mathcal{C}} \langle x \otimes_{\mathcal{C}} y, a \rangle \\
		&= \int^{(\overline{x}, \overline{y}) \in (\mathcal{C}^{\textnormal{opp}(1)})^2}\ _{\mathcal{C}} \langle (b \otimes_{\mathcal{C}} c) \otimes_{\mathcal{C}} d, x \otimes_{\mathcal{C}} y \rangle \times\ _{\mathcal{C}} \langle x \otimes_{\mathcal{C}} y, a \rangle
		\cong\ _{\mathcal{C}} \langle (b \otimes_{\mathcal{C}} c) \otimes_{\mathcal{C}} d, a \rangle.
	\end{align*}
	Similarly one shows that 
	\[
	\mu \circ (\textnormal{id}_{\mathcal{C}^{\textnormal{opp}(1)}} \times \mu)(a,\overline{b}, \overline{c}, \overline{d}) \cong\ _{\mathcal{C}} \langle b \otimes_{\mathcal{C}} (c \otimes_{\mathcal{C}} d), a \rangle.
	\]
	Thus, the associator $\alpha$ yields for fixed $b,c,d \in \mathcal{C}$ and any $a \in \mathcal{C}$ a natural isomorphism $_{\mathcal{C}} \langle (b \otimes_{\mathcal{C}} c) \otimes_{\mathcal{C}} d, a \rangle \cong\ _{\mathcal{C}} \langle b \otimes_{\mathcal{C}} (c \otimes_{\mathcal{C}} d), a \rangle$. Finally, the Yoneda Lemma implies the existence of an isomorphism 
	\[
	(\alpha_{\mathcal{C}})_{b,c,d}: (b \otimes_{\mathcal{C}} c) \otimes_{\mathcal{C}} d \overset{\simeq}{\rightarrow} b \otimes_{\mathcal{C}} (c \otimes_{\mathcal{C}} d).
	\]
	By the very construction we find for any $a,b,c,d \in \mathcal{C}$ a natural isomorphism of profunctors
	\begin{equation}
	\label{alphacstar2}
		\alpha_{a,\overline{b},\overline{c},\overline{d}} \cong ((\alpha_{\mathcal{C}}^{\textnormal{opp}(1)})_{\star})_{a,\overline{b},\overline{c},\overline{d}}.
	\end{equation}
	Analogously, left unitor $\lambda: \mu \circ (\textnormal{id}_{\mathcal{C}^{\textnormal{opp}(1)}} \times \eta) \overset{\simeq}{\rightarrow} \textnormal{id}_{\mathcal{C}^{\textnormal{opp}(1)}}$ and right unitor $\varrho: \mu \circ (\eta \times \textnormal{id}_{\mathcal{C}^{\textnormal{opp}(1)}}) \overset{\simeq}{\rightarrow} \textnormal{id}_{\mathcal{C}^{\textnormal{opp}(1)}}$ in $\Prof$ induce left and right unitors $\lambda_{\mathcal{C}}, \varrho_{\mathcal{C}}$ in $\Cat$ in the following way. For $a,b \in \mathcal{C}$ we find
	\begin{align*}
		&\mu \circ  (\textnormal{id}_{\mathcal{C}^{\textnormal{opp}(1)}} \times \eta)(a,\overline{b}, \overline{\star})
		= \int^{(\overline{x}, \overline{y}) \in (\mathcal{C}^{\textnormal{opp}(1)})^2}  (\textnormal{id}_{\mathcal{C}^{\textnormal{opp}(1)}} \times \eta)(x,y, \overline{b}, \overline{\star}) \times \mu(a,b, \overline{x}, \overline{y}) \\
		&= \int^{(\overline{x}, \overline{y}) \in (\mathcal{C}^{\textnormal{opp}(1)})^2} \textnormal{id}_{\mathcal{C}^{\textnormal{opp}(1)}}(x, \overline{b}) \times \eta(y, \overline{\star}) \times \mu(a, \overline{x}, \overline{y}) \\
		&=  \int^{(\overline{x}, \overline{y}) \in (\mathcal{C}^{\textnormal{opp}(1)})^2}\ _{\mathcal{C}} \langle b,x \rangle \times\ _{\mathcal{C}}\langle \eta_{\mathcal{C}}(\star), y \rangle \times\ _{\mathcal{C}} \langle x \otimes_{\mathcal{C}} y, a \rangle \\
		&= \int^{(\overline{x}, \overline{y}) \in (\mathcal{C}^{\textnormal{opp}(1)})^2}\ _{\mathcal{C}} \langle b \otimes_{\mathcal{C}} \eta_{\mathcal{C}}(\star), x \otimes_{\mathcal{C}} y \rangle \times\ _{\mathcal{C}} \langle x \otimes_{\mathcal{C}} y, a \rangle \\
		&=\ _{\mathcal{C}} \langle b \otimes_{\mathcal{C}} \eta_{\mathcal{C}}(\star), a \rangle.
	\end{align*}
	Since moreover $\textnormal{id}_{\mathcal{C}^{\textnormal{opp}(1)}}(a, \overline{b}) =\ _{\mathcal{C}} \langle b, a \rangle$, it follows that the left unitor $\lambda$ provides an isomorphism $_{\mathcal{C}} \langle b \otimes_{\mathcal{C}} \eta_{\mathcal{C}}(\star), a \rangle \cong\ _{\mathcal{C}} \langle b, a \rangle$. Thus the Yoneda Lemma yields a natural isomorphism $(\lambda_{\mathcal{C}})_{b}: b \otimes_{\mathcal{C}} \eta_{\mathcal{C}}(\star) \overset{\simeq}{\rightarrow} b$. Analogously one constructs a right unitor $\varrho_{\mathcal{C}}$. Again, by the very construction we find for any $a,b \in \mathcal{C}$ isomorphisms of profunctors
		\begin{equation}
		\label{unitorstar2}
		\lambda_{a,\overline{b},\overline{\star}} \cong ((\lambda_{\mathcal{C}}^{\textnormal{opp}(1)})_{\star})_{a,\overline{b},\overline{\star}} \quad \textnormal{and} \quad \varrho_{a,\overline{b},\overline{\star}} \cong ((\varrho_{\mathcal{C}}^{\textnormal{opp}(1)})_{\star})_{a,\overline{b},\overline{\star}}.
	\end{equation}
	
	Combined, this provides a monoidal category $(\mathcal{C}, \otimes_{\mathcal{C}}, \eta_{\mathcal{C}}(\star), \alpha_{\mathcal{C}}, \lambda_{\mathcal{C}}, \varrho_{\mathcal{C}})$.
	\item By assumption the form  $\sigma: \mathcal{C}^{\textnormal{opp}(1)} \times  \mathcal{C}^{\textnormal{opp}(1)} \nrightarrow \unit_{\Prof}$ provides a biexact pairing $\mathcal{C}^{\textnormal{opp}(1)} \dashv \mathcal{C}^{\textnormal{opp}(1)}$ with some coevaluation $\widetilde{\sigma}: \unit_{\Prof} \nrightarrow \mathcal{C}^{\textnormal{opp}(1)} \times  \mathcal{C}^{\textnormal{opp}(1)}$. Thus \autoref{biexactpairinginvertible} yields mutually quasi-inverse profunctors $\tilde{D}: \mathcal{C}^{\textnormal{opp}(1)} \nrightarrow \mathcal{C}$ and $\tilde{D}^{-1}: \mathcal{C} \nrightarrow \mathcal{C}^{\textnormal{opp}(1)}$ related to the form $\sigma$ by
	\begin{equation}
	\label{sigmaeq}
		\sigma(\overline{\star},\overline{x}, \overline{y}) = \tilde{D}(\overline{x}, \overline{y}).
	\end{equation}
	Every equivalence can be refined to an adjoint equivalence. Thus we can assume the mutually quasi-inverse profunctors $\tilde{D}$ and $\tilde{D}^{-1}$ to be in an adjoint relation. Then it follows from \autoref{rightadjointrep} that $\tilde{D}$ and $\tilde{D}^{-1}$ are representable, i.e. there exist functors $D: \mathcal{C} \rightarrow  \mathcal{C}^{\textnormal{opp}(1)}$ and $D^{-1}: \mathcal{C}^{\textnormal{opp}(1)} \rightarrow  \mathcal{C}$, such that 
	\begin{equation}
	\label{Drepr}
		\tilde{D} \cong (D^{\textnormal{opp}(1)})_{\star} \quad \textnormal{and} \quad \tilde{D}^{-1} \cong ((D^{-1})^{\textnormal{opp}(1)})_{\star}.
	\end{equation}  In fact, the functors $D$ and $D^{-1}$ are mutually quasi-inverse. Indeed, for any $a,b \in \mathcal{C}$
	\begin{align*}
		_{\mathcal{C}} \langle b,a \rangle 
		&\cong \textnormal{id}_{\mathcal{C}^{\textnormal{opp}(1)}}(a,\overline{b}) 
		\cong \widetilde{D} \circ \widetilde{D}^{-1}(a,\overline{b}) 
		\overset{\eqref{Drepr}}{\cong} (D^{\textnormal{opp}(1)})_{\star} \circ ((D^{-1})^{\textnormal{opp}(1)})_{\star}(a,\overline{b}) \\
		 &\overset{\eqref{starmonoidalfunctor}}{\cong} ((D \circ D^{-1})^{\textnormal{opp}(1)})_{\star}(a,\overline{b})
		 \overset{\eqref{defstar}}{\cong}\ _{\mathcal{C}} \langle D \circ D^{-1}(b), a \rangle,
	\end{align*}
	which by the Yoneda Lemma induces a natural isomorphism $b \cong D \circ D^{-1}(b)$. Analogously one shows $b \cong D^{-1} \circ D(b)$.
	The previous results yield
	\begin{equation}
		\label{sigmarepr}
		\sigma(\star, \overline{x}, \overline{y}) 
		\overset{\eqref{sigmaeq}}{=}   \tilde{D}(\overline{x}, \overline{y}) 
		\overset{\eqref{Drepr}}{\cong} (D^{\textnormal{opp}(1)})_{\star}(\overline{x}, \overline{y}) 
		\overset{\eqref{defstar}}{=}\ _{\mathcal{C}} \langle x, D^{\textnormal{opp}(1)}(\overline{y}) \rangle.
	\end{equation}
	\item
	 By assumption there exists an invertible $2$-cell
	\begin{equation}
	\label{invariance}
	\sigma \circ (\mu \otimes \textnormal{id}_{\mathcal{C}^{\textnormal{opp}(1)}}) \equiv \includegraphics[height=1cm]{pics/proof1/frobenius2} \cong \includegraphics[height=1cm]{pics/proof1/frobenius1} \equiv  \sigma \circ (\textnormal{id}_{\mathcal{C}^{\textnormal{opp}(1)}} \otimes \mu). 
	\end{equation}
	For the left hand side we compute
	 \begin{align*}
	 	&\sigma \circ (\mu \otimes \textnormal{id}_{\mathcal{C}^{\textnormal{opp}(1)}})(\overline{\star}, \overline{x}, \overline{y}, \overline{z})
	 	\overset{\eqref{profunctorscompo}}{=} \int^{(\overline{a},\overline{b}) \in (\mathcal{C}^{\textnormal{opp}(1)})^2} (\mu \otimes \textnormal{id}_{\mathcal{C}^{\textnormal{opp}(1)}})(a,b, \overline{x}, \overline{y}, \overline{z}) \times \sigma(\overline{\star}, \overline{a},\overline{b}) \\
	 	&= \int^{(\overline{a},\overline{b}) \in (\mathcal{C}^{\textnormal{opp}(1)})^2} \mu(a, \overline{x}, \overline{y}) \times  \textnormal{id}_{\mathcal{C}^{\textnormal{opp}(1)}}(b, \overline{z}) \times \sigma(\overline{\star}, \overline{a}, \overline{b}) \\
	 	&\overset{\eqref{murepr},\eqref{sigmarepr}}{=}  \int^{(\overline{a},\overline{b}) \in (\mathcal{C}^{\textnormal{opp}(1)})^2}\ _{\mathcal{C}} \langle x \otimes_{\mathcal{C}} y, a \rangle \times\ _{\mathcal{C}} \langle z,b \rangle \times\ _{\mathcal{C}} \langle a, D^{\textnormal{opp}(1)}(\overline{b}) \rangle \\
	 	&\overset{\eqref{coendrechnung}}{\cong} \int^{\overline{b} \in \mathcal{C}^{\textnormal{opp}(1)}}\ _{\mathcal{C}} \langle x \otimes_{\mathcal{C}} y, D^{\textnormal{opp}(1)}(\overline{b}) \rangle \times\ _{\mathcal{C}} \langle z,b \rangle \\
	 	&\overset{D^{-1}}{\cong} \int^{\overline{b} \in \mathcal{C}^{\textnormal{opp}(1)}}\ _{\mathcal{C}} \langle b, D^{-1}(\overline{x \otimes_{\mathcal{C}} y}) \rangle \times\ _{\mathcal{C}} \langle z,b \rangle
	 	\overset{\eqref{coendrechnung}}{\cong}\ _{\mathcal{C}} \langle z, D^{-1}(\overline{x \otimes_{\mathcal{C}} y}) \rangle \\
	 	&\overset{D}{\cong}\ _{\mathcal{C}} \langle  x \otimes_{\mathcal{C}} y, D^{\textnormal{opp}(1)}(\overline{z}) \rangle.
	 \end{align*}
	 Similarly, one finds that for the right hand side it holds
	 \[
	 \sigma \circ (\textnormal{id}_{\mathcal{C}^{\textnormal{opp}(1)}} \otimes \mu)(\overline{\star}, \overline{x}, \overline{y}, \overline{z}) \cong\ _{\mathcal{C}} \langle x, D^{\textnormal{opp}(1)}(\overline{y \otimes_{\mathcal{C}} z}) \rangle.
	 \] 
Thus \eqref{invariance} provides a natural isomorphism
	 \begin{equation*}
	 	_{\mathcal{C}} \langle x \otimes_{\mathcal{C}} y, D^{\textnormal{opp}(1)}(\overline{z}) \rangle \cong\  _{\mathcal{C}} \langle x, D^{\textnormal{opp}(1)}(\overline{y \otimes_{\mathcal{C}} z}) \rangle.
	 	 	 	 \end{equation*}
This shows that the monoidal category $(\mathcal{C}, \otimes_{\mathcal{C}}, \eta_{\mathcal{C}}(\star), \alpha_{\mathcal{C}}, \lambda_{\mathcal{C}}, \varrho_{\mathcal{C}})$ is $\star$-autonomous with duality $D: \mathcal{C} \rightarrow \mathcal{C}^{\textnormal{opp}(1)}$.
\end{enumerate}
\end{proof}

\begin{theorem}
\label{autonominducesfrobenius}
	A $\star$-autonomous category induces a $\dag$-Frobenius pseudomonoid  in $\Prof$.
\end{theorem}
\begin{proof}
\begin{enumerate}
	\item In our case it is convenient to think of a $\star$-autonomous category as in $2.$ of \autoref{starautonomequivalent}, i.e. as  a monoidal category $(\mathcal{C}, \otimes_{\mathcal{C}}, \unit_{\mathcal{C}}, \alpha, \lambda, \varrho)$ together with an equivalence $D: \mathcal{C} \rightarrow \mathcal{C}^{\textnormal{opp}(1)}$ and a natural isomorphism
	\begin{equation}
	\label{starautonom}
		_{\mathcal{C}} \langle x \otimes_{\mathcal{C}} y, D^{\textnormal{opp}(1)}(\overline{z}) \rangle \cong\ _{\mathcal{C}} \langle x, D^{\textnormal{opp}(1)}(\overline{y \otimes_{\mathcal{C}} z}) \rangle, \quad x,y, z \in \mathcal{C}.
	\end{equation}
	Since every equivalence can be refined to an adjoint equivalence, we can assume that $D$ is an adjoint equivalence, and we will use this fact frequently in our calculations. 
	\item The monoidal unit $\unit_{\mathcal{C}}$ yields an unique functor $\eta_{\mathcal{C}}: \unit_{\Cat} \rightarrow \mathcal{C}$ from the category with one object to $\mathcal{C}$, such that $\eta_{\mathcal{C}}(\star) = \unit_{\mathcal{C}}$. From \autoref{rightadjointrep} it follows that the profunctors 
	\begin{equation}
	\label{mustar22}
		\mu := (\otimes_{\mathcal{C}}^{\textnormal{opp}(1)})_{\star}: \mathcal{C}^{\textnormal{opp}(1)} \times  \mathcal{C}^{\textnormal{opp}(1)} \nrightarrow  \mathcal{C}^{\textnormal{opp}(1)}, \quad (x, \overline{y}, \overline{z}) \mapsto\ _{\mathcal{C}} \langle y  \otimes_{\mathcal{C}} z, x \rangle
	\end{equation}
	and 
	\begin{equation}
	\label{etastar22}
		\eta := (\eta^{\textnormal{opp}(1)}_{\mathcal{C}})_{\star}:  \unit_{\Prof} \nrightarrow \mathcal{C}^{\textnormal{opp}(1)}, \quad (x, \star) \mapsto\ _{\mathcal{C}} \langle \unit_{\mathcal{C}}, x \rangle.
	\end{equation}
	together with 
	\begin{equation}
	\label{alphastar22}
		(\alpha^{\textnormal{opp}(1)})_{\star}, (\lambda^{\textnormal{opp}(1)})_{\star}, (\varrho^{\textnormal{opp}(1)})_{\star}
	\end{equation}
yield a pseudomonoid $(\mathcal{C}^{\textnormal{opp}(1)}, \mu, \eta)$ in $\Prof$ such that $\mu$ and $\eta$ have right adjoints. 
	\item Equation \eqref{starautonom} imposes the definition of a form $\sigma: \mathcal{C}^{\textnormal{opp}(1)} \times \mathcal{C}^{\textnormal{opp}(1)} \nrightarrow \unit_{\Prof}$ by
	\begin{equation}
	\label{sigma22}
		\sigma(\overline{\star}, \overline{x}, \overline{y}) :=\ _{\mathcal{C}} \langle x, D^{\textnormal{opp}(1)}(\overline{y}) \rangle= (D^{\textnormal{opp}(1)})_{\star}(\overline{x}, \overline{y}), \quad x,y \in \mathcal{C} .
	\end{equation} 
	Similarly one defines a profunctor $\tilde{\sigma}: \unit_{\Prof} \nrightarrow \mathcal{C}^{\textnormal{opp}(1)} \otimes \mathcal{C}^{\textnormal{opp}(1)}$ by
	\[
	\tilde{\sigma}(x, y, \star) :=\ _{\mathcal{C}} \langle D^{\textnormal{opp}(1)}(\overline{x}), y  \rangle = ((D^{-1})^{\textnormal{opp}(1)})_{\star}(x, y), \quad x,y \in \mathcal{C}.
	\]
	A short calculation shows that we can identify $\sigma$ with $(D^{\textnormal{opp}(1)})_{\star}$, and $\tilde{\sigma}$ with $((D^{-1})^{\textnormal{opp}(1)})_{\star}$ in the sense of \autoref{biexactpairinginvertible}. Moreover, since $(D^{\textnormal{opp}(1)})_{\star}$ and $((D^{-1})^{\textnormal{opp}(1)})_{\star}$ are mutually inverse, $\sigma$ and $\tilde{\sigma}$ witness the existence of an biexact pairing $\mathcal{C}^{\textnormal{opp}(1)} \dashv \mathcal{C}^{\textnormal{opp}(1)}$.
	 Indeed, explicitly, for $a,b \in \mathcal{C}$ we find
	\begin{align*}
		&(\sigma \otimes \textnormal{id}_{\mathcal{C}^{\textnormal{opp}(1)}}) \circ (\textnormal{id}_{\mathcal{C}^{\textnormal{opp}(1)}} \otimes \tilde{\sigma})(a,\overline{b}) \equiv (\sigma \otimes \textnormal{id}_{\mathcal{C}^{\textnormal{opp}(1)}}) \circ (\textnormal{id}_{\mathcal{C}^{\textnormal{opp}(1)}} \otimes \tilde{\sigma})(\overline{\star}, a,\overline{b}, \star) \\
		&\overset{\eqref{profunctorscompo}}{=} \int^{(\overline{x}, \overline{y}, \overline{z}) \in (\mathcal{C}^{\textnormal{opp}(1)})^3}\ (\textnormal{id}_{\mathcal{C}^{\textnormal{opp}(1)}} \otimes \tilde{\sigma})(x,y,z,\overline{b}, \star) \times (\sigma \otimes \textnormal{id}_{\mathcal{C}^{\textnormal{opp}(1)}})(\overline{\star}, a, \overline{x}, \overline{y}, \overline{z}) \\
		&= \int^{(\overline{x}, \overline{y}, \overline{z}) \in (\mathcal{C}^{\textnormal{opp}(1)})^3}\ \textnormal{id}_{\mathcal{C}^{\textnormal{opp}(1)}}(x,\overline{b}) \otimes \tilde{\sigma}(y,z, \star) \times \sigma(\overline{\star}, \overline{x}, \overline{y}) \otimes \textnormal{id}_{\mathcal{C}^{\textnormal{opp}(1)}}(a, \overline{z}) \\
		&= \int^{(\overline{x}, \overline{y}, \overline{z}) \in (\mathcal{C}^{\textnormal{opp}(1)})^3}\ _{\mathcal{C}} \langle b, x \rangle \times\ _{\mathcal{C}} \langle D^{\textnormal{opp}(1)}(\overline{y}), z \rangle \times\ _{\mathcal{C}} \langle x, D^{\textnormal{opp}(1)}(\overline{y}) \rangle \times\  _{\mathcal{C}} \langle z, a \rangle\\
		&\overset{\eqref{coendrechnung}}{\cong}\ _{\mathcal{C}}\langle b, a \rangle = \textnormal{id}_{\mathcal{C}^{\textnormal{opp}(1)}}(a, \overline{b}).
	\end{align*}
	The other zig-zag identity \eqref{zigzagidentities2} follows similarly. 
	\item 
	It remains to show the existence of an invertible $2$-cell
	\begin{equation}
	\label{missingiso}
		\sigma \circ (\mu \otimes \textnormal{id}_{\mathcal{C}^{\textnormal{opp}(1)}}) \cong \sigma \circ (\textnormal{id}_{\mathcal{C}^{\textnormal{opp}(1)}} \otimes \mu)
	\end{equation}
	as in \eqref{form}. For $x,y,z \in \mathcal{C}$ we compute
	\begin{align*}
		&\sigma \circ (\mu \otimes \textnormal{id}_{\mathcal{C}^{\textnormal{opp}(1)}})(\overline{\star}, \overline{x}, \overline{y}, \overline{z})
	\overset{\eqref{profunctorscompo}}{=} \int^{(\overline{a}, \overline{b}) \in (\mathcal{C}^{\textnormal{opp}(1)})^2} (\mu \otimes \textnormal{id}_{\mathcal{C}^{\textnormal{opp}(1)}})(a,b, \overline{x}, \overline{y}, \overline{z}) \times \sigma(\overline{\star}, \overline{a}, \overline{b}) \\
	&= \int^{(\overline{a}, \overline{b}) \in (\mathcal{C}^{\textnormal{opp}(1)})^2} \mu(a,\overline{x}, \overline{y}) \times \textnormal{id}_{\mathcal{C}^{\textnormal{opp}(1)}}(b, \overline{z}) \times \sigma(\overline{\star}, \overline{a}, \overline{b})\\
	&= \int^{(\overline{a}, \overline{b}) \in (\mathcal{C}^{\textnormal{opp}(1)})^2}\ _{\mathcal{C}^{\textnormal{opp}(1)}} \langle \overline{a}, \overline{x \otimes_{\mathcal{C}} y} \rangle \times\ _{\mathcal{C}^{\textnormal{opp}(1)}} \langle \overline{b}, \overline{z} \rangle \times\ _{\mathcal{C}} \langle a, D^{\textnormal{opp}(1)}(\overline{b}) \rangle \\
	&= \int^{(\overline{a}, \overline{b}) \in (\mathcal{C}^{\textnormal{opp}(1)})^2}\ _{\mathcal{C}} \langle x \otimes_{\mathcal{C}} y, a \rangle \times\ _{\mathcal{C}} \langle z, b \rangle \times\ _{\mathcal{C}} \langle a, D^{\textnormal{opp}(1)}(\overline{b}) \rangle \\
	&\overset{\eqref{coendrechnung}}{\cong} \int^{\overline{b} \in \mathcal{C}^{\textnormal{opp}(1)}}\ _{\mathcal{C}} \langle x \otimes_{\mathcal{C}} y, D^{\textnormal{opp}(1)}(\overline{b}) \rangle \times\ _{\mathcal{C}} \langle z, b \rangle \\
	&\overset{D^{-1}}{\cong} \int^{\overline{b} \in \mathcal{C}^{\textnormal{opp}(1)}}\ _{\mathcal{C}} \langle b, D^{-1}(\overline{x \otimes_{\mathcal{C}} y})\rangle \times\ _{\mathcal{C}} \langle z, b \rangle \\
	&\overset{\eqref{coendrechnung}}{\cong}\ _{\mathcal{C}} \langle z, D^{-1}(\overline{x \otimes_{\mathcal{C}} y}) \rangle.
	\end{align*} 
	Analogously one shows that
	\[
	\sigma \circ (\textnormal{id}_{\mathcal{C}^{\textnormal{opp}(1)}} \otimes \mu)(\overline{\star}, \overline{x}, \overline{y}, \overline{z}) \cong\ _{\mathcal{C}} \langle y \otimes_{\mathcal{C}} z, D^{-1}(\overline{x}) \rangle , \quad x,y, z \in \mathcal{C}.
	\]
	From \eqref{starautonom} and the adjointness of $D$ one obtains \begin{equation*}
		_{\mathcal{C}}\langle z, D^{-1}(\overline{x \otimes_{\mathcal{C}} y}) \rangle \cong\ _{\mathcal{C}} \langle y \otimes_{\mathcal{C}} z, D^{-1}(\overline{x}) \rangle, \quad x,y, z \in \mathcal{C},
	\end{equation*}
	which then can be used to compose the $2$-cell \eqref{missingiso}. 
	
	To sum up, we showed that $(\mathcal{C}^{\textnormal{opp}(1)}, \mu, \eta, \sigma)$ is a $\dag$-Frobenius pseudomonoid in $\Prof$.
\end{enumerate}	
	\end{proof}

Comparing \eqref{muetaconstruction}, \eqref{alphacstar2}, \eqref{unitorstar2}, \eqref{sigmarepr} with \eqref{mustar22}, \eqref{etastar22}, \eqref{alphastar22}, \eqref{sigma22}, respectively, shows that the constructions of \autoref{frobeniusinducesstarautonom} and \autoref{autonominducesfrobenius} are inverse to each other up to natural isomorphisms of functors.

\chapter{Applications}
\section{Exact endofunctors}

In this section we follow up on \cite[Remark 3.17]{fuchs2016eilenberg}, in which the existence of a $\star$-autonomous structure on the category of exact endofunctors of some finite linear category was proven. Precisely, we show that the $\star$-autonomous structure, under certain conditions, is induced by rigidity, and thus, the second tensor product (cf. \eqref{secondtensorproductintro}) is isomorphic to the first tensor product, which is given as the composition of endofunctors.

\subsection{Finite categories}
In this subsection let $\mathbbm{k}$ be a fixed algebraically closed field. A category is \textit{additive}, if it admits finitary products $\prod$ (or equivalently finitary coproducts $\coprod$), and every morphism space has the structure of an abelian group, such that the composition is bilinear. 
A functor $F: \mathcal{C} \rightarrow \mathcal{D}$ between additive categories $\mathcal{C}, \mathcal{D}$ is additive, if the induced maps
\begin{equation}
\label{additivefunctorinducedmaps}
\textnormal{Hom}_{\mathcal{C}}(x,y) \rightarrow \textnormal{Hom}_{\mathcal{D}}(F(x), F(y))
\end{equation}
are homomorphism of abelian groups. This yields, in particular, that the induced morphisms $F(x) \coprod F(y) \rightarrow F(x \prod y)$ are isomorphisms for all $x,y \in \mathcal{C}$.

Moreover, an additive category $\mathcal{C}$ is said to be $\mathbbm{k}$\textit{-linear}, if all morphism spaces have the structure of $\mathbbm{k}$-vector spaces, such that the composition is $\mathbbm{k}$-linear. A functor $F$ between $\mathbbm{k}$-linear categories $\mathcal{C}, \mathcal{D}$ is $\mathbbm{k}$-linear, if the induced maps \eqref{additivefunctorinducedmaps} are $\mathbbm{k}$-linear for all $x,y \in \mathcal{C}$. 
The main objects of this section are special $\mathbbm{k}$-linear categories: \textit{finite} categories.

\begin{definition}[Finite linear category]
	A $\mathbbm{k}$-linear abelian category $\mathcal{C}$ is \textit{finite} (\textit{linear}), if
	\begin{itemize}
		\item All morphism spaces of $\mathcal{C}$ are finite-dimensional;
		\item Every object of $\mathcal{C}$ has finite length;
		\item Every simple object of $\mathcal{C}$ has a projective cover;
		\item There are finitely many isomorphism classes of simple objects.
	\end{itemize}
\end{definition}

For an explanation of the occurring notions we refer to \cite{etingof2015tensor}. Examples of finite categories are given by the categories $A \modl$ of finite-dimensional representations of finite-dimensional $\mathbbm{k}$-algebras $A$.
In fact, the following result shows that these categories form \textit{generic} examples of finite categories. 

\begin{lemma}
\label{finitelinear}
A $\mathbbm{k}$-linear abelian category is finite if and only if it is equivalent as a linear category to the category $A\modl$ of finite-dimensional (left or right) modules over some finite-dimensional $\mathbbm{k}$-algebra $A$.
\end{lemma}

 Two finite-dimensional $\mathbbm{k}$-algebras $A$ and $B$ are \textit{Morita equivalent}, if their representation categories $A\modl$ and $B\modl$ are equivalent as $\mathbbm{k}$-linear categories. For every finite category \autoref{finitelinear} yields an algebra $A$ up to Morita equivalence. 
 
In this chapter we will extensively use that every finite linear category $\mathcal{C}$ is a module category over $\vect$ (cf. \cite[p.8]{fuchs2016eilenberg}). Explicitly, for $c \in \mathcal{C}$ and $v \in \vect$, the element $c \otimes v \cong v \otimes c \in \mathcal{C}$ is defined such that
\begin{equation}
\label{finitemodulecat}
	_{\mathcal{C}} \langle c', v \otimes c \rangle \cong v \otimes_{\mathbbm{k}}\ _{\mathcal{C}} \langle c', c \rangle \quad \textnormal{and} \quad _{\mathcal{C}} \langle v \otimes c, c' \rangle \cong v^{\star} \otimes_{\mathbbm{k}}\  _{\mathcal{C}} \langle c, c' \rangle
\end{equation}
for all $c' \in \mathcal{C}$.
 
Before giving the next Theorem, we need to introduce some more definitions. 

First, note that the underlying vector space of some $\mathbbm{k}$-algebra $A$ is a finite-dimensional bimodule over itself by multiplication. Moreover, in general, for every $(A,B)$-bimodule $V$, the linear dual $V^{\star}$ is a $(B,A)$-bimodule. Thus, in particular, the linear dual $A^{\star}$ is again an $A$-bimodule. 

Secondly, a $\mathbbm{k}$-linear functor is \textit{left exact}, if it preserves short left exact sequences, or equivalently, preserves finite limits (cf. \cite{mac2013categories}). We denote the collection of left exact functors $F: \mathcal{C} \rightarrow \mathcal{D}$ between finite linear categories $\mathcal{C}, \mathcal{D}$ as $\Lex(\mathcal{C}, \mathcal{D})$. Analogous one defines \textit{right exact} functors and denotes the collection of them as $\Rex(\mathcal{C}, \mathcal{D})$. If $F: A\modl \rightarrow B\modl$ is a $\mathbbm{k}$-linear functor and $V \in A\modl$ some left $A$-module, one can show that $F(V)$ is not only a left $B$-module, but in fact can be equipped with the structure of a $(B,A)$-bimodule (cf. \cite{fuchs2016eilenberg}).

The next theorem is known as a variant of the \textit{Eilenberg-Watts theorem} (cf. \cite{eilenberg1960abstract}, \cite{watts1960intrinsic}). 
 
\begin{theorem}
\label{eilenbergwattssummarize}
	Let $F,G: \mathcal{C} \rightarrow \mathcal{D}$ be functors between finite linear functors. The following statements are equivalent:
	\begin{itemize}
		\item $F \in \Lex(\mathcal{C}, \mathcal{D})$. 
		\item $F$ admits a left adjoint.
	\end{itemize}
	Likewise, the following are equivalent:
	\begin{itemize}
		\item $G \in \Rex(\mathcal{C}, \mathcal{D})$.
		\item $G$ admits a right adjoint
	\end{itemize}
\end{theorem}

\subsection{End and coend}

\begin{definition}[End and Coend]
\label{defend}
	The \textit{end} of a functor $F: \mathcal{C}^{\textnormal{opp}(1)} \times \mathcal{C} \rightarrow \mathcal{D}$ is given by an object $\int_{c \in \mathcal{C}} F(\bar{c}, c) \in \mathcal{D}$ and morphisms $( \omega_c: \int_{c \in \mathcal{C}} F(\bar{c}, c) \rightarrow F(\bar{c},c) )_{c \in \mathcal{C}}$ in $\mathcal{D}$, satisfying $F(\overline{f}, d) \circ \omega_d = F(\overline{c},f) \circ \omega_c$
	 for all $f: c \rightarrow d$, and universal in the following way:
\begin{quote}
	Let $x \in \mathcal{D}$ and $\lbrace \beta_c: x \rightarrow F(\bar{c},c) \rbrace_{c \in \mathcal{C}}$ satisfy $F(\overline{f}, d) \circ \beta_d = F(\overline{c},f) \circ \beta_c$ for all $f: c \rightarrow d$. Then there exists an unique morphism $h: x \rightarrow \int_{c \in \mathcal{C}} F(\bar{c}, c) \in \mathcal{D}$ in $\mathcal{D}$, such that $\beta_c = \omega_c \circ h$ for all $c \in \mathcal{C}$,
\begin{equation*}
	 \begin{tikzcd}[row sep=large, column sep = large]
	 	x \arrow[bend left]{rdr}{\beta_d} \arrow[dashed]{dr}{\exists ! h} \arrow[bend right]{ddr}[left]{\beta_c} \\
	 	& \int_{c \in \mathcal{C}} F(\bar{c}, c) \arrow{d}[left]{\omega_c} \arrow{r}{\omega_d} & F(\overline{d}, d) \arrow{d}{F(\overline{f}, d)} \\
	 	& F(\overline{c},c) \arrow{r}[below]{F(\overline{c},f)} & F(\overline{c}, d).
	 \end{tikzcd}
\end{equation*}
\end{quote}
	The \textit{coend} of $F: \mathcal{C}^{\textnormal{opp}(1)} \times \mathcal{C} \rightarrow \mathcal{D}$ is given by an object $\int^{c \in \mathcal{C}} F(\bar{c}, c) \in \mathcal{D}$ and morphisms $(\omega_c: F(\bar{c},c) \rightarrow \int^{c \in \mathcal{C}} F(\bar{c}, c))_{c \in \mathcal{C}}$ in $\mathcal{D}$, universal in a dual way.
\end{definition}

The universal property implies that end and coend, respectively, if existent, are unique up to unique isomorphism.

Ends and coends enjoy a variety of interesting properties. In particular the next two results (cf. \cite{mac2013categories}) turn out to be helpful for our purposes.

\begin{lemma}
Let $\mathcal{C}$ be an arbitrary category and $\mathcal{D}$ be a linear category.
	\begin{itemize}
		\item  Let $F: \mathcal{C}^{opp} \times \mathcal{C} \rightarrow \mathcal{D}$ be a functor whose end exists. Then there is a natural isomorphism 
			\begin{equation}
			\label{endcontineq}
				_{\mathcal{D}}\langle -, \int_{c \in \mathcal{C}} F(\bar{c},c) \rangle \cong \int_{c \in \mathcal{C}}\ _{\mathcal{D}}\langle -, F(\bar{c},c) \rangle.
			\end{equation}
				\item Let $F: \mathcal{C}^{opp} \times \mathcal{C} \rightarrow \mathcal{D}$ be a functor whose coend exists. Then there is a natural isomorphism 
			\begin{equation}
						\label{coendcontineq}
				_{\mathcal{D}}\langle \int^{c \in \mathcal{C}} F(\bar{c},c), - \rangle \cong \int_{c \in \mathcal{C}}\ _{\mathcal{D}}\langle F(\bar{c},c), - \rangle.
			\end{equation}
	\end{itemize}
\end{lemma}

\begin{lemma}
	Let $\mathcal{C}$ and $\mathcal{D}$ be arbitrary categories. For any pair of functors $F,G: \mathcal{C} \rightarrow \mathcal{D}$, there is a bijection
\begin{equation}
\label{natcoend}
	\textnormal{Nat}(F,G) \cong \int_{c \in \mathcal{C}}\ \textnormal{Hom}_{\mathcal{D}}(F(c),G(c)),
\end{equation}
where $\textnormal{Nat}(F,G)$ denotes the collection of natural transformations from $F$ to $G$.
\end{lemma}

The following is an analogue of the classical tensor product of vector spaces for the case of finite categories.

\begin{definition}[Deligne product]
	Let $\mathcal{C}$ and $\mathcal{D}$ be finite linear categories. The $\textit{Deligne product}$, $\mathcal{C} \boxtimes \mathcal{D}$, is a finite linear category together with a bifunctor $\boxtimes: \mathcal{C} \times \mathcal{D} \rightarrow \mathcal{C} \boxtimes \mathcal{D}$, right exact in both variables, such that the following universal property holds: 
\begin{quote}
	Let $\mathcal{A}$ be a finite linear category together with a bifunctor $F: \mathcal{C} \times \mathcal{D} \rightarrow \mathcal{A}$, right exact in both variables, then there exists a unique right exact functor $\tilde{F}: \mathcal{C} \boxtimes \mathcal{D} \rightarrow \mathcal{A}$ such that $\tilde{F} \circ \boxtimes = F$.
\end{quote}
\end{definition}

\begin{definition}[Eilenberg-Watts functors]
	Let $\mathcal{A}$ and $\mathcal{B}$ be finite linear categories. The (abstract) \textit{Eilenberg-Watts functors} for $\mathcal{A}$ and $\mathcal{B}$ are the following four linear functors: 
\begin{alignat*}{2}
		&\Phi^{\textnormal{l}} \equiv \Phi^{\textnormal{l}}_{\mathcal{A}, \mathcal{B}}:& \qquad \mathcal{A}^{\textnormal{opp}} \boxtimes \mathcal{B} 	&\rightarrow  \mathcal{L}ex(\mathcal{A}, \mathcal{B})  \\
 & & \quad \bar{a} \boxtimes b &\mapsto b \otimes_{\mathbbm{k}}\ _{\mathcal{A}}\langle a, - \rangle,\\
	&\Psi^{\textnormal{l}} \equiv \Psi^{\textnormal{l}}_{\mathcal{A}, \mathcal{B}}:&  \mathcal{L}ex(\mathcal{A}, \mathcal{B}) &\rightarrow \mathcal{A}^{\textnormal{opp}} \boxtimes \mathcal{B} \\
& & \quad F &\mapsto \int^{a \in \mathcal{A}} \bar{a} \boxtimes F(a),\\
	&\Phi^{\textnormal{r}} \equiv \Phi^{\textnormal{r}}_{\mathcal{A}, \mathcal{B}}:& \mathcal{A}^{\textnormal{opp}} \boxtimes \mathcal{B} &\rightarrow \mathcal{R}ex(\mathcal{A}, \mathcal{B})  \\
& &\bar{a} \boxtimes b &\mapsto b \otimes_{\mathbbm{k}}\ _{\mathcal{A}} \langle -, a  \rangle^{\star}, \\
	& \Psi^{\textnormal{r}} \equiv \Psi^{\textnormal{r}}_{\mathcal{A}, \mathcal{B}}:& \mathcal{R}ex(\mathcal{A}, \mathcal{B}) &\rightarrow \mathcal{A}^{\textnormal{opp}} \boxtimes \mathcal{B}\\
& &G &\mapsto \int_{a \in \mathcal{A}} \bar{a} \boxtimes G(a).
	\end{alignat*}
	\end{definition}

The following result (cf. \cite[Theor. 3.2]{fuchs2016eilenberg}) is known as \textit{Categorical Eilenberg-Watts theorem}.

\begin{theorem}
\label{categoricaleilenberg}
For any pair of finite linear categories $\mathcal{A}$ and $\mathcal{B}$ the functors in the triangle
	\begin{equation*}
		\begin{tikzcd}[row sep=large, column sep = large]
			& \mathcal{A}^{\textnormal{opp}(1)} \boxtimes \mathcal{B} \arrow{dl}{\Phi^{\textnormal{l}}} \arrow{dr}{\Phi^{\textnormal{r}}} \\
			\Lex(\mathcal{A}, \mathcal{B}) \arrow{rr}{\Gamma^{\textnormal{rl}} := \Phi^{\textnormal{r}} \circ \Psi^{\textnormal{l}}} \arrow{ur}{\Psi^{\textnormal{l}}}  && \Rex(\mathcal{A}, \mathcal{B}) \arrow{ll}{\Gamma^{\textnormal{lr}} := \Phi^{\textnormal{l}} \circ \Psi^{\textnormal{r}}} \arrow{ul}{\Psi^{\textnormal{r}}}  
		\end{tikzcd}
	\end{equation*}
	constitute quasi-inverse pairs of adjoint equivalences
	\[
	\Lex(\mathcal{A}, \mathcal{B}) \cong \mathcal{A}^{\textnormal{opp}(1)} \boxtimes \mathcal{B} \cong \Rex(\mathcal{A}, \mathcal{B}).
	\]
\end{theorem}

Using \eqref{endcontineq}, \eqref{coendcontineq} one can show the following identities for $\Gamma^{\textnormal{rl}}$ and $\Gamma^{\textnormal{rl}}$ (cf. \cite{fuchs2016eilenberg}). We will use them in particular in the proof of \autoref{gammamonoidal}.

\begin{lemma}
Let $\mathcal{A}$ and $\mathcal{B}$ be finite linear categories. For $F \in \Lex(\mathcal{A}, \mathcal{B})$ and $G \in \Rex(\mathcal{A}, \mathcal{B})$ it holds
	\begin{equation}
\label{gammaexplicitend}
	 \Gamma^{\textnormal{rl}}(F) \cong  \int^{a \in \mathcal{A}}\ _{\mathcal{A}} \langle -, a  \rangle^{\star} \otimes F(a)
\quad \textnormal{and} \quad 	\Gamma^{\textnormal{lr}}(G) \cong \int_{a \in \mathcal{A}}\ _{\mathcal{A}} \langle a, - \rangle \otimes G(a).
\end{equation}
\end{lemma}

Our own next result can be interpreted as an indication to which extend $\Gamma^{\textnormal{rl}}$ and $\Gamma^{\textnormal{lr}}$ fail to be monoidal. Note that the domain of the functors $\Gamma^{\textnormal{rl}}$ and $\Gamma^{\textnormal{lr}}$ can in fact be extended to the category of \textit{all} endofunctors (cf. \cite[p.24]{fuchs2016eilenberg}). We postpone the proof to the Appendix.

\begin{lemma}
\label{gammamonoidal}
	Let $\mathcal{X}$ be a finite linear category and $F,G \in \mathcal{L}ex(\mathcal{X}, \mathcal{X})$ be two left exact functors. Then 
	\begin{equation}
	\label{gammarlmonoidal}
		\Gamma^{\textnormal{rl}}(F) \circ \Gamma^{\textnormal{rl}}(G) \cong \Gamma^{\textnormal{rl}}(\Gamma^{\textnormal{rl}}(F) \circ G).
	\end{equation}
	On the other hand, let $F,G \in \mathcal{R}ex(\mathcal{X}, \mathcal{X})$ be two right exact functors. Then 
		\begin{equation}
		\label{gammalrmonoidal}
		\Gamma^{\textnormal{lr}}(F) \circ \Gamma^{\textnormal{lr}}(G) \cong \Gamma^{\textnormal{lr}}(\Gamma^{\textnormal{lr}}(F) \circ G).
	\end{equation}
\end{lemma}

\subsection{Nakayama functors}
Let $A$ be a finite-dimensional $\mathbbm{k}$-algebra.
For the sake of simplicity, in the following we abbreviate 
\[ _{A} \langle -,- \rangle := \textnormal{Hom}_{A\modl}(-,-), \quad \textnormal{and} \quad  _{A^{\textnormal{opp}}} \langle -,- \rangle := \textnormal{Hom}_{\modr A}(-,-) \] for the space of $A$-linear maps between finite dimensional left $A$-modules and right $A$-modules, respectively. The algebra $A$ can be considered as a finite-dimensional bimodule $_A A_A$ over itself by multiplication. The right exact contravariant \textit{transposition functor} $(-)^t$ is defined as
\begin{equation}
	(-)^t :=\ _A \langle -,\ _A A_A \rangle : A\modl \rightarrow \modr A
\end{equation}
with action $(f.a)(m) := f(a.m)$ for $f \in (_A M)^t, a \in A, m \in M$. The left exact contravariant transposition functor is defined as
\begin{equation}
	(-)^{t^{\textnormal{opp}}} :=\ _{A^{\textnormal{opp}}} \langle -,\ _A A_A \rangle : \modr A \rightarrow A \modl
\end{equation}
with action $(a.f)(m) := f(m.a)$ for $f \in (M_A)^{t^{\textnormal{opp}}}, a \in A, m \in M$.
Combined with the exact contravariant duality $(-)^{\star}$, we obtain the right exact covariant \textit{Nakayama functor}
\begin{equation}
\label{nakayama}
	\textnormal{N}^{\textnormal{r}}_A := (-)^{\star} \circ (-)^t =\ _A \langle -, _A A_A \rangle^{\star} : A\modl \rightarrow A \modl
\end{equation}
and the left exact covariant \textit{inverse Nakayama functor}
 \begin{equation}
 \label{nakayamainverse}
	\textnormal{N}^{\textnormal{l}}_A := (-)^{t^{\textnormal{opp}}} \circ (-)^{\star} = \ _{A^{\textnormal{opp}}} \langle (-)^{\star}, _A A_A \rangle : A \modl \rightarrow A \modl.
\end{equation}

Restricted to the subcategory of finite-dimensional projective $A$-modules and injective $A$-modules, respectively, the Nakayama functors provide a quasi-inverse equivalence $A\mathsf{-Projmod} \overset{\cong}{\longleftrightarrow} A\mathsf{-Injmod}$ (cf. \cite[Sec. 3]{ivanov2012nakayama}).

An application of the classical Eilenberg-Watts Theorem is given in the following result (cf. \cite[Proposition 3.1]{ivanov2012nakayama}, \cite[Lemma III.2.9]{assem2006elements}).

\begin{proposition}
\label{nakayamaisomorphisms}
	There exist natural isomorphisms
	\begin{equation*}
		\textnormal{N}^{\textnormal{r}}_A \cong (_A A_A)^{\star} \otimes_A - \qquad \textnormal{and} \qquad \textnormal{N}^{\textnormal{l}}_A \cong\ _A\langle (_A A_A)^{\star}, - \rangle.
	\end{equation*}
\end{proposition}

Surprisingly, this representation allows us to relate the equivalences $\Gamma^{\textnormal{lr}}$ and $\Gamma^{\textnormal{rl}}$ defined in \autoref{categoricaleilenberg} to the Nakyama functors \eqref{nakayama}, \eqref{nakayamainverse}.
\begin{lemma}
	\label{abstractnakayamaid}
For $\mathcal{A} \cong A\textnormal{-mod}$ a finite linear category there are natural isomorphism 	
\begin{equation}
		\Gamma^{\textnormal{rl}}(\textnormal{id}_{\mathcal{A}}) \cong (_A A_A)^{\star} \otimes_A -
\quad \textnormal{and} \quad
		\Gamma^{\textnormal{lr}}(\textnormal{id}_{\mathcal{A}}) \cong\ _A\langle (_A A_A)^{\star}, - \rangle
	\end{equation}
\end{lemma}

In view of \autoref{abstractnakayamaid} and \autoref{nakayamaisomorphisms} it is reasonable to give the following definition.

\begin{definition}[Nakayama functor]
\label{generalnakayamafunctor}
	The \textit{Nakayama functor} of a finite linear category $\mathcal{X}$ is the endofunctor 
	\begin{equation}
		\textnormal{N}^{\textnormal{r}}_{\mathcal{X}} := \Gamma^{\textnormal{rl}}(\textnormal{id}_{\mathcal{X}}) = \int^{x \in \mathcal{X}} x \otimes_{\mathbbm{k}}\ _{\mathcal{X}} \langle -, x  \rangle^{\star} \in \mathcal{R}ex(\mathcal{X}, \mathcal{X}).
	\end{equation}
	The \textit{left exact analogue of the Nakayama functor} of $\mathcal{X}$ is the functor
	\begin{equation}
		\textnormal{N}^{\textnormal{l}}_{\mathcal{X}} := \Gamma^{\textnormal{lr}}(\textnormal{id}_{\mathcal{X}}) = \int_{x \in \mathcal{X}} x \otimes_{\mathbbm{k}}\ _{\mathcal{X}}\langle x, - \rangle \in \mathcal{L}ex(\mathcal{X}, \mathcal{X}).
	\end{equation}
\end{definition}

It is an interesting question, in which cases the above Nakayama functors are quasi-inverses to each other. In the classical case \eqref{nakayama}, \eqref{nakayamainverse}, we've already seen that this is the case on the subcategories of finite-dimensional projective $A$-modules and injective $A$-modules. In fact, it is sufficient to require that the finite-dimensional algebra $A$ is \textit{self-injective}, i.e. injective as a left, or right, module over itself (cf. e.g. \cite[Prop. IV.3.1]{auslander1997representation}). Moreover, it turns out that if $A$ has the structure of a Frobenius algebra, it is in particular also self-injective. More precisely, if $A$ is Frobenius with Frobenius form $\sigma: A \otimes A \rightarrow \mathbbm{k}$, the \textit{Nakayama automorphism} is the automorphism $\nu: A \rightarrow A$, such that $\sigma(a,\nu(b)) = \sigma(b,a)$ for all $a,b \in A$. The Nakayama functors $\textnormal{N}^{\textnormal{l}}_A, \textnormal{N}^{\textnormal{r}}_A$ then amount to twisting the left, or right, action of $A$ by the Nakayama automorphism. In particular, if $A$ is symmetric, the Nakayama automorphism is given by the identity, and thus, the Nakayama functors are isomorphic to the identity functor. 

Since being symmetric Frobenius is invariant under Morita equivalence, it seems reasonable to give the following definition (cf. \cite{fuchs2016eilenberg}) for finite linear categories.

\begin{definition}[Symmetric Frobenius category]
	A finite linear category is called \textit{symmetric Frobenius} if and only if it is equivalent  as a linear category to the category of modules over a symmetric Frobenius $\mathbbm{k}$-algebra.
\end{definition}

The following result is due to \cite[Prop. 3.24]{fuchs2016eilenberg}.

\begin{proposition}
Let $\mathcal{X}$ be a finite linear category. The Nakayama functors $\textnormal{N}^{\textnormal{r}}_{\mathcal{X}}$ and $\textnormal{N}^{\textnormal{l}}_{\mathcal{X}}$ are isomorphic to the identity functor,
$
\Gamma^{\textnormal{lr}}(\textnormal{id}_{\mathcal{X}}) \cong \textnormal{id}_{\mathcal{X}} \cong \Gamma^{\textnormal{rl}}(\textnormal{id}_{\mathcal{X}}),
$ if and only if $\mathcal{X}$ is symmetric Frobenius.
\end{proposition}

The following Corollary shows that $\mathcal{X}$ being symmetric Frobenius implies not only the invariance, up to natural isomorphism, of the identity functor $\textnormal{id}_{\mathcal{X}}$ under $\Gamma^{\textnormal{lr}}$ and $\Gamma^{\textnormal{rl}}$, but in fact the invariance of \textit{all} left and right exact endofunctors of $\mathcal{X}$, respectively.

\begin{corollary}
\label{symmetricfrobcor}
	Let $\mathcal{X}$ be a symmetric Frobenius finite linear category. For $F \in \mathcal{R}ex(\mathcal{X}, \mathcal{X})$ and $G \in \mathcal{L}ex(\mathcal{X}, \mathcal{X})$ there exist natural isomorphism
$		\Gamma^{\textnormal{lr}}(F) \cong F$ and $
		\Gamma^{\textnormal{rl}}(G) \cong G.
$
\end{corollary}
\begin{proof}
	By assumption the Nakayama functors are natural isomorphic to the identity functor, $\textnormal{N}^{\textnormal{l}}_{\mathcal{X}} =\Gamma^{\textnormal{lr}}(\textnormal{id}_{\mathcal{X}}) \cong \textnormal{id}_{\mathcal{X}} \cong \Gamma^{\textnormal{rl}}(\textnormal{id}_{\mathcal{X}}) = \textnormal{N}^{\textnormal{r}}_{\mathcal{X}}$. Thus, for $F \in \mathcal{R}ex(\mathcal{X}, \mathcal{X})$ and $G \in \mathcal{L}ex(\mathcal{X}, \mathcal{X})$, we compute, using \autoref{gammamonoidal},
\[
			\Gamma^{\textnormal{lr}}(F) = \Gamma^{\textnormal{lr}}(F) \circ \textnormal{id}_{\mathcal{X}} \cong \Gamma^{\textnormal{lr}}(F) \circ \Gamma^{\textnormal{lr}}(\textnormal{id}_{\mathcal{X}}) \overset{\eqref{gammalrmonoidal}}{\cong} \Gamma^{\textnormal{lr}}(\Gamma^{\textnormal{lr}}(F) \circ \textnormal{id}_{\mathcal{X}}) = \Gamma^{\textnormal{lr}}(\Gamma^{\textnormal{lr}}(F)),
\]
\[
		\Gamma^{\textnormal{rl}}(G) = \Gamma^{\textnormal{rl}}(G) \circ \textnormal{id}_{\mathcal{X}} \cong \Gamma^{\textnormal{rl}}(G) \circ \Gamma^{\textnormal{rl}}(\textnormal{id}_{\mathcal{X}}) \overset{\eqref{gammarlmonoidal}}{\cong} \Gamma^{\textnormal{rl}}(\Gamma^{\textnormal{rl}}(G) \circ \textnormal{id}_{\mathcal{X}}) = \Gamma^{\textnormal{rl}}(\Gamma^{\textnormal{rl}}(G)).
\]
	Since by \autoref{categoricaleilenberg}, $\Gamma^{\textnormal{rl}}$ and $\Gamma^{\textnormal{lr}}$ form an (adjoint) equivalence, composition with the respective inverse yields the claim.
\end{proof}

Note that this implies, in particular, $\mathcal{R}ex(\mathcal{X}, \mathcal{X})= \mathcal{L}ex(\mathcal{X}, \mathcal{X})$ for any symmetric Frobenius finite linear category $\mathcal{X}$.
\begin{lemma}
\label{rigidleftadjoint}
	Let $\mathcal{X}$ be a symmetric Frobenius finite linear category. Then $(\mathcal{R}ex(\mathcal{X}, \mathcal{X})= \mathcal{L}ex(\mathcal{X}, \mathcal{X}), \circ, \textnormal{id}_{\mathcal{X}})$ is rigid with respect to the duality given by taking the left or right adjoint. Furthermore, left and right adjoints are natural isomorphic, $(-)^{\textnormal{l.a.}} \cong (-)^{\textnormal{r.a.}}$.
\end{lemma}
\begin{proof}
Compare the definition of adjoint functors (e.g. \cite{mac2013categories}) with the definition of rigidity (cf. \autoref{rigid}) to see that duals in the monoidal category of endofunctors with composition as tensor product, coincide with adjoint functors. The existence of left and right adjoints is provided by the classical Eilenberg-Watts-Theorem (cf. \autoref{eilenbergwattssummarize}). Moreover,
	\cite[Corollary 3.9]{fuchs2016eilenberg} states that for any $F \in \mathcal{R}ex(\mathcal{X}, \mathcal{X}) = \mathcal{L}ex(\mathcal{X}, \mathcal{X})$ it holds  
$		\Gamma^{\textnormal{rl}}(F^{\textnormal{r.a.}}) = (\Gamma^{\textnormal{lr}}(F))^{\textnormal{l.a.}}.
$ Applying \autoref{symmetricfrobcor} thus yields the claim, $F^{\textnormal{r.a.}} \cong F^{\textnormal{l.a.}}.$
\end{proof}

Let $\mathcal{X}$ be an arbitrary finite linear category, not necessarily symmetric Frobenius. In \cite[Remark 3.17]{fuchs2016eilenberg} it was proposed that one can define a monoidal category $(\mathcal{R}ex(\mathcal{X}, \mathcal{X}), \circ$$, \textnormal{id}_{\mathcal{X}})$ with $\star$-autonomous structure witnessed by the dualizing object $N^{\textnormal{r}}_{\mathcal{X}}$. The corresponding duality functor is defined by $G \mapsto\Gamma^{\textnormal{rl}}(G^{\textnormal{r.a.}}) \cong (\Gamma^{\textnormal{lr}}(G))^{\textnormal{l.a.}}$ for $G \in \mathcal{R}ex(\mathcal{X}, \mathcal{X})$ (cf. \cite[Corollary 3.9]{fuchs2016eilenberg}). An interesting question is, whether the induced second tensor product (cf. \eqref{secondtensorproductintro}) is isomorphic to the first one, or not.

In the particular case that $\mathcal{X}$ is a symmetric Frobenius finite linear category, we can give an answer to that question. Indeed,  \autoref{symmetricfrobcor} implies that the corresponding duality of $(\mathcal{R}ex(\mathcal{X}, \mathcal{X}) = \mathcal{L}ex(\mathcal{X}, \mathcal{X}) , \mathcal{X}), \circ, \textnormal{id}_{\mathcal{X}})$ boils down to $
G \mapsto G^{\textnormal{r.a.}} \cong G^{\textnormal{l.a.}}
$. This, however, is precisely the rigid structure as defined in \autoref{rigidleftadjoint}. From \autoref{rigidmonoidalequivalence} it follows that the second tensor product of the $\star$-autonomous structure is isomorphic to the first one: the composition of endofunctors.

\section{Graded vector spaces}

In this Section we will introduce a family of $\star$-autonomous structures on the category of $G$-graded vector spaces and characterise weak ribbon structures on these categories.

\subsection{Quadratic forms}
\begin{definition}[Quadratic form]
\label{classicalquadraticformdef}
	Let $G$ be a finite group. A \textit{quadratic form} on $G$ with values in $\mathbbm{k}^{\times}$ is a map
	\[
	q: \quad G \rightarrow \mathbbm{k}^{\times},
	\]
	such that
	\[
	q(g) = q(g^{-1})
	\]
	and such that the function 
	\begin{align}
	\label{defassocbihom}
	\begin{split}
				 \beta_q: \quad G \times G &\rightarrow \mathbbm{k}^{\times} \\
		(g_1, g_2) &\mapsto q(g_1 g_2)q(g_1)^{-1}q(g_2)^{-1}
	\end{split}
	\end{align}
	is a bihomomorphism. The collection of quadratic forms on $G$ with values in $\mathbbm{k}^{\times}$ forms in a canonical way an abelian group, which we denote by $\QF$.
	We will sometimes use an equivalence relation $\sim$ on $\QF$ defined by
	\[ q_1 \sim q_2\ :\Leftrightarrow\ \exists f \in \textnormal{Aut}(G): q_1 = f^{\star}q_2.  \]
	for $q_1,q_2 \in \QF$.
\end{definition}

In the literature one frequently encounters a definition of quadratic forms containing the axiom $q(g^k) = q(g)^{k^2}$ for all $k \in \mathbb{N}$ (sometimes written in additive notation), instead of $q(g) = q(g^{-1})$. We later obtain the equivalence of both definitions as a consequence of \autoref{weakhochk}.

\begin{definition}[Group cohomology]
	Let $G$ be a finite group. A \textit{k-cochain} on $G$ with values in $\mathbbm{k}^{\times}$ is a function
	\[
	\kappa:\quad G^{k} \equiv G \times ... \times G \rightarrow  \mathbbm{k}^{\times},
	\]
	$\kappa$ is \textit{normalized}, if in the case that one of its arguments is the unit element $1_G \in G$, the value of $\kappa$ is $1 \in \mathbbm{k}^{\times}$. The collection of $k$-cochains on $G$ forms in a canonical way a group, which we denote by $C^{k}(G, \mathbbm{k}^{\times})$. The \textit{coboundary operator}
	\[ d = d_k: \quad C^{k}(G, \mathbbm{k}^{\times}) \rightarrow C^{k+1}(G,  \mathbbm{k}^{\times}) \]
	is defined by
	\begin{align*}
		d \kappa(g_1, ..., g_{k+1}) &= \kappa(g_2, g_3, ..., g_{k+1}) \kappa(g_1 g_2, g_3,...,g_{k+1})^{-1} \kappa(g_1, g_2 g_3, ..., g_{k+1}) \\
		&...\ \kappa(g_1, g_2, ..., g_k g_{k+1})^{(-1)^{k}} \kappa(g_1, g_2, ..., g_k)^{(-1)^{k+1}}.
	\end{align*}
	The elements of the subgroup \[
	 Z^{k}(G,\mathbbm{k}^{\times}) := \textnormal{ker}(d_k) \subseteq C^{k}(G, \mathbbm{k}^{\times})
	\]
	are called $k$\textit{-cocycles} and the elements of the subgroup
	\[
	B^{k+1}(G,\mathbbm{k}^{\times}) := \textnormal{im}(d_k) \subseteq C^{k+1}(G,  \mathbbm{k}^{\times})
	\]
	are called $(k+1)$\textit{-coboundaries}. The $k$\textit{-th cohomology group} $H^k(G,  \mathbbm{k}^{\times})$ of $G$ is defined as the quotient group
	\[
	H^k(G,  \mathbbm{k}^{\times}) :=  Z^{k}(G,\mathbbm{k}^{\times}) /  B^{k}(G,\mathbbm{k}^{\times}).
	\]
\end{definition}

\begin{definition}[Abelian group cohomology]
\label{abeliangroupcohomology}
	Let $G$ be a finite abelian group. An \textit{abelian 2-cochain} is an ordinary $2$-cochain. An \textit{abelian $3$-cochain} is a pair
	\[
		(\psi, \Omega)
	\]
	consisting of an ordinary $3$-cochain $\psi \in C^{3}(G, \mathbbm{k}^{\times})$ and an ordinary $2$-cochain $\Omega \in C^{2}(G, \mathbbm{k}^{\times})$. It is called \textit{normalized}, if both $\psi$ and $\Omega$ are normalized as ordinary cochains.
	An abelian $3$-cochain $(\psi, \Omega)$ is an \textit{abelian 3-cocycle}, if $d\psi = 1$ and 
	\begin{align}
	\label{cocycleequation}
	\begin{split}
		\psi(g_2, g_3, g_1)^{-1} \Omega(g_1, g_2 g_3) \psi(g_1, g_2, g_3)^{-1} &= \Omega(g_1, g_3) \psi(g_2, g_1, g_3)^{-1} \Omega(g_1, g_2), \\
		\psi(g_3, g_1, g_2) \Omega(g_1 g_2, g_3) \psi(g_1, g_2, g_3) &= \Omega(g_1, g_3) \psi(g_1, g_3, g_2) \Omega(g_2, g_3).
	\end{split}
	\end{align}
	The \textit{abelian coboundary} $d_{\textnormal{ab}}(\kappa)$ of an abelian $2$-cochain $\kappa \in C^{2}(G, \mathbbm{k}^{\times})$ is defined as
	\[
	d_{\textnormal{ab}}(\kappa) := (d\kappa, \kappa_{\textnormal{comm}}),
	\]
	where the \textit{commutator cocycle} $\kappa_{\textnormal{comm}}$ of $\kappa$ is given as
	\[
	\kappa_{\textnormal{comm}}(g_1,g_2) := \kappa(g_1, g_2) \kappa(g_2, g_1)^{-1}.
	\]
	The \textit{third abelian cohomology group} $H^{3}_{\textnormal{ab}}(G, \mathbbm{k}^{\times})$ of $G$ is defined as the quotient of the group $Z^3_{\textnormal{ab}}(G, \mathbbm{k} ^{\times})$ of normalized abelian $3$-cocycles by its subgroup $B^3_{\textnormal{ab}}(G, \mathbbm{k} ^{\times})$ consisting of abelian coboundaries of normalized abelian $2$-cochains. 
\end{definition}

We will use an equivalence relation $\sim$ on $H^3_{\textnormal{ab}}(G, \mathbbm{k} ^{\times})$ defined by 	
	\[ \lbrack (\psi_1, \Omega_1)\rbrack \sim \lbrack (\psi_2, \Omega_2)\rbrack  :\Leftrightarrow \exists f \in \textnormal{Aut}(G) :  \lbrack (\psi_1, \Omega_1) \rbrack
 = \lbrack (f^{\star}\psi_2, f^{\star}\Omega_2) \rbrack \]
 for $(\psi_1, \Omega_1),(\psi_2, \Omega_2) \in Z^3_{\textnormal{ab}}(G, \mathbbm{k}^{\times})$. Note that, if $(\psi_1, \Omega_1)$ and $(\psi_2, \Omega_2)$ are equivalent in $H^{3}_{\textnormal{ab}}(G, \mathbbm{k} ^{\times})$, they are also equivalent with respect to $\sim$, witnessed by the identity. 
 
 In \cite{eilenberg1953groups} (see also \cite{joyal1993braided}) it was shown that the map $\lbrack (\psi, \Omega) \rbrack \mapsto q_{\Omega}$ with $\quad q_{\Omega}(g) := \Omega(g,g)$ yields an isomorphism $H^{3}_{\textnormal{ab}}(G, \mathbbm{k}^{\times}) \cong \QF$. Here we give a variant of this statement with respect to the equivalence relations $\sim$.

\begin{lemma}
\label{emiso}
	Let $G$ be a finite abelian group. The map
	\begin{align*}
		\textnormal{EM}: \quad H^{3}_{\textnormal{ab}}(G, \mathbbm{k}^{\times})/\sim &\rightarrow \textnormal{QF}(G, \mathbbm{k}^{\times})/\sim \\
		\lbrack (\psi, \Omega) \rbrack &\mapsto \lbrack q_{\Omega} \rbrack \quad \textnormal{with} \quad q_{\Omega}(g) := \Omega(g,g)
	\end{align*}
	is an isomorphism of abelian groups. Moreover, for the associated bihomomorphism it holds
	\begin{equation}
	\label{assocbihomomega}
		\beta_{q_{\Omega}}(g_1,g_2) = \Omega(g_1,g_2)\Omega(g_2,g_1).
	\end{equation}
\end{lemma}
\begin{proof}
	We only need to show the differences to the original statement. 	First of all, if $\lbrack (\psi_1, \Omega_1) \rbrack \sim \lbrack (\psi_2, \Omega_2) \rbrack$, then by definition there exists some $f \in \textnormal{Aut}(G)$ and $\kappa \in C^2(G, \mathbbm{k}^{\times})$, such that in particular $\Omega_1 = f^{\star}(\Omega_2) \cdot \kappa_{\textnormal{comm}}$. Thus it follows
	\[ q_{\Omega_1}(g) = \Omega_1(g,g) = \Omega_2(f(g),f(g))\kappa_{\textnormal{comm}}(g,g) = \Omega_2(f(g),f(g)) = f^{\star}q_{\Omega_2}(g), \]
	which implies $q_{\Omega_1} \sim q_{\Omega_2}$. Conversely, note that if $(\psi, \Omega) \in Z^3_{\textnormal{ab}}(G, \mathbbm{k}^{\times})$, also $(f^{\star}\psi, f^{\star}\Omega) \in Z^3_{\textnormal{ab}}(G, \mathbbm{k}^{\times})$ for all $f \in \textnormal{Aut}(G)$. Let $q, q' \in \QF$, such that $q \sim q'$, witnessed by $f \in \textnormal{Aut}(G)$. Moreover, let $(\psi_q, \Omega_q), (\psi_{q'}, \Omega_{q'})$ be representatives of the inverses, i.e. $q(g) = \Omega_q(g,g)$ and $q'(g) = \Omega_{q'}(g,g)$. Then
	\[
	q_{f^{\star}\Omega_{q'}}(g) = f^{\star}\Omega_{q'}(g, g) = \Omega_{q'}(f(g), f(g)) = q'(f(g)) = q(g) = \Omega_{q}(g,g) = q_{\Omega_{q}}(g),
	\]
	which implies $\lbrack (\psi_q, \Omega_q) \rbrack = \lbrack (f^{\star}\psi_{q'}, f^{\star}\Omega_{q'}) \rbrack$. This shows that $\lbrack (\psi_q, \Omega_q) \rbrack \sim \lbrack (\psi_{q'}, \Omega_{q'}) \rbrack$.
	Since the equivalence relation $\sim$ on $H^{3}_{\textnormal{ab}}(G, \mathbbm{k}^{\times})$ is coarser than the original one on $H^{3}_{\textnormal{ab}}(G, \mathbbm{k}^{\times})$, the surjectivity of $\textnormal{EM}$ is induced from the classical statement. Similarly, if $q_{\Omega} \sim 1$, then $q_{\Omega} = f^{\star}1 = 1$ for some $f \in \textnormal{Aut}(G)$. Thus, using the injectivity provided by the original map, $\lbrack (\psi_q, \Omega_q) \rbrack = \lbrack (1,1) \rbrack = \lbrack (f^{\star}1, f^{\star}1) \rbrack$. This shows $(\psi_q, \Omega_q) \sim (1,1)$, and hence, the injectivity of $\textnormal{EM}$. 
\end{proof}

\subsection{Ribbon structures of graded vector spaces}

In this subsection we will introduce the category $\Gvect$ of finite-dimensional $G$-graded vector spaces and recall that, surprisingly, quadratic forms are related to ribbon structures of $\Gvect$.

\begin{definition}[$G$-graded vector space]
\label{gradedvecdef}
	Let $G$ be a finite group. A \textit{G-graded vector space} $V$ over $\mathbbm{k}$ is a direct sum $V = \oplus_{g \in G} V_g$ of $\mathbbm{k}$-vector spaces $V_g$ indexed by $G$. We say $V$ is finite-dimensional, if $V_g$ is finite-dimensional for all $g \in G$. In the following we will write $V_h$ for the vector space of $V= \oplus_{g \in G} V_g$ that is indexed by an element $h \in G$.
	The direct sum $V \oplus W$ of $G$-graded $\mathbbm{k}$-vector spaces $V,W$ is defined as
\[
(V \oplus W)_g := V_g \oplus W_g,
\]
and the tensor product $V \otimes W$ as 
	\begin{equation}
	\label{tensorproductgraded}
		(V \otimes W)_g := \bigoplus_{hk = g} V_h \otimes W_k.
	\end{equation}
	
 We write $\mathbbm{k}_g$ for the $G$-graded vector space with $(\mathbbm{k}_g)_h = \delta_{gh} \mathbbm{k}$, that is, for the (representatives of isomorphism classes of) simple objects indexed by elements of $G$. In particular, we will abbreviate $\mathbbm{k} := \mathbbm{k}_{1_G}$ for the simple object indexed by the unit element $1_G$ of $G$. 
 
The \textit{category of G-graded} $\mathbbm{k}$-\textit{vector spaces} $\GVect$ has as objects $G$-graded $\mathbbm{k}$-vector spaces and as morphisms $\mathbbm{k}$-linear maps $f: V \rightarrow W$ such that $f(V_g) \subseteq W_g$ for all $g \in G$. The subcategory of finite-dimensional $G$-graded vector spaces is denoted by $\Gvect$.	 
\end{definition}

An important example is given by \textit{super vector spaces}, which are $\mathbb{Z}_2$-graded vector spaces.

The category $\Gvect$ can be equipped with a monoidal structure $(\otimes, \mathbbm{k}, \alpha, \lambda, \varrho)$, with tensor product \eqref{tensorproductgraded}, tensor unit $\mathbbm{k} = \mathbbm{k}_{1_G}$, and from $\vect$ inherited associator $\alpha$, and left and right unitors $\lambda, \varrho$. Moreover, $\Gvectc$ equipped with this monoidal structure is rigid with respect to the duality $(-)^{\star}$ given on $V \in \Gvectc$ at index $g \in G$ by
\begin{equation}
\label{classicalgradedduality}
	(V^{\star})_g := (V_{g^{-1}})^{\star} \in \vectc.
\end{equation}

\begin{definition}[Ribbon monoidal structure]
\label{classicalribbonmonoidalstructures}
Let $G$ be a finite abelian group and let $\RMS$ be the set of ribbon monoidal structures on $G\textnormal{-vect}_{\mathbbm{k}}$, that is, $\RMS$ has as elements triples 
	\[ (\alpha, \gamma, \theta), \] 
	where $\alpha: \otimes \circ (\otimes \times \textnormal{id}_{G\textnormal{-vect}_{\mathbbm{k}}}) \cong \otimes \circ (\textnormal{id}_{G\textnormal{-vect}_{\mathbbm{k}}} \times \otimes)$, $\gamma: \otimes \cong \otimes \circ \tau$ and $\theta: \textnormal{id}_{G\textnormal{-vect}_{\mathbbm{k}}} \cong \textnormal{id}_{G\textnormal{-vect}_{\mathbbm{k}}}$ are natural isomorphisms, such that
$	(\otimes, \mathbbm{k}, \alpha, \lambda, \varrho, \gamma, \theta)
$	defines a ribbon monoidal structure on  $G\textnormal{-vect}_{\mathbbm{k}}$ with respect to the duality \eqref{classicalgradedduality}. One can define an equivalence relation $\sim$ on $\RMS$ by
\[	(\alpha_1, \gamma_1, \theta_1) \sim (\alpha_2, \gamma_2, \theta_2), \] 
if and only if there exists a braided monoidal equivalence 
		\[
	(F, \Phi, \phi): (G\textnormal{-vect}_{\mathbbm{k}}, \otimes, \mathbbm{k}, \alpha_1, \lambda, \varrho, \gamma_1) \cong (G\textnormal{-vect}_{\mathbbm{k}}, \otimes, \mathbbm{k}, \alpha_2, \lambda, \varrho, \gamma_2),
	\]
	such that $F((\theta_1)_{\mathbbm{k}_g}) = (\theta_2)_{F(\mathbbm{k}_g)}$.
\end{definition}


\begin{definition}
Let $G$ be a finite abelian group.
	For $(\psi, \Omega) \in Z^3_{\textnormal{ab}}(G, \mathbbm{k} ^{\times})$ a normalized abelian $3$-cocycle and $q: G \rightarrow \mathbbm{C}^{\times}$ any map, let natural isomorphisms $\alpha_{\psi}: \otimes \circ (\otimes \times \textnormal{id}_{G\textnormal{-vect}_{\mathbbm{k}}}) \cong \otimes \circ (\textnormal{id}_{G\textnormal{-vect}_{\mathbbm{k}}} \times \otimes)$, $\gamma_{\Omega}: \otimes \cong \otimes \circ \tau$, and $\theta_q: \textnormal{id}_{\Gvect} \cong \textnormal{id}_{\Gvect}$ be defined as 
	\begin{alignat*}{2}
		(\alpha_{\psi})_{g_1,g_2,g_3}&: \quad &(\mathbbm{k}_{g_1} \otimes \mathbbm{k}_{g_2}) \otimes \mathbbm{k}_{g_3} &\rightarrow (\mathbbm{k}_{g_1} \otimes \mathbbm{k}_{g_2}) \otimes \mathbbm{k}_{g_3}  \\
		 & &(1_{g_1} \otimes 1_{g_2}) \otimes 1_{g_3} &\mapsto \psi(g_1,g_2,g_3) 1_{g_1} \otimes (1_{g_2} \otimes 1_{g_3}), \\
		(\gamma_{\Omega})_{g_1,g_2} &:  &\mathbbm{k}_{g_1} \otimes \mathbbm{k}_{g_2} &\rightarrow \mathbbm{k}_{g_2} \otimes \mathbbm{k}_{g_1} \\
		& & \quad 1_{g_1} \otimes 1_{g_2} &\mapsto \Omega(g_2, g_1) 1_{g_2} \otimes 1_{g_1}, \\
		(\theta_q)_g &:  &\mathbbm{k}_g &\rightarrow \mathbbm{k}_g \\
		& & \quad 1_g &\mapsto q(g) 1_g.
	\end{alignat*}
\end{definition}

The following \autoref{cohomologybraided} is a known generalisation (from braided monoidal structures to ribbon monoidal structures) of \cite{eilenberg1950cohomology} (see also \cite[Prop. 2.6.1]{etingof2015tensor}). 

To this end, let us define equivalence relations $\sim$ on 
$H^3_{\textnormal{ab}}(G, \mathbbm{k} ^{\times}) \oplus \textnormal{Hom}(G, \lbrace \pm 1 \rbrace)$ and $\QF \oplus \textnormal{Hom}(G, \lbrace \pm 1 \rbrace)$, respectively, by 
\[
(\lbrack (\psi, \Omega)\rbrack, \eta) \sim (\lbrack (\psi', \Omega')\rbrack, \eta')  :\Leftrightarrow \exists f \in \textnormal{Aut}(G) :  \lbrack (\psi, \Omega) \rbrack
 = \lbrack (f^{\star}\psi', f^{\star}\Omega') \rbrack,\ \eta = f^{\star}\eta'
\]
and
\[
(q, \eta) \sim (q', \eta') :\Leftrightarrow \exists f \in \textnormal{Aut}(G) : q\eta=f^{\star}(q'\eta').
\]
\begin{lemma}
\label{cohomologybraided}
	Let $G$ be a finite abelian group. The map
	\begin{align*}
		(H^3_{\textnormal{ab}}(G, \mathbbm{k} ^{\times}) \oplus \textnormal{Hom}(G, \lbrace \pm 1 \rbrace))/\sim\ &\rightarrow \RMS/\sim \\
 \lbrack ((\psi, \Omega), \eta) \rbrack &\mapsto \lbrack (\alpha_{\psi}, \gamma_{\Omega}, \theta_{q_{\Omega}\eta} ) \rbrack
	\end{align*}
is a bijection.
\end{lemma}

We will obtain the following Corollary later as a special case of \autoref{theoremweakribbon}.

\begin{corollary}
\label{ribbonqf}
	Let $G$ be a finite abelian group. There exists a bijection 
	\[
	(\QF \oplus \textnormal{Hom}(G, \lbrace \pm 1 \rbrace))/\sim\ \cong\ \RMS/\sim
	\]
	between the groupoid of quadratic forms on $G$ with values in $\mathbbm{k}^{\times}$ and the groupoid of braided monoidal structures on $G\textnormal{-vect}_{\mathbbm{k}}$.
\end{corollary}
\begin{proof}
	Combine \autoref{cohomologybraided} and \autoref{emiso}.
\end{proof}

\subsection{Generalised duality of graded vector spaces}
\begin{definition}
Let $G$ be a finite group and let $V,W \in \GVectc$ be $G$-graded vector spaces. A linear map $f: V \rightarrow W$ \textit{is of degree} $g \in G$, if
	\[
	f(V_h) \subseteq W_{gh} \quad \textnormal{for all } h \in G.
	\]
	The collection of morphisms between $V$ and $W$ of degree $g$ is denoted by $\textnormal{Hom}_g(V,W)$.
\end{definition}

\begin{lemma}
\label{internalhomgradedlemma}
	Let $G$ be a finite group and let $(\otimes, \mathbbm{k}, \alpha', \lambda, \rho)$ be a monoidal structure of $\GVect$ as in \eqref{tensorproductgraded}, but with possibly different associator $\alpha'$. Then
	\begin{equation}
	\label{internalhomgraded}
		\underline{\textnormal{Hom}}(V,W)_g := \textnormal{Hom}_{g}(V,W)
	\end{equation}
	defines an internal hom in $(G\textnormal{-vect}_{\mathbbm{k}}, \otimes, \mathbbm{k}, \alpha', \lambda, \varrho)$.
\end{lemma}
\begin{proof}
	We have to show that for any $G$-graded vector spaces $U, V, W$ there exists a natural bijection
	\[
	\textnormal{Hom}(U \otimes V, W) \cong \textnormal{Hom}(U, \underline{\textnormal{Hom}}(V,W)).
	\]
	Let $f: U \otimes V \rightarrow W$ be given. Define $\tilde{f}: U \rightarrow \underline{\textnormal{Hom}}(V,W)$ as follows. For $g,h \in G$ and $x \in U_g, y \in V_h$ let 
	\[
	\tilde{f}(x)(y) := f(x \otimes y) \in W_{gh}.
	\]
	Conversely, let $f: U \rightarrow \underline{\textnormal{Hom}}(V,W)$ be given. Define $\tilde{f}: U \otimes V \rightarrow W$ as follows. For $g,h,k$ such that $hk=g$ and $x \in U_h, y \in V_k$ let 
	\[
	\tilde{f}(x \otimes y) := f(x)(y) \in W_{hk} = W_{g}.
	\]
	It follows immediately from the definitions that both constructions are inverse to each other.
\end{proof}

\begin{definition}[Dual space $V^{g_0}$]
\label{generaldualitygraded}
	Let $G$ be a finite group. For a $G$-graded vector space $V \in \GVectc$ and $g_0 \in G$ define $V^{g_0}$ by
	\[
	V^{g_0} := \underline{\textnormal{Hom}}(V, \mathbbm{k}_{g_0}).
	\]
\end{definition}

\begin{lemma}
\label{gradeddualeq}
	Let $G$ be a finite group. For a $G$-graded vector space $V \in \GVectc $ and $g,g_0,g_1 \in G$ it holds
	\begin{equation}
	\label{dualgeneral}
		(V^{g_0})_g = (V_{g^{-1}g_0})^{\star}.
	\end{equation}
	Thus, if $V$ is finite-dimensional, $V \in \Gvectc$, there exists an isomorphism of vector spaces
	\begin{equation}
	\label{doubledualgeneral}
		((V^{g_0})^{g_1})_g \cong V_{g_1^{-1} g g_0}.
	\end{equation}
\end{lemma}
\begin{proof}
Writing out the definitions we find
	\begin{align*}
		(V^{g_0})_g 
		&= \underline{\textnormal{Hom}}(V, \mathbbm{k}_{g_0})_g
		= \textnormal{Hom}_{g}(V,\mathbbm{k}_{g_0})
		= \lbrace f: V \rightarrow \mathbbm{k}_{g_0} \mid f(V_h) \subseteq (\mathbbm{k}_{g_0})_{gh} \textnormal{ for all } h \in G \rbrace \\
		&= \lbrace f: V \rightarrow \mathbbm{k}_{g_0} \mid f(V_{g^{-1}g_0}) \subseteq \mathbbm{k} \wedge f(V_h) = 0 \textnormal{ for all } h \in G \setminus \lbrace g^{-1}g_0 \rbrace \rbrace \\
		&= \textnormal{Hom}_{\mathbbm{k}}(V_{g^{-1}g_0}, \mathbbm{k})
		= (V_{g^{-1}g_0})^{\star}.
	\end{align*}
	Thus it follows
	\[
	((V^{g_0})^{g_1})_g = ((V^{g_0})_{g^{-1}g_1})^{\star} = (V_{(g^{-1}g_1)^{-1}g_0})^{\star \star} \cong V_{g_1^{-1} g g_0}.
	\]
\end{proof}

The duality $(-)^{g_0}$ for $g_0 \in G$ can be extended to a contravariant functor in the following way. If $f: V \rightarrow W$ is a morphism of $G$-graded vector spaces, define a morphism $f^{g_0}: W^{g_0} \rightarrow V^{g_0}$ as
	\[
	f^{g_0}(w')(v) := w'(f(v)),
	\]
	where $w' \in (W^{g_0})_g = (W_{g^{-1}g_0})^{\star}$ and $v \in V_{g^{-1}g_0}$.

Assuming that $G$ is not only finite, but also abelian, and $V$ finite-dimensional, yields the following result. 

\begin{corollary}
\label{doubledualgeneral}
Let $G$ be a finite abelian group, $V \in \Gvectc$ a finite-dimensional $G$-graded vector space and $g_0 \in G$, then 
	\[ (V^{g_0})^{g_0} \cong V. \]
\end{corollary}

\begin{definition}[Tensor product $\otimes_{g_0}$]
\label{generaltensordefdef}
	Let $G$ be a finite group and $g_0 \in G$. For $G$-graded vector spaces $V,W \in \Gvectc$
	\begin{equation}
	\label{generaltensordef}
	V \otimes_{g_0} W := (V^{g_0} \otimes W^{g_0})^{g_0}.
	\end{equation}
\end{definition}

\begin{lemma}
\label{generaltensorlemma}
	Let $G$ be a finite abelian group, $g,g_0 \in G$ and $V,W \in \Gvectc$ finite-dimensional $G$-graded vector spaces. Then it holds
	\begin{equation}
	\label{generaltensor}
			(V \otimes_{g_0} W)_g \cong (V \otimes W)_{g g_0}.
	\end{equation}
\end{lemma}
\begin{proof}
The claim follows from
	\begin{align*}
		&(V \otimes_{g_0} W)_g 
		= ((V^{g_0} \otimes W^{g_0})^{g_0})_g 
		\overset{\eqref{dualgeneral}}{=} ((V^{g_0} \otimes W^{g_0})_{g^{-1}g_0})^{\star}
		= (\bigoplus_{hk = g^{-1}g_0} (V^{g_0})_h \otimes (W^{g_0})_k)^{\star} \\
		&\overset{\eqref{dualgeneral}}{=}  (\bigoplus_{hk = g^{-1}g_0} (V_{h^{-1}g_0})^{\star} \otimes (W_{k^{-1}g_0})^{\star})^{\star}
		\cong \bigoplus_{hk = g^{-1}g_0} V_{h^{-1}g_0} \otimes W_{k^{-1}g_0}
		\cong \bigoplus_{h'k'=gg_0} V_{h'} \otimes W_{k'} \\
		&= (V \otimes W)_{g g_0}.
	\end{align*}
\end{proof}

The following is a standard result, which we will not prove.

\begin{lemma}
\label{generalinduced}
Let $(\mathcal{C}, \otimes_{\mathcal{C}}, \unit, \alpha, \lambda, \varrho)$ be a monoidal category and $\mathcal{D}$ an arbitrary category. Assume there exists an equivalence $F: \mathcal{C} \rightarrow \mathcal{D}$ of the underlying categories with quasi-inverse $F^{-1}$. Then $\mathcal{D}$ can be equipped with a monoidal structure, such that the tensor product is given by 
\[
 x \otimes^F_{\mathcal{D}} y := F(F^{-1}(x)  \otimes_{\mathcal{C}} F^{-1}(y)).
 \]
 Furthermore, the functor $F$ can be equipped with the structure of a strong monoidal functor $(\mathcal{C}, \otimes_{\mathcal{C}}, \unit) \rightarrow (\mathcal{D}, \otimes^F_{\mathcal{D}}, F(\unit))$.
\end{lemma}

\begin{corollary}
	Let $G$ be a finite group. For any $g_0 \in G$ there exists a monoidal structure $(\otimes_{g_0}, \mathbbm{k}_{g_0}, \alpha_{g_0}, \lambda_{g_0}, \varrho_{g_0})$ on $\Gvectc$, such that the tensor product is given by \eqref{generaltensordef}. The category $(\Gvect, \otimes, \mathbbm{k}, \alpha, \lambda, \varrho)$ is $\star$-autonomous with internal hom \eqref{internalhomgraded} and dualizing object $\mathbbm{k}_{g_0}$. The duality functor 
	\[
	(-)^{g_0} = \underline{\textnormal{Hom}}(-,\mathbbm{k}_{g_0}):\quad (\Gvectc ,\otimes, \mathbbm{k}) \rightarrow (\Gvectc,\otimes_{g_0}, \mathbbm{k}_{g_0})^{\textnormal{opp}(0,1)}.
	\] 
	provides a strong monoidal equivalence.
	\end{corollary}
\begin{proof}
	For the first statement we use \autoref{generalinduced} with $\mathcal{C} = (\Gvect, \otimes, \mathbbm{k})$, $\mathcal{D} = \Gvect$ and $F = F^{-1} = (-)^{g_0}$. It remains to show that $F(\mathbbm{k}) = \mathbbm{k}_{g_0}$. But for any $g,h \in G$ it holds $(\mathbbm{k}_{g})^h \cong \mathbbm{k}_{hg^{-1}}$. Thus, in particular, $F(\mathbbm{k}) = F(\mathbbm{k}_{1_G}) = (\mathbbm{k}_{1_G})^{g_0} \cong \mathbbm{k}_{g_0}$. Moreover, from the definition of the internal hom $\eqref{internalhomgraded}$ and \autoref{doubledualgeneral} it follows that $\mathbbm{k}_{g_0}$ is a dualizing object. The duality $(-)^{g_0}$ is strong monoidal in the above sense by the general result \autoref{linearlydistributivemonoidalfunctor} and the definition of the tensor product $\otimes_{g_0}$.
	\end{proof}

\subsection{Weak quadratic forms}

In \autoref{ribbonqf} we have seen that quadratic forms, which are symmetric with respect to $1_G$, are related to ribbon structures on $\Gvect$ with respect to the duality $(-)^{1_G}$. The purpose of this Section 4.2.4 and the following Section 4.2.5 is to give a general version of this statement. It remains to define an appropriate notion of ribbon structures with respect to such a relaxed duality -- this will be done in Section $4.2.5.$ On the other hand, it seems natural to generalise the symmetry axiom $q(g)=q(g^{-1}1_G)$ of quadratic forms to $g(g) = q(g^{-1}g_0)$ for arbitrary $g_0 \in G$. In fact, we will show that under certain conditions, we do not need to introduce such an axiom at all, and instead deduce a symmetry.

\begin{definition}[Weak quadratic form]
\label{weakquadraticformdef}
	Let $G$ be a finite group. A \textit{weak quadratic form} on $G$ with values in $\mathbbm{k}^{\times}$ is a map
	\[
	q: \quad G \rightarrow \mathbbm{k}^{\times},
	\]
	such that the function 
	\begin{align}
	\label{assocbihomweak}
	\begin{split}
		\beta_q: \quad G \times G &\rightarrow \mathbbm{k}^{\times} \\
		(g_1, g_2) &\mapsto q(g_1 g_2)q(g_1)^{-1}q(g_2)^{-1}
	\end{split}		
	\end{align}
	is a bihomomorphism. The collection of weak quadratic forms on $G$ with values in $\mathbbm{k}^{\times}$ forms in a canonical way an abelian group, which we denote by $\WQF$. We will sometimes use an equivalence relation $\sim$ on $\WQF$ defined by 
	\[ q_1 \sim q_2\ :\Leftrightarrow\ \exists f \in \textnormal{Aut}(G): q_1 = f^{\star}q_2.  \]
	for $q_1,q_2 \in \WQF$.
\end{definition}

For the proof of \autoref{weakquadraticformrepr} we need the following two technical results.

\begin{lemma}
\label{weakeq}
	Let $q$ be a weak quadratic form on $\mathbb{Z}_n$. Then it holds
	\begin{equation}
		\label{qk}
			q(k) = q(1)^k \beta_q(1,1)^{\binom{k}{2}}
	\end{equation}
	for $k \geq 2$.
\end{lemma}
\begin{proof}
	We apply a simple induction over $k$ and use additive notation for $\mathbb{Z}_n$. For $k = 2$ the definition \eqref{defassocbihom} yields
	\begin{align*}
		q(2) = q(1+1) = q(1)q(1) \beta_q(1,1) = q(1)^2 \beta(1,1)^{\binom{2}{2}}.
	\end{align*}
	Assume that the statement is true for $k \geq 2$. Again, using $\eqref{defassocbihom}$, we find
	\begin{align*}
		q(k+1) = q(k)q(1)\beta_q(1,1)^k \overset{\eqref{qk}}{=} q(1)^{k+1} \beta_q(1,1)^{\binom{k}{2} + k} = q(1)^{k+1} \beta_q(1,1)^{\binom{k+1}{2}}.
	\end{align*}
	The last equation follows from the general result $\binom{n+1}{k+1} = \binom{n}{k} + \binom{n}{k+1}$.
\end{proof}

\begin{lemma}
\label{weakeq2}
	Let $q$ be a weak quadratic form on $\mathbb{Z}_n$ with values in $\mathbbm{k}^{\times}$. Then it holds
	\begin{equation}
	\label{qn1}
		q(n)=1.
	\end{equation}
\end{lemma}
\begin{proof}
	We will use additive notation for $\mathbb{Z}_n$. Since $\beta_q$ is a bihomomorphismus it holds
	\[
	\beta_q(n,k) = \beta(n+n,k) = \beta(n,k)^2.
	\]
	Thus we find
	\[ \beta_q(n,k) = 1. \] Using in addition $\eqref{defassocbihom}$ one checks that
	\[
	1 = \beta_q(n,k) = q(n+k)q(n)^{-1}q(k)^{-1} = q(k)q(n)^{-1}q(k)^{-1} = q(n)^{-1},
	\]
	which implies the claim $q(n) = 1$.
\end{proof}

\begin{theorem}
	\label{weakquadraticformrepr}
Let $G$ be a finite abelian group and $q \in \WQF$. There exists a character $\eta: G \rightarrow \mathbbm{k}^{\times}$ and a quadratic form $\tilde{q} \in \QF$ on $G$ with values in $\mathbbm{k}^{\times}$, such that
\[
q = \tilde{q} \cdot \eta.
\]  
Moreover, the associated bilinearforms coincide, $\beta_{q} = \beta_{\tilde{q}}$.
\end{theorem}
\begin{proof}
\begin{enumerate}
	\item We begin with $G = \mathbb{Z}_n$. Note that we will use additive notation for $\mathbb{Z}_n$. Abbreviate $B := \beta(1,1)$ and $C := q(1)$. Since $\beta$ is a bihomomorphismus it holds $B^n = 1$. 
Define
\begin{equation*}
	\eta(k) := \begin{cases}
	(CB^{-\frac{1}{2}})^k, & 2 \mid n \\
	(C B^{-l})^k, & 2 \nmid n,\ 2l = n-1
	\end{cases}
\end{equation*}
and
\begin{equation*}
\tilde{q}(k) := \begin{cases}
	(B^{\frac{1}{2}})^k B^{\binom{k}{2}}, & 2 \mid n \\
	(B^{l})^k B^{\binom{k}{2}}, & 2 \nmid n,\ 2l = n-1.	
\end{cases}	
\end{equation*}
From $\eqref{qk}$ it follows that in both cases $q(k) = \tilde{q}(k)\eta(k)$. It remains to show that $\eta$ and $\tilde{q}$ are well defined and $\tilde{q} \in \textnormal{QF}(G, \mathbbm{k}^{\times})$.
	\begin{itemize}
		\item We show that $\eta(k) = (CB^{-\frac{1}{2}})^k$ is well defined for $2 \mid n$. It is clear that $\eta$ is linear in $k$. Note that $1 \overset{\eqref{qn1}}{=} q(n) \overset{\eqref{qk}}{=} C^n B^{\binom{n}{2}}$ implies that $C^n B^{-\frac{n}{2}} = B^{-\frac{n^2}{2}}$. It follows
		\[
		\eta(n) = (CB^{-\frac{1}{2}})^n = C^n B^{-\frac{n}{2}} = B^{-\frac{n^2}{2}} = (B^{n})^{-\frac{n}{2}} = 1^{-\frac{n}{2}} = 1,
		\]
		where we used in the last equation that $\frac{n}{2} \in \mathbb{Z}$. 
		\item We show that $\eta(k) = (C B^{-l})^k$ is well defined for $2 \nmid n$. . It is again clear that $\eta$ is linear in $k$. Furthermore, $1 = C^n B^{\binom{n}{2}}$ implies that $C^{n}B^{-ln} = B^{-\frac{n(n-1)}{2}} B^{-ln}$. Thus it follows
		\[
		\eta(n) = (C B^{-l})^n = C^n B^{-ln} = B^{-\frac{n(n-1)}{2}} B^{-ln} = (B^n)^{-\frac{n-1}{2}} (B^n)^{-l} = 1,
		\]
		where we used in the last equation that $\frac{(n-1)}{2}, l \in \mathbb{Z}$.
		\item We show that $\tilde{q}(k) = (B^{\frac{1}{2}})^k B^{\binom{k}{2}}$ is well defined for $2 \mid n$. For $a,k \in \mathbb{Z}$ we find
		\[
		\tilde{q}(a+kn) = (B^{\frac{1}{2}})^{a+kn} B^{\binom{a+kn}{2}} = (B^{\frac{1}{2}})^a (B^{n})^k B^{\frac{(a+kn)(a+kn-1)}{2}} = (B^{\frac{1}{2}})^a B^{\frac{a^2 -a}{2}} = \tilde{q}(a).
		\]
		Furthermore, $\tilde{q}$ is symmetric. Indeed,
		\begin{align*}
			\tilde{q}(-k) = (B^{\frac{1}{2}})^{-k} B^{\binom{n-k}{2}} = (B^{\frac{1}{2}})^{-k} B^{\frac{k^2 + k}{2}} = (B^{\frac{1}{2}})^{-k} B^{\frac{k^2 -k}{2}} (B^{\frac{1}{2}})^{2k} = \tilde{q}(k).
		\end{align*}
		Since $\eta$ is linear, it follows
		\[
		\beta_{q}(k_1, k_2) = \frac{q(k_1+k_2)}{q(k_1)q(k_2)}
		= \frac{f(k_1+k_2) \tilde{q}(k_1+k_2)}{f(k_1)\tilde{q}(k_1)f(k_2) \tilde{q}(k_2)} = \frac{ \tilde{q}(k_1+k_2)}{\tilde{q}(k_1)\tilde{q}(k_2)} = \beta_{\tilde{q}}(k_1,k_2).
		\]
		Thus we've shown that $\tilde{q} \in \textnormal{QF}(G, \mathbbm{k}^{\times})$.

		\item Analogously to the previous point one shows that for $2 \nmid n$ with $2l = n-1$ the map $\tilde{q}(k) = (B^{l})^k B^{\binom{k}{2}}$ is a well defined quadratic form.
	\end{itemize}
\item We now consider arbitrary finite abelian groups. W.l.o.g we can assume $G = \bigoplus_{i=1}^n \mathbb{Z}_{n_i}$ for some $n \in \mathbb{N}$. Let $q$ be a weak quadratic form on $G$. Define weak quadratic forms $q_i := \left.q\right|_{\mathbb{Z}_{n_i}}$ on $\mathbb{Z}_{n_i}$ as restriction of $q$. 
\begin{itemize}
	\item One checks that
\[
q(k_1,...,k_n) = \prod_{i=1}^n q_i(k_i) \prod_{i < j} \beta_{q}(k_i, k_j).
\]
	\item We have shown that we can decompose the weak quadratic forms $q_i$ as $q_i = \eta_i \cdot \widetilde{q_i}$ for $i = 1,...,n$. Thus
\[
q(k_1,...,k_n) = \prod_{i=1}^n \eta_i(k_i) \prod_{i=1}^n \widetilde{q_i}(k_i)  \prod_{i < j} \beta_{q}(k_i, k_j)
= \eta(k_1,...,k_n) \cdot \tilde{q}(k_1,...,k_n)
\]
with maps 
\[ \eta(k_1,...,k_n) :=  \prod_{i=1}^n \eta_i(k_i) \] 
and
\[ \tilde{q}(k_1,...,k_n) :=  \prod_{i=1}^n \widetilde{q_i}(k_i)  \prod_{i < j} \beta_{q}(k_i, k_j) . \]
\item The map $\eta$ is linear in $\bigoplus_{i=1}^n \mathbb{Z}_{n_i}$ since the maps $\eta_i$ are linear in $\mathbb{Z}_{n_i}$. Furthermore, $\tilde{q}$ is symmetric since all $\tilde{q_i}$ are symmetric and $\beta_q$ is a bilinear.
\item The associated bilinearforms coincide, $\beta_{q} = \beta_{\tilde{q}}$. Indeed, using that $\beta_{q_i} = \beta_{\widetilde{q_i}}$, we find for $x,y \in \bigoplus_{i=1}^n \mathbb{Z}_{n_i}$ that
\begin{align*}
	&\beta_q(x,y) = \frac{q(x+y)}{q(x)q(y)} = \prod_{i=1}^n \frac{q_i((x+y)_i)}{q_i(x_i)q_i(y_i)} \prod_{i < j} \frac{\beta_q((x+y)_i, (x+y)_j)}{\beta_q(x_i, x_j) \beta_q (y_i, y_j)} \\
	&= \prod_{i=1}^n \frac{\widetilde{q_i}((x+y)_i)}{\widetilde{q_i}(x_i)\widetilde{q_i}(y_i)} \prod_{i < j} \frac{\beta_q((x+y)_i, (x+y)_j)}{\beta_q(x_i, x_j) \beta_q (y_i, y_j)} 
 = \frac{\tilde{q}(x+y)}{\tilde{q}(x)\tilde{q}(y)} = \beta_{\tilde{q}}(x,y).
\end{align*}
\end{itemize}
\end{enumerate}
\end{proof}
We define an equivalence relation $\sim$ on $\textnormal{QF}(G, \mathbbm{k}^{\times}) \oplus \textnormal{Hom}(G, \mathbbm{k}^{\times})$ by \begin{equation*}
 \label{equivrelat2}
 	(\widetilde{q}, \eta) \sim (\widetilde{q}', \eta') :\Leftrightarrow \exists f \in \textnormal{Aut}(G): \widetilde{q}\eta = f^{\star}(\widetilde{q}'\eta').
 \end{equation*}

\begin{lemma} 
\label{wqfcharaclemma}
Let $G$ be a finite abelian group.
The decomposition in $\autoref{weakquadraticformrepr}$ induces a well-defined bijection
\begin{align}
\label{wqfcharac}
\begin{split}
F:\quad \WQF/\sim\ &\rightarrow (\QF \oplus \charac)/\sim\\
	\lbrack q \rbrack  &\mapsto \lbrack (\widetilde{q}, \eta) \rbrack.	
\end{split}	
\end{align}
\end{lemma}
\begin{proof}
	If $q \sim q'$ in $\WQF$, there exists by definition some $f \in \textnormal{Aut}(G)$, such that $q= f^{\star}q'$. Thus, if $\widetilde{q}, \widetilde{q}' \in \QF$ and $\eta, \eta' \in \charac$ such that $q = \widetilde{q}\eta$ and $q' = \widetilde{q}'\eta'$, it follows by definition of $\sim$ that $(\widetilde{q}, \eta) \sim (\widetilde{q}, \eta)$ in $(\QF \oplus \charac)$.
	The canonical candidate for an inverse of $F$ is 
\[
	G:\quad (\QF \oplus \charac)/\sim\ \rightarrow \WQF/\sim, \quad 
	\lbrack (\widetilde{q}, \eta)\rbrack \mapsto \lbrack \widetilde{q}\eta \rbrack.
\]
Again, from the definitions of the equivalence relations it follows that $G$ is well-defined. By construction it holds $G \circ F = \textnormal{id}_{\WQF/\sim}$. Conversely we find $F \circ G(\lbrack (\widetilde{q}, \eta) \rbrack) = F(\lbrack (\widetilde{q}\eta) \rbrack ) =: \lbrack (\widetilde{q}', \eta') \rbrack = \lbrack (\widetilde{q}, \eta)\rbrack$, since $\widetilde{q}\eta = \widetilde{q}'\eta'$.
\end{proof}

The following Corollary to \autoref{weakquadraticformrepr} shows that for weak quadratic forms, under certain conditions, a generalised, shifted, version of the symmetry of quadratic forms in the sense of \autoref{classicalquadraticformdef} holds.

\begin{corollary}
\label{symmetrycorollary}
	Let $G$ be a finite abelian group and $q = \tilde{q} \cdot \eta$ a weak quadratic form on $G$ with values in $\mathbbm{k}^{\times}$. If there exists some $g_0 \in G$, such that $\eta(g) = \beta_q(g,g_0)$, it follows
	\[
	q(g) = \frac{\tilde{q}(gg_0)}{\tilde{q}(g_0)}.
	\]
	This implies, in particular, that $q$ is symmetric with respect to $\tilde{g_0} := g_0^{-2}$,
	\[
	q(g) = q(\tilde{g_0}g^{-1}).
	\]
\end{corollary}
\begin{proof}
The first statement follows from
	\begin{align*}
		q(g) 
		&= \tilde{q}(g) \beta_q(g,g_0)  
		= \tilde{q}(g) \frac{q(gg_0)}{q(g)q(g_0)}  
		= \tilde{q}(g) \frac{\beta_q(gg_0,g_0) \tilde{q}(gg_0)}{\beta_q(g,g_0) \tilde{q}(g) \beta_q(g_0,g_0) \tilde{q}(g_0)} 
		=\frac{\tilde{q}(gg_0)}{\tilde{q}(g_0)}.
	\end{align*}
	The shifted symmetry of $q$ is induced by the symmetry of $\tilde{q}$,\[
	q(\widetilde{g_0}g^{-1}) = q(g_0^{-2}g^{-1}) = \frac{\tilde{q}(g_0^{-2}g^{-1}g_0)}{\tilde{q}(g_0)} = \frac{\tilde{q}(g_0^{-1} g g_0^2)}{\tilde{q}(g_0)} = \frac{\tilde{q}(gg_0)}{\tilde{q}(g_0)} = q(g).
	\]
\end{proof}

In view of \autoref{symmetrycorollary} it is reasonable to give the following definitions.

\begin{definition}[Representable quadratic form]
\label{weakrepresentableqf}
		Let $G$ be a finite abelian group. The \textit{groupoid of weak representable quadratic forms on} $G$ with values in $\mathbbm{k}^{\times}$ is 
		\[ \textnormal{WRQF}(G, \mathbbm{k}^{\times}) := \lbrace (q, \eta, g_0) \in \textnormal{QF}(G, \mathbbm{k}^{\times}) \oplus \textnormal{Hom}(G, \mathbbm{k}^{\times}) \oplus G \mid \eta(g) = \beta_q(g,g_0) \rbrace. \]
		One can define an equivalence relation $\sim$ on $\textnormal{WRQF}(G, \mathbbm{k}^{\times})$ by
			\[
	(q, \eta, g_0) \sim (q', \eta', g_0') :\Leftrightarrow \exists f \in \textnormal{Aut}(G): q \eta = f^{\star}(q'\eta'), f(g_0)=g_0'.
	\]
	\end{definition}

\begin{definition}[Weak symmetric quadratic form]
\label{weaksymmqf}
	Let $G$ be a finite abelian group. The \textit{groupoid of weak symmetric quadratic forms on} $G$ with values in $\mathbbm{k}^{\times}$ is 
	\[
	\WSQF := \lbrace (q, g_0) \in \WQF \oplus G \mid q(g) = q(g^{-1}g_0) \forall g \in G \rbrace.
	\]
	One can define an equivalence relation $\sim$ on $\WSQF$ by
	\[
	(q, g_0) \sim (q', g_0') :\Leftrightarrow \exists f \in \textnormal{Aut}(G): q = f^{\star}(q'),\ f(g_0) = g_0'.
	\]
\end{definition}

The following two technical results are needed, in particular, for the proof of \autoref{wrqfwsqf}. To this end, let $\sumf : \mathbb{N} \rightarrow \mathbb{N}$ be recursively defined as $\sumf(0) = 0$ and $\sumf(n) = \sumf(n-1) + (n-1)$. Explicitly, $\sumf(n)$ is the sum of all strict smaller natural numbers than $n$, $\sumf(n) =  1 + 2 + ... + (n-1)$.

\begin{lemma}
\label{weakhochk}
	Let $G$ be a finite abelian group
	and let $(q, g_0) \in \WSQF$ be a weak quadratic form symmetric with respect to $g_0$, i.e. $q(g) = q(g^{-1}g_0)$ for all $g \in G$. Then
	\begin{align}
	\label{weakhochkeq}
	\begin{split}
		q(g^k) &= q(g)^{k^2- \sumf(k)} q(gg_0)^{\sumf(k)}, \\
		q(g^{-k}) &= q(g)^{\sumf(k)}q(gg_0)^{k^2- \sumf(k)},
	\end{split}	
	\end{align}
	for all $k \in \mathbb{N}$.
\end{lemma}
\begin{proof}
	First of all note that, using $q(g) = q(g^{-1}g_0)$, in particular $q(g_0) = q(1) = 1$, and the bilinearity of $\beta_q$,
	\begin{align*}
		&\frac{q(g^{-1})}{q(g^2)q(g)}
		= \frac{q(g_0g)}{q(g^2)q(g)}
		= \frac{q(g^2g_0g^{-1})}{q(g^2)q(g_0g^{-1})}
		\overset{\eqref{assocbihomweak}}{=} \beta_q(g^2,g_0g^{-1})
		= \beta_q(g,g_0)^2 \beta_q(g,g)^{-2} \\
		&\overset{\eqref{assocbihomweak}}{=} \frac{q(gg_0)^2}{q(g)^2 q(g_0)^2} \frac{q(g)^4}{q(g^2)^2}
		= \frac{q(g^{-1})^2 q(g)^2}{q(g^2)^2}.
	\end{align*}
	This implies 
	\begin{equation}
	\label{qg2}
		q(g^2) = q(g)^3 q(g^{-1}) = q(g)^3 q(g g_0).
	\end{equation}
	We use induction over $k$. For $k=0$, both statements are true, and equal. For $k=1$, the first statement is true, since trivially $q(g) = q(g)$, and the second statement is true, since $q$ is symmetric w.r.t. $g_0$, $q(g^{-1}) = q(gg_0)$. Now let both statements hold for some fixed $k \geq 1$. Then, using the using the bilinearity of $\beta_q$, the recursive definition of $\sumf$, and, as shown above, $q(g^2) = q(g)^3 q(g g_0)$, we compute 
	\begin{align*}
		q(g^{k+1}) &= q(gg^k) \overset{\eqref{assocbihomweak}}{=} \beta_q(g,g^k)q(g)q(g^k) \overset{\eqref{assocbihomweak}}{=} q(g^2)^k q(g)^{-2k} q(g) q(g^k) \\
		&\overset{\eqref{qg2}, \eqref{weakhochkeq}}{=} q(g)^{3k} q(g g_0)^k q(g)^{-2k} q(g) q(g)^{k^2 - \sumf(k)} q(gg_0)^{\sumf(k)} \\
		&= q(g)^k q(g)^{k^2 - \sumf(k) +1} q(gg_0)^{\sumf(k) +k} \\
		&= q(g)^{k^2 +k +1 - \sumf(k)} q(gg_0)^{\sumf(k+1)} \\
		&= q(g)^{k^2 +k +1 - (\sumf(k+1)-k)} q(gg_0)^{\sumf(k+1)} \\
		&= q(g)^{(k+1)^2 - \sumf(k+1)} q(gg_0)^{\sumf(k+1)},
	\end{align*}
	and
	\begin{align*}
		q(g^{-(k+1)}) &= q(g^{-k}g^{-1}) \overset{\eqref{assocbihomweak}}{=} \beta_q(g^{-k},g^{-1})q(g^{-k})q(g^{-1}) = \beta_q(g,g)^k g(g^{-k})q(gg_0) \\
		&\overset{\eqref{assocbihomweak}}{=} q(g^2)^k q(g)^{-2k} g(g^{-k})q(gg_0) \\
		&\overset{\eqref{qg2}, \eqref{weakhochkeq}}{=} q(g)^{3k} q(gg_0)^k q(g)^{-2k}  q(gg_0)^{k^2-\sumf(k)} q(g)^{\sumf(k)}q(gg_0) \\
		&= q(g)^{k + \sumf(k)} q(gg_0)^{k + 1 + k^2 - \sumf(k)} \\
		&= q(g)^{\sumf(k+1)} q(gg_0)^{k+ 1 + k^2 -(\sumf(k+1)-k)} \\
		&= q(g)^{\sumf(k+1)} q(gg_0)^{(k+1)^2 - \sumf(k+1)}.
	\end{align*}
\end{proof}

In particular, for quadratic forms in the sense of \autoref{classicalquadraticformdef}, which are symmetric with respect to $g_0 = 1_G$, this yields $q(g^k) = q(g)^{k^2} = q(g^{-k})$. 
In fact, one often encounters the requirement  $q(g^k) = q(g)^{k^2}$ for all $k \in \mathbb{N}$ instead of the symmetry $q(g) = q(g^{-1})$ in the definition of quadratic forms. Here we have shown, in particular, the equivalence of both definitions. 

\begin{lemma}
\label{samebeta}
	Let $G$ be a finite abelian group and let $(q_0, g_0),(q_1,g_1) \in \WSQF$ be weak quadratic forms symmetric with respect to $g_0, g_1 \in G$, i.e. $q_0(g) = q_0(g^{-1}g_0)$ and $q_1(g) = q_1(g^{-1}g_1)$ for all $g \in G$. Assume $\beta_{q_0} = \beta_{q_1}$, then there exists a character $\eta \in \textnormal{Hom}(G, \mathbbm{k}^{\times})$ such that
	\[
	q_0 = q_1 \cdot \eta.
	\]
	 In particular, if $g_0=g_1=1$ it follows $\eta \in \textnormal{Hom}(G, \lbrace \pm 1 \rbrace)$. Conversely, such a decomposition implies $\beta_{q_0} = \beta_{q_1}$. 
	 \end{lemma}
\begin{proof}
	Using \autoref{weakhochk} we find for $i=0,1$
	\[
	\beta_{q_i}(g,g) = \frac{q_i(g^2)}{q_i(g)^2} = \frac{q_i(g)^3q_i(gg_i)}{q_i(g)^2} = q_i(g)q_i(gg_i).
	\]
	This implies, under the assumption $\beta_{q_0} = \beta_{q_1}$, that $\eta(g) := \frac{q_0(g)}{q_1(g)} = \frac{q_1(gg_1)}{q_0(gg_0)}$. In particular, for $g_0=g_1=1$, it holds $\eta = \eta^{-1}$, i.e. $\eta(g) \in \lbrace \pm 1 \rbrace$. One checks that the equation $\beta_{q_1} = \beta_{q_0} = \beta_{q_1 \eta}$ implies the linearity of $\eta$. Conversely, if $q_0 = q_1 \cdot \eta$, the linearity of $\eta$ shows that $\beta_{q_0} = \beta_{q_1}$.
	\end{proof}

We say a group $G$ \textit{has square roots}, if for all $g \in G$, there exists some element $g^{\frac{1}{2}} \in G$, such that $g^{\frac{1}{2}} g^{\frac{1}{2}} = g$.

\begin{proposition}
\label{wrqfwsqf}
 Let $G$ be a finite abelian group with square roots. The map
	\begin{align}
	\label{wrqfwsqfeq}
	\begin{split}
		F: \quad \WRQF/ \sim &\rightarrow \WSQF/ \sim, \\
		\lbrack (q, \eta, g_0) \rbrack &\mapsto \lbrack (q \eta, g_0^{-2}) \rbrack
	\end{split}
	\end{align}
	is a well-defined bijection.
\end{proposition}
\begin{proof}
	Let $(q, \eta, g_0) \in \WRQF$, then \autoref{symmetrycorollary} implies that $(q \eta, g_0^{-2})$\\ $\in \WSQF$. Moreover, if $(q, \eta, g_0) \sim (q', \eta', g_0')$ in $\WRQF$, there exists by definition some $f \in \textnormal{Aut}(G)$, such that $q\eta = f^{\star}(q'\eta')$ and $f(g_0) = g_0'$. Thus, using the linearity of $f$, it follows $f(g_0^{-2}) = f(g_0)^{-2} = (g_0')^{-2}$. This shows that $(q \eta, g_0^{-2}) \sim (q' \eta', (g_0')^{-2})$, i.e. $F$ is well-defined. 
	
	Define a candidate for an inverse map as 
	\[
G:\quad  \WSQF/\sim\ \rightarrow \WRQF/\sim,\ \lbrack (q, g_0) \rbrack \mapsto \lbrack (\widetilde{q},\eta, g_0^{-\frac{1}{2}}) \rbrack,
\]
with $\widetilde{q} \in \QF$ and $\eta \in \charac$ provided by \autoref{weakquadraticformrepr}, i.e. $q=\widetilde{q}\eta$. 

Let $(q, g_0) \in \WSQF$, $\widetilde{q} \in \QF$ and $\eta \in \charac$ such that $q = \widetilde{q}{\eta}$. Then \autoref{samebeta} implies that $\beta_{q} = \beta_{\widetilde{q}}$, and, as $\widetilde{q}$ is symmetric with respect to the unit $1_G$ of $G$, that $\eta(g) = \frac{q(g)}{\widetilde{q}(g)} = \frac{\widetilde{q}(g)}{q(gg_0)}$. This implies 
$
q(gg_0) = \frac{\widetilde{q}(g)^2}{q(g)}.
$ Using that, and $q(g_0)=1$, it follows
\begin{align*}
	\beta_{\widetilde{q}}(g,g_0^{-\frac{1}{2}}) &= \beta_{q}(g,g_0^{-\frac{1}{2}}) = \beta_{q}(g,g_0)^{-\frac{1}{2}} = (\frac{q(gg_0)}{q(g)q(g_0)})^{-\frac{1}{2}} = (\frac{q(gg_0)}{q(g)})^{-\frac{1}{2}} \\
	 &= (\frac{\widetilde{q}(g)^2}{q(g)^2})^{-\frac{1}{2}} = \frac{q(g)}{\widetilde{q}(g)} = \eta(g).
\end{align*}
This shows that $(\widetilde{q},\eta, g_0^{-\frac{1}{2}}) \in \WRQF$.

Let $(q, g_0) \sim (q', g_0')$ in $\WSQF$, then by definition there exists some $f \in \textnormal{Aut}(G)$, such that $q = f^{\star}q'$ and $f(g_0) = g_0'$. If $\widetilde{q},  \widetilde{q}' \in \QF$ and $\eta, \eta' \in \charac$, such that $q = \widetilde{q}\eta$ and $q' = \widetilde{q}'\eta'$, it follows by assumption and linearity of $f$ that $(\widetilde{q}, \eta, g_0^{-\frac{1}{2}}) \sim (\widetilde{q}, \eta, g_0^{-\frac{1}{2}})$ in $\WRQF$. Thus, $G$ is well-defined.

It remains to show that $F$ and $G$ are mutually inverse. We find
\[	G\circ F(\lbrack (q, \eta, g_0) \rbrack ) 
	= G(\lbrack (q\eta, g_0^{-2}) \rbrack )
	= \lbrack (q',\eta', g_0) \rbrack 
	= \lbrack (q,\eta, g_0) \rbrack ,
\]
since $q'\eta' = q\eta$; and 
\[
F\circ G(\lbrack q,g_0 \rbrack) = F(\lbrack \widetilde{q}, \eta, g_0^{-\frac{1}{2}} \rbrack) = \lbrack (\widetilde{q}\eta, g_0)\rbrack = \lbrack (q,g_0)\rbrack.
\] 
\end{proof}

The results for a finite abelian group $G$ with square roots are summarised in the following diagram:

\begin{equation*}
		\begin{tikzcd}[row sep= 3em, column sep = 2em]
\WQF/\sim \arrow[leftrightarrow]{rr}{\overset{\eqref{wqfcharac}}{\simeq}} && (\QF \oplus \charac)\sim \\
(\WQF \oplus G)\sim \arrow[leftrightarrow]{rr}{\overset{\eqref{wqfcharac}}{\simeq}} \arrow[twoheadrightarrow]{u} &&(\QF \oplus \charac \oplus G)/\sim \arrow[twoheadrightarrow]{u} \\
\WSQF/\sim \arrow[leftrightarrow]{rr}{\overset{\eqref{wrqfwsqfeq}}{\simeq}} \arrow[hook]{u} &&\WRQF/\sim \arrow[hook]{u} \\
& \QF/\sim \arrow[hook]{lu}[left]{\lbrack q \rbrack \mapsto \lbrack (q,1) \rbrack \quad} \arrow[hook]{ru}[right]{\quad \lbrack q \rbrack \mapsto \lbrack (q,1,1) \rbrack }
		\end{tikzcd}.
\end{equation*}

\subsection{Generalised ribbon structures of graded vector spaces}

The purpose of this subsection is to introduce an appropriate notion of ribbon structures with respect to the generalised duality $(-)^{g_0}$ on $\Gvect$ which we defined in \autoref{generaldualitygraded}, and to relate it to the notion of weak quadratic forms which we defined in \autoref{weakquadraticformdef}, similarly to \autoref{ribbonqf}.

\begin{definition}[Weak ribbon monoidal structure]
\label{weakribbonmonoidal}
	Let $\WRMS$ be the set of \textit{weak ribbon monoidal structures} on $G\textnormal{-vect}_{\mathbbm{k}}$, that is, $\WRMS$ has as elements quadruples 
	\[ (\alpha, \gamma, \theta, g_0)\] 
	where $\alpha: \otimes \circ (\otimes \times \textnormal{id}_{G\textnormal{-vect}_{\mathbbm{k}}}) \cong \otimes \circ (\textnormal{id}_{G\textnormal{-vect}_{\mathbbm{k}}} \times \otimes)$, $\gamma: \otimes \cong \otimes \circ \tau$ and $\theta: \textnormal{id}_{G\textnormal{-vect}_{\mathbbm{k}}} \cong \textnormal{id}_{G\textnormal{-vect}_{\mathbbm{k}}}$ are natural isomorphisms, and $g_0 \in G$, such that
$	(\otimes, \mathbbm{k}, \alpha, \lambda, \varrho, \gamma, \theta)$ (cf. \eqref{generaltensordef})
defines a weak ribbon monoidal structure (cf. \autoref{weakribbondef}) on  $G\textnormal{-vect}_{\mathbbm{k}}$ with respect to the $\star$-autonomous duality $(-)^{g_0}$ defined by
	\[
	(V^{g_0})_g := \underline{\textnormal{Hom}}(V,\mathbbm{k}_{g_0}) = (V_{g^{-1}g_0})^{\star}.
	\]
	  One can define an equivalence relation $\sim$ on $\WRMS$ by defining
\[	(\alpha, \gamma, \theta, g_0) \sim (\alpha', \gamma', \theta', g_0'), \] 
if and only if there exists a braided monoidal equivalence 
		\[
	(F, \Phi, \phi): (G\textnormal{-vect}_{\mathbbm{k}}, \otimes, \mathbbm{k}, \alpha, \lambda, \varrho, \gamma) \cong (G\textnormal{-vect}_{\mathbbm{k}}, \otimes, \mathbbm{k}, \alpha', \lambda, \varrho, \gamma'),
	\]
	such that $F(\theta_{\mathbbm{k}_g}) = \theta'_{F(\mathbbm{k}_g)}$.
\end{definition}



\begin{theorem}
\label{theoremweakribbon}
	Let $G$ be a finite abelian group. The map
	\begin{align*}
		\WRQF/\sim &\rightarrow  \WRMS/\sim \\
 \lbrack (q, \eta, g_0) \rbrack &\mapsto \lbrack (\alpha_{\psi_q}, \gamma_{\Omega_q}, \theta_{q\eta}, g_0^{-2} ) \rbrack,
	\end{align*}
for $(\psi_q, \Omega_q) \in Z^3_{\textnormal{ab}}(G, \mathbbm{k} ^{\times})$, such that $\textnormal{EM}(\lbrack (\psi_q, \Omega_q) \rbrack) = q$, is a bijection.
\end{theorem}
\begin{proof}
(Sketch) Let $(q, \eta, g_0) \in \WRQF$ and $(\psi_q, \Omega_q) \in Z^3_{\textnormal{ab}}(G, \mathbbm{k} ^{\times})$, such that $\textnormal{EM}(\lbrack (\psi_q, \Omega_q) \rbrack) = q$.
Then $(\otimes, \mathbbm{k}, \alpha_{\psi_q}, \lambda, \varrho, \gamma_{\Omega_q}, \theta_{q\eta})$ defines a weak ribbon monoidal structure on $\Gvectc$ with respect to the duality $(-)^{g_0^{-2}}$. Indeed, it is monoidal since the pentagon axiom follows from $d\psi_q = 1$ and the triangle axiom is true since $\psi_q$ is normalized. It is braided monoidal since the hexagon axioms follow from \eqref{cocycleequation}. Moreover, $\theta_{q\eta}$ is a twist, since the twist condition $(\theta_{q\eta})_{g_1g_2} = (\gamma_{\Omega_q})_{g_2,g_1} \circ (\gamma_{\Omega_q})_{g_1,g_2} \circ ((\theta_{q\eta})_{g_1} \otimes (\theta_{q\eta})_{g_2})
$	corresponds to
	\[
	q(g_1g_2) \eta(g_1g_2) = \Omega_q(g_2,g_1) \Omega_q(g_1,g_2) q(g_1)\eta(g_1) q(g_2)\eta(g_2),
	\]
which is true by \eqref{assocbihomomega} and the linearity of $\eta$. 
Finally, by \autoref{symmetrycorollary} the weak quadratic form $q\eta$ is symmetric with respect to $g_0^{-2}$, which shows that $\theta_{q\eta}$ is ribbon with respect to the duality $(-)^{g_0^{-2}}$.
Thus, we showed that $(\alpha_{\psi_q}, \gamma_{\Omega_q}, \theta_{q\eta}, g_0^{-2}) \in \WRMS$.

If $(q, \eta, g_0) \sim (q', \eta', g_0')$, there exists by definition some group automorphism $f \in \textnormal{Aut}(G)$, such that $q\eta=f^{\star}(q'\eta')$ and $f(g_0) = g_0'$. This implies that $\lbrack (\psi_q, \Omega_q) \rbrack = \lbrack (f^{\star}\psi_{q'}, f^{\star}\Omega_{q'}) \rbrack$. Thus, there exists some normalized abelian $2$-cochain $\kappa \in C^2(G, \mathbbm{k}^{\times})$ such that
	$(\psi_q, \Omega_q) = (f^{\star}(\psi_{q'}) \cdot d\kappa, f^{\star}(\Omega_{q'}) \cdot \kappa_{\textnormal{comm}}).$ This gives rise to a candidate for a braided monoidal functor
	\[ (F, \Phi, \psi): (G\textnormal{-vect}_{\mathbbm{k}}, \otimes, \mathbbm{k}, \alpha_{\psi_q}, \lambda, \varrho, \gamma_{\Omega_q}) \rightarrow (G\textnormal{-vect}_{\mathbbm{k}}, \otimes, \mathbbm{k},  \alpha_{\psi_{q'}}, \lambda, \varrho, \gamma_{\Omega_{q'}}).
	\]
 defined by $F(\mathbbm{k}_{g}) := \mathbbm{k}_{f(g)}$,\ $F(\textnormal{id}_{\mathbbm{k}_{g}}) = \textnormal{id}_{\mathbbm{k}_{f(g)}}$,\ $\Phi_{g_1,g_2} := \kappa(g_1,g_2) \textnormal{id}_{\mathbbm{k}_{f(g_1g_2)}}$ and $\psi = \textnormal{id}_{\mathbbm{k}}$. By definition $F(\mathbbm{k}_{1_G}) = \mathbbm{k}_{f(1_G)} = \mathbbm{k}_{1_G}$. Moreover, 
	using the definitions of $d\kappa, \kappa_{\textnormal{comm}}$ and that $\kappa$ is normalized, one checks that this is indeed
		is a braided monoidal functor. The equality $F((\theta_{q\eta})_{\mathbbm{k}_g}) = (\theta_{q'\eta'})_{F(\mathbbm{k}_g)}$ is witnessed by $q\eta(g) = (q'\eta')(f(g))$.  This shows that $(\alpha_{\psi_q},\gamma_{\Omega_q}, \theta_{q\eta}, g_0^{-2}) \sim  (\alpha_{\psi_{q'}},\gamma_{\Omega_{q'}}, \theta_{q'\eta'}, (g_0')^{-2})$.
	
	Conversely, let $(\alpha, \gamma, \theta, g_0) \in \WRMS$. We then find for any $g,g_1, g_2, g_3 \in G$ scalars $\psi_{\alpha}(g_1, g_2, g_3), \Omega_{\gamma}(g_2,g_1), \Theta_{\theta}(g) \in \mathbbm{k}^{\times}$, such that
	\[ \alpha((1_{g_1} \otimes 1_{g_2}) \otimes 1_{g_3}) = \psi_{\alpha}(g_1, g_2, g_3) 1_{g_1} \otimes (1_{g_2} \otimes 1_{g_3}),\]
	\[ \gamma(1_{g_1} \otimes 1_{g_2}) = \Omega_{\gamma}(g_2, g_1) 1_{g_2} \otimes 1_{g_1}, \]
	\[ \theta(1_g) = \Theta_{\theta}(g) 1_g. \]
	One shows that $(\psi_{\alpha}, \Omega_{\gamma}) \in Z^3_{\textnormal{ab}}(G, \mathbbm{k} ^{\times})$ in the following way. Firstly, that $\psi_{\alpha}, \Omega_{\gamma}$ are normalized follows from the triangle axiom; that $d\psi = 1$ follows from the pentagon axiom; and finally, $\eqref{cocycleequation}$ follows from the hexagon axiom. Moreover, the twist condition of $\theta$ implies that $\beta_{\Theta_{\theta}} = \beta_{q_{\Omega_{\gamma}}}$. Since by \autoref{emiso} it holds $q_{\Omega_{\gamma}} \in \QF$, it follows that $\beta_{q_{\Omega_{\gamma}}}$, and thus $\beta_{\Theta_{\theta}}$, is bilinear. Moreover, the ribbon condition of $\theta$ with respect to the duality $(-)^{g_0}$ implies that $\Theta_{\theta}(g^{-1}g_0) = \Theta_{\theta}(g)$ for all $g \in G$. Combined, this shows that $(\Theta_{\theta},g_0) \in \WSQF$ . As a consequence, \autoref{samebeta} implies the existence of some character $\eta \in \charac$, such that $q_{\Omega_{\gamma}} \eta = \Theta_{\theta}$ and $\frac{\Theta_{\theta}(g)}{q_{\Omega_{\gamma}}(g)} = \eta(g) = \frac{q_{\Omega_{\gamma}}(g)}{\Theta_{\theta}(gg_0)}$.
	Using that $\Theta_{\theta}(g_0) = \Theta_{\theta}(g_0^{-1}g_0) = \Theta_{\theta}(1) = 1$ in the second to last step, it follows
	\begin{align*}
		\beta_{q_{\Omega_{\gamma}}}(g,g_0^{-\frac{1}{2}}) 
		&= \beta_{\Theta_{\theta}}(g,g_0)^{-\frac{1}{2}} 
		= \big( \frac{\Theta(gg_0)}{\Theta_{\theta}(g)\Theta_{\theta}(g_0)} \big)^{-\frac{1}{2}}
		= \big( \frac{q_{\Omega_{\gamma}}(g)^2}{\Theta_{\theta}(g)^2 \Theta_{\theta}(g_0)} \big)^{-\frac{1}{2}} \\
		&= \frac{\Theta_{\theta}(g)\Theta_{\theta}(g_0)^{\frac{1}{2}}}{q_{\Omega_{\gamma}}(g)} 
		= \frac{\Theta_{\theta}(g)}{q_{\Omega_{\gamma}}(g)} = \eta(g).
	\end{align*}
	This shows that $(q_{\Omega_{\gamma}}, \eta, g_0^{-\frac{1}{2}}) \in \WRQF$.

	If $(\alpha, \gamma, \theta, g_0) \sim (\alpha', \gamma', \theta', g_0')$ there exists by definition a braided monoidal equivalence
	\[
	(F, \Phi, \phi): (\Gvectc, \otimes, \mathbbm{k}, \alpha, \lambda, \varrho, \gamma) \cong (\Gvectc, \otimes, \mathbbm{k}, \alpha', \lambda, \varrho, \gamma').
	\]
	such that $F(\theta_{\mathbbm{k}_g}) = \theta'_{F(\mathbbm{k}_g)}$.
	Defining $f: G \rightarrow G$ such that $F(\mathbbm{k}_{g}) \cong \mathbbm{k}_{f(g)}$ yields an automorphism $f \in \textnormal{Aut}(G)$. Indeed, the invertibility of $f$ is induced by the invertibility of $F$, and the linearity of $f$ follows from
	\[ \mathbbm{k}_{f(g_1)f(g_2)} \cong \mathbbm{k}_{f(g_1)} \otimes \mathbbm{k}_{f(g_2)} \cong F(\mathbbm{k}_{g_1}) \otimes F(\mathbbm{k}_{g_2}) \cong F(\mathbbm{k}_{g_1} \otimes \mathbbm{k}_{g_2}) \cong F(\mathbbm{k}_{g_1g_2}) \cong \mathbbm{k}_{f(g_1g_2)}. \]
	The equality $F(\theta_{\mathbbm{k}_g}) = \theta'_{F(\mathbbm{k}_g)}$ yields $\Theta_{\theta}(g) = \Theta_{\theta'}(f(g))$, which translates to $q_{\Omega_{\gamma}}\eta(g) = f^{\star}(\Omega_{\gamma'}\eta')(g)$. Moreover, the isomorphism $F((\mathbbm{k}_g)^{g_0^{-\frac{1}{2}}}) \cong F(\mathbbm{k}_g)^{(g_0')^{-\frac{1}{2}}}$ implies that $f(g_0^{-\frac{1}{2}})$ $= (g_0')^{-\frac{1}{2}}$. Together this shows $(q_{\Omega_{\gamma}}, \eta, g_0^{-\frac{1}{2}}) \sim (q_{\Omega_{\gamma'}}, \eta', (g_0')^{-\frac{1}{2}})$.
	
	After a careful consideration one understands that both construction are mutually inverse. 
\end{proof}

This allows us to deduce \autoref{ribbonqf}, i.e. $(\QF \oplus \textnormal{Hom}(G, \lbrace \pm 1 \rbrace))/\sim\ \cong \RMS/\sim$, from the composition
$(\QF \oplus \textnormal{Hom}(G, \lbrace \pm 1 \rbrace))/\sim\ \rightarrow \WSQF/\sim\ \rightarrow \WRQF/\sim\ \rightarrow \WRMS/\sim\ \rightarrow \RMS/\sim$ given by
\[ \lbrack (q, \eta) \rbrack \mapsto \lbrack (q\eta, 1_G)\rbrack \mapsto \lbrack (q, \eta, 1_G) \rbrack \mapsto \lbrack (\alpha_{\psi_q}, \gamma_{\Omega_q}, \theta_{q\eta}, 1_G) \rbrack \mapsto \lbrack (\alpha_{\psi_q}, \gamma_{\Omega_q}, \theta_{q\eta}) \rbrack . \]

\section{Topological vector spaces}	

This section investigates a duality of topological vector spaces based on work of Barr \cite{barr2000autonomous}. Let $\mathbbm{k}$ be any field. If not otherwise stated, $\mathbbm{k}$-vector spaces are possibly infinitely dimensional.

\subsection{The categories $\chu$ and $\TVS$}

\begin{definition}[Pair]
\label{defpair}
	A \textit{pair} is a triple $(V, W, \langle - , - \rangle)$ of $\mathbbm{k}$-vector spaces $V, W$, together with a linear map $\langle -, - \rangle: V \otimes W \rightarrow \mathbbm{k}$ called \textit{pairing}. A \textit{morphism of pairs} $(V_1, W_1, \langle -,- \rangle_1)$ and $(V_2, W_2, \langle -,- \rangle_2)$ is a tuple $(f,g)$ of linear maps $f: V_1 \rightarrow V_2$ and $g: W_2 \rightarrow W_1$, such that
	\begin{equation}
	\label{pair}
		\langle f(v), w \rangle_2 =\ \langle v, g(w) \rangle_1
	\end{equation}
	for all $v \in V_1$ and $w \in W_2$.
\end{definition}
A canonical pair is given by $\mathbbm{k} := (\mathbbm{k}, \mathbbm{k}, \cdot)$ with the multiplication as pairing.

\begin{definition}[Separated, extensional]
	A pair $(V, W, \langle - , - \rangle)$ is called \textit{separated}, if the induced map $V \rightarrow W^{\star}$, $v \mapsto \langle v , - \rangle$  is injective. A pair $(V, W, \langle - , - \rangle)$ is called \textit{extensional}, if the induced map $W \rightarrow V^{\star}$, $w \mapsto \langle - , w \rangle$ is injective.
\end{definition}

\begin{definition}[Category $\chu$]
\label{chudef}
	The category of separated and extensional pairs, and morphism of pairs is denoted $\chu = \chu(\textnormal{Vect}, \mathbbm{k})$ due to P.H. Chu.
\end{definition}

Let $X$ be a set and $\mathcal{T} \subseteq \mathcal{P}(X)$ a family of subsets of $X$. $\mathcal{T}$ is called \textit{topology} on $X$, if $\emptyset, X \in \mathcal{T}$; any union of elements of $\mathcal{T}$ is an element of $\mathcal{T}$; and the intersection of finitely many elements of $\mathcal{T}$ is an element of $\mathcal{T}$. A \textit{neighbourhood} of $x \in X$ is a subset $U \subseteq X$, such that there exists a set $T \in \mathcal{T}$ with $x \in T \subseteq U$. The collection of neighbourhoods of $x \in X$ is denoted by $\mathcal{U}(x)$. A \textit{neighbourhood basis} of $x \in X$ is a subset $\mathcal{B}(x) \subseteq \mathcal{U}(x)$ of neighbourhoods of $x$, such that for all $U \in \mathcal{U}(x)$ there exists a $B \in \mathcal{B}(x)$ with $B \subseteq U$. A topological space $(X, \mathcal{T})$ is called \textit{Hausdorff}, if for every $x,y \in X$, $x \not = y$, there exist neighbourhoods $U \in \mathcal{U}(x)$ of $x$, and $V \in \mathcal{U}(y)$ of $y$, such that $U \cap V = \emptyset$.

A \textit{topological field} is a field $\mathbbm{k}$ endowed with a topology, such that the addition and the multiplication are continuous as maps $\mathbbm{k} \times \mathbbm{k} \rightarrow \mathbbm{k}$, where $\mathbbm{k} \times \mathbbm{k}$ carries the product topology. A \textit{topological vector space} is a vector space $V$ over a topological field $\mathbbm{k}$, endowed with a topology, such that the vector addition $V \times V \rightarrow V$ and the scalar multiplication $\mathbbm{k} \times V \rightarrow V$ are continuous, where the products carry again the product topology.

\begin{definition}[Convex, balanced, absorbing]
	Let $V$ be a vector space over $\mathbbm{k} \in \lbrace \mathbbm{R}, \mathbbm{C} \rbrace$. A subset $C \subseteq V$ is called
	\begin{enumerate}
		\item \textit{convex}, if for all $x,y \in C$ and $0 \leq t \leq 1$, it follows $tx + (1-t)y \in C$;
		\item \textit{{balanced}}, if for all $x \in C, \lambda \in \mathbbm{k}$, with $\vert \lambda \vert \leq 1$, it follows $\lambda x \in C$;
		\item \textit{absorbing}, if for all $x \in V$ there exists a $r \in \mathbbm{R}$, such that for all $\lambda \in \mathbbm{k}$ with $\vert \lambda \vert \geq r$, it follows $x \in \lambda C$.
	\end{enumerate}
\end{definition}

\begin{definition}[Locally convex]
\label{locallyconvexdef1}
	A topological vector space $V$ is called \textit{locally convex}, if $0 \in V$ has a neighbourhood basis consisting of balanced absorbing convex sets.
\end{definition}

\begin{definition}[$\TVS$]
\label{tvsdef}
	The category of Hausdorff locally convex topological vector spaces and continuous linear maps is denoted by $\TVS$.
\end{definition}

Note that Hausdorff locally convex topological vector spaces are sometimes simply called topological vector spaces, since the preconditions are considered standard across functional analysis.

The categories $\chu$ and $\TVS$ can be related in the following way. Let $V \in \TVS$ be a Hausdorff locally convex topological vector space with topological dual 
\[V' := \lbrace f \in V^{\star} \mid f \textnormal{ continuous} \rbrace . \] We can define a pair 
\[ T(V) := (V, V', \textnormal{eval}) \] with pairing given by the evaluation map. Note that the topology structure of $V$ is encoded in the topological dual $V'$. By construction the induced pairing $V' \rightarrow V^{\star}$ is injective, i.e. $T(V)$ extensional. Moreover, the Hahn-Banach separation theorem states that in a Hausdorff locally convex topological vector space points can be separated, i.e. for all $x,y \in V$ with $x \not = y$ there exists a continuous linear functional $f \in V'$, such that $f(x) \not = f(y)$. In particular for every $x \not = 0$, there exists a linear functional $f \in V'$, such that $f(x) \not = 0$. This shows that $T(V)$ is also separated, and thus, $T(V) \in \chu$. Defining in addition $T(f) = (f, f^{\star}$) on morphisms yields a functor
\begin{equation}
\label{Tfunctordef}
	T: \TVS \rightarrow \chu.
\end{equation}

In the next section we will introduce subcategories of $\TVS$, such that $T$ becomes an equivalence.

\subsection{Weakly and Mackey topologized vector spaces}

\begin{definition}[Dual topology]
\label{dualtop}
	Let $(V,W,\langle - , - \rangle) \in \chu$ be a separated and extensional pair. A \textit{dual topology} on $V$ is a topology $\mathcal{T}$, such that $(V, \mathcal{T}) \in \TVS$ and such that the injective map $W \rightarrow V^{\star}$ induces an isomorphism
	\[
	W \simeq (V, \mathcal{T})'.
	\]
\end{definition} 

For example, if $(V, \mathcal{T}) \in \TVS$ is a Hausdorff locally convex topological vector space, then $\mathcal{T}$ is by definition a dual topology with respect to the pair $T(V) = (V, (V, \mathcal{T})', \textnormal{eval}) \in \chu$. 

In general there could exist more than one dual topology to a pair. In the following we give the two extremal choices. 

Given two topologies $\mathcal{T}_1, \mathcal{T}_2$ on some set $X$, we say $\mathcal{T}_1$ is \textit{weaker} than $\mathcal{T}_2$, and $\mathcal{T}_2$ is \textit{stronger} than $\mathcal{T}_1$, if $\mathcal{T}_1 \subseteq \mathcal{T}_2$. 
 
\begin{definition}[Weak, strong topology]
\label{weakstrongdef}
	Let $(V, W, \langle - , - \rangle) \in \chu$ be a separated and extensional pair. The \textit{weak topology} $\sigma(V, W, \langle -,- \rangle)$ is the weakest dual topology on $V$ with respect to the pair $(V, W, \langle - , - \rangle)$. 
	The \textit{strong topology} $\tau(V, W, \langle -, - \rangle)$ is the strongest dual topology on $V$ with respect to the pair $(V, W, \langle - , - \rangle)$. If the pairing $\langle -,- \rangle$ is clear from the context, we abbreviate $\sigma(V, W) := \sigma(V, W, \langle -,- \rangle)$ and $\tau(V, W) := \tau(V, W, \langle -,- \rangle)$
\end{definition}



\begin{definition}[Weakly, Mackey spaces]
\label{weaklymackeydef}
	Let $V \in \TVS$ be a topological vectorspace and $V' \in \Vect$ its topological dual. The topology $\sigma(V, V')$ is called the \textit{weak topology on} $V$. The topological vector space $V_{\sigma} := (V, \sigma(V, V')) \in \TVS$ is called \textit{weakly topologized}.
  The topology $\sigma(V', V)$ is called the \textit{weak}-$^\star$ \textit{topology} on $V'$ and we abbreviate $(V')_{\sigma} := (V', \sigma(V', V))$. The topology $\tau(V, V')$ is called the \textit{Mackey topology on} $V$. The topological vector space $V_{\tau} := (V, \tau(V, V')) \in \TVS$ is called \textit{Mackey space}.
\end{definition}

The \textit{Mackey-Arens Theorem} (cf. \cite{mackey1946convex}, \cite{bourbaki1977}) ensures that the Mackey topology and the weak topology exist. 

The weak topology $\sigma(V, V')$ on $V \in \TVS$ is an example of an initial topology. In particular, it enjoys the following universal property:
\begin{quote}
A function $f: U \rightarrow V_{\sigma}$ from some topological vector space $U$ to $V$ is continuous if and only if $v' \circ f$ is continuous for all continuous linear functionals $v' \in V'$, 
\begin{equation}
\label{weakuniversal}
	\begin{tikzcd}[row sep= large, column sep = large]
	V_{\sigma} \arrow{r}{v'} &\mathbbm{k} \\
	U \arrow{u}{f} \arrow{ur}[right]{v' \circ f}
	\end{tikzcd}.
\end{equation}
\end{quote}

\begin{definition}[Categories $\TVSw$, $\TVSm$]
\label{mackeyweaklycatdef}
	Let $\TVSw$ denote the full subcategory of $\TVS$ of weakly topologized vector spaces and $\TVSm$ the full subcategory of Mackey spaces.
\end{definition}

The functor $T: \TVS \rightarrow \chu$, defined in \eqref{Tfunctordef}, restricts to functors $T_w: \TVSw \rightarrow \chu$ and $T_m: \TVSm \rightarrow \chu$. Conversely, define functors $L: \chu \rightarrow  \TVS_m$ and $R: \chu \rightarrow \TVS_w$ in the following way. For $(V, W, \langle - , - \rangle) \in \chu$ let
\[
L((V, W, \langle - , - \rangle)) := (V, \tau(V, W)) \quad \textnormal{and} \quad R((V, W, \langle - , - \rangle)) := (V, \sigma(V, W))
\]
and on morphism let $L((f,g)) = R((f,g)) := f$. The following result is due to \cite[Sec. 4.5]{barr2000autonomous}.

\begin{theorem}
\label{chutvsequivalence}
The functors in the diagram
\begin{equation}
\label{adjointequivalencechudiagram}
	\begin{tikzcd}[row sep= large, column sep = large]
 &\chu \arrow[shift left=0.5ex]{dl}{L} \arrow[shift right=0.5ex, swap]{dr}{R} \\
\TVSm \arrow[shift left=0.5ex]{ur}{T_m} && \TVSw  \arrow[shift right=0.5ex, swap]{ul}{T_w}
	\end{tikzcd}
\end{equation}
provide quasi-inverse adjoint equivalences. 
\end{theorem} 

\subsection{$\star$-autonomous structure of $\chu$}

For every pair $\mathcal{V}=(V, W, \langle -,- \rangle) \in \chu$ there exists a \textit{dual} pair 
\[ \mathcal{V}^{\star} :=(W, V, \langle -,- \rangle) \in \chu \]
 with pairing induced by the pairing of $\mathcal{V} \in \chu$ using the symmetry of the tensor product. Similarly for a morphism $(f,g): \mathcal{V} \rightarrow \mathcal{W}$ between pairs $\mathcal{V}, \mathcal{W} \in \chu$ one defines a dual morphism
 \[
 (f,g)^{\star} := (g,f).
 \]
 By the very definition one thus has 
 \[ \mathcal{V}^{\star \star} = \mathcal{V} \quad \textnormal{and} \quad (f,g)^{\star \star} = (f,g), \]
 similar to the case of finite dimensional vector spaces. 
  
  Note that the collection of morphisms $\textnormal{Hom}_{\chu}(\mathcal{V}, \mathcal{W})$ between pairs $\mathcal{V}, \mathcal{W} \in \chu$ can be equipped in a canonical way with a vector space structure. 	
  	
\begin{definition}
\label{chuinternalhom}
	Given pairs $\mathcal{V}_1=(V_1, W_1, \langle - , - \rangle_1)$ and $\mathcal{V}_2=(V_2, W_2, \langle - , - \rangle_2)$ in $\chu$, let $\underline{\textnormal{Hom}}(\mathcal{V}_1, \mathcal{V}_2)$ be the pair $(\textnormal{Hom}_{\chu}(\mathcal{V}_1, \mathcal{V}_2), V_1 \otimes W_2, \langle - , - \rangle)$  with pairing 
	\[ \langle (f,g), v \otimes w \rangle := \langle f(v), w \rangle_2 = \langle v, g(w) \rangle_1 \]
	for $(f,g) \in \textnormal{Hom}_{\chu}(\mathcal{V}_1,\mathcal{V}_2)$ and $v \in V_1, w \in W_2$.
\end{definition}

One can show (cf. \cite[Sec. 4.2]{barr2000autonomous}) that the pair $\underline{\textnormal{Hom}}(\mathcal{V}_1, \mathcal{V}_2)$ is separated and extensional, i.e. $\underline{\textnormal{Hom}}(\mathcal{V}_1, \mathcal{V}_2) \in \chu$.

In fact, Barr showed a variety of more properties (cf. \cite[Sec. 4.1]{barr2000autonomous}), as listed in the following result.

\begin{theorem}
For any pairs $\mathcal{U}, \mathcal{V}, \mathcal{W} \in \chu$ it holds
		\begin{equation}
		\label{preautonom2}
 			\textnormal{Hom}_{\chu}(\mathbbm{k}, \underline{\textnormal{Hom}}( \mathcal{V}, \mathcal{W})) = \textnormal{Hom}_{\chu}( \mathcal{V},  \mathcal{W});
		 \end{equation}
		\begin{equation}
		\label{theorem2}
 	\underline{\textnormal{Hom}}(\mathcal{U}, \underline{\textnormal{Hom}}( \mathcal{V},  \mathcal{W})) \cong\ \underline{\textnormal{Hom}}( \mathcal{V}, \underline{\textnormal{Hom}}(\mathcal{U},  \mathcal{W}));
 \end{equation}
		\begin{equation}
 	\label{theorem3}
 \underline{\textnormal{Hom}}(\mathcal{V},  \mathcal{W}) \cong \underline{\textnormal{Hom}}( \mathcal{W}^{\star},  \mathcal{V}^{\star});
 \end{equation}
		\begin{equation} 
		\label{theorem4}  \mathcal{V}^{\star} \cong \underline{\textnormal{Hom}}(\mathcal{V}, \mathbbm{k}).	
 \end{equation}
\end{theorem}

One can use these isomorphisms to equip $\chu$ with the structure of a symmetric monoidal closed category, as shown in the following result (cf. \cite[Theorem 2.6]{barr2000autonomous}).

\begin{lemma}
\label{inducedtensor}
	The category $\chu$ is symmetric monoidal closed with tensor product $\otimes$ defined as
	\begin{equation}
	\label{chutensor}
		\mathcal{V} \otimes  \mathcal{W} :=\ \underline{\textnormal{Hom}}(\mathcal{V},  \mathcal{W}^{\star})^{\star}
	\end{equation}
	for pairs $\mathcal{V}, \mathcal{W} \in \chu$ and internal hom $\underline{\textnormal{Hom}}(-,-)$ as defined in \autoref{chuinternalhom}.
\end{lemma}
\begin{proof}
For pairs $\mathcal{U}, \mathcal{V}, \mathcal{W} \in \chu$ we find
	\begin{align*}
		&\textnormal{Hom}_{\chu}(\mathcal{U} \otimes \mathcal{V}, \mathcal{W}) 
		\overset{\eqref{chutensor}}{=} \textnormal{Hom}_{\chu}(\underline{\textnormal{Hom}}(\mathcal{U}, \mathcal{V}^{\star})^{\star}, \mathcal{W}) 
		\overset{\eqref{preautonom2}}{=} \textnormal{Hom}_{\chu}(\mathbbm{k}, \underline{\textnormal{Hom}}(\underline{\textnormal{Hom}}(\mathcal{U}, \mathcal{V}^{\star})^{\star}, \mathcal{W})) \\
		&\overset{\eqref{theorem3}}{\cong}  \textnormal{Hom}_{\chu}(\mathbbm{k}, \underline{\textnormal{Hom}}(\mathcal{W}^{\star},\underline{\textnormal{Hom}}(\mathcal{U}, \mathcal{V}^{\star})))
		\overset{\eqref{theorem2}}{\cong}  \textnormal{Hom}_{\chu}(\mathbbm{k}, \underline{\textnormal{Hom}}(\mathcal{U},\underline{\textnormal{Hom}}(\mathcal{W}^{\star}, \mathcal{V}^{\star}))) \\
		&\overset{\eqref{theorem3}}{\cong}  \textnormal{Hom}_{\chu}(\mathbbm{k}, \underline{\textnormal{Hom}}(\mathcal{U},\underline{\textnormal{Hom}}(\mathcal{V}, \mathcal{W})))
		\overset{\eqref{preautonom2}}{=} \textnormal{Hom}_{\chu}(\mathcal{U},\underline{\textnormal{Hom}}(\mathcal{V}, \mathcal{W})).
	\end{align*}
	The symmetry follows from $\mathcal{V}^{\star \star} = \mathcal{V}$ and \eqref{theorem3}. The monoidal unit is given by the pair $\mathbbm{k} \in \chu$.
\end{proof}

In particular, \eqref{theorem4} implies that the pair $\mathbbm{k} \in \chu$ is a dualizing object, 

\[
\mathcal{V} = \mathcal{V}^{\star \star} \cong \underline{\textnormal{Hom}}(\underline{\textnormal{Hom}}(\mathcal{V}, \mathbbm{k}), \mathbbm{k}).
\]

\begin{lemma}
	The category $\chu$ can be equipped with the structure of a symmetric $\star$-autonomous category with tensor product \eqref{chutensor}, internal hom as in \autoref{chuinternalhom} and dualizing object $\mathbbm{k} \in \chu$.
\end{lemma}

\subsection{$\star$-autonomous structure of $\TVSw$ and $\TVSm$}

The equivalence between $\chu$, $\TVSw$ and $\TVSm$, as shown in \autoref{chutvsequivalence}, induces a $\star$-autonomous structure on the categories $\TVSw$ and $\TVSm$. In this section we will describe these two structures in a very concrete way. In the following $\otimes_{\mathbbm{k}}$ will denote the algebraic tensor product of $\mathbbm{k}$-vector spaces.

\begin{definition}
Let $\underline{\textnormal{Hom}}(-,-)$ and $(-)^{\star}$ denote the internal hom and duality of $\chu$, respectively. Moreover, let $R,L, T_m, T_w$ be the functors as in \eqref{adjointequivalencechudiagram}. 
\begin{itemize}
	\item $\underline{\textnormal{Hom}}_{w} := R \circ \underline{\textnormal{Hom}} \circ (T_w^{\textnormal{opp}(1)} \times T_w): \TVSw^{\textnormal{opp}(1)} \times \TVSw \rightarrow \TVSw$,
	\item $\underline{\textnormal{Hom}}_{m} := L \circ \underline{\textnormal{Hom}} \circ (T_m^{\textnormal{opp}(1)} \times T_m): \TVSm^{\textnormal{opp}(1)} \times \TVSm \rightarrow \TVSm$,
	\item $D_{w} := R^{\textnormal{opp}(1)} \circ (-)^{\star} \circ T_w: \TVSw \rightarrow \TVSw ^{\textnormal{opp}(1)}$,
	\item $D_{m} := L^{\textnormal{opp}(1)} \circ (-)^{\star} \circ T_m: \TVSm \rightarrow \TVSm ^{\textnormal{opp}(1)}$,
	\item $V \otimes_w^1 W := D_w(\underline{\textnormal{Hom}}_w(V, D_w(W))), \quad V,W \in \TVSw,$
	\item $V \otimes_m^1 W := D_m(\underline{\textnormal{Hom}}_m(V, D_m(W))), \quad V,W \in \TVSm$.
	\item $
	V \otimes_w^2 W := D_w^{-1}(D_w(W) \otimes_w^1 D_w(V)), \quad V,W \in \TVSw,
	$
	\item $
	V \otimes_m^2 W := D_m^{-1}(D_m(W) \otimes_w^1 D_m(V)), \quad V,W \in \TVSm.$
\end{itemize}
\end{definition}

\begin{corollary}
\label{tvsstarautonom}
		The categories $\TVS_m$ and $\TVS_w$ can be equipped with the structure of $\star$-autonomous categories. More precisely, there exist symmetric closed monoidal structures $(\TVS_m, \otimes_m^1, L(\mathbbm{k}), \underline{\textnormal{Hom}}_{m})$ and $(\TVS_w, \otimes_w^1, R(\mathbbm{k}), \underline{\textnormal{Hom}}_{w})$ with dualizing objects $(D_m \circ L)(\mathbbm{k})$ and $(D_w \circ R)(\mathbbm{k})$, respectively.
\end{corollary}

\begin{lemma}
\label{internalhomtvs}
The internal homs $\underline{\textnormal{Hom}}_{m}(-,-): \TVSm^{\textnormal{opp}(1)} \times \TVSm \rightarrow \TVSm$ and  \\$\underline{\textnormal{Hom}}_{w}(-,-): \TVSw^{\textnormal{opp}(1)} \times \TVSw \rightarrow \TVSw$ satisfy
\begin{align*}
	\underline{\textnormal{Hom}}_{w}(V,W) &\cong (L(V,W), \sigma(L(V,W), V \otimes_{\mathbbm{k}} W')), \quad V,W \in \TVSw,\\
	\underline{\textnormal{Hom}}_{m}(V,W) &\cong (L(V,W), \tau(L(V,W), V \otimes_{\mathbbm{k}} W')), \quad V,W \in \TVSm.
\end{align*}
Here $L(V,W)$ denotes the space of continuous linear maps $V \rightarrow W$. 
\end{lemma}
\begin{proof}
Using the definition of the internal hom in $\chu$, we deduce
\begin{align*}
	&(\underline{\textnormal{Hom}} \circ (T_w^{\textnormal{opp}(1)} \times T_w))(V, W) 
	= \underline{\textnormal{Hom}}((V,V', \textnormal{eval}),(W,W',\textnormal{eval})) \\
	&= (\textnormal{Hom}_{\chu}((V,V', \textnormal{eval}),(W,W', \textnormal{eval})), V \otimes_{\mathbbm{k}} W', \langle -,- \rangle).
\end{align*}
The pairing $\langle -,- \rangle$ is defined on $(f,g) \in \textnormal{Hom}_{\chu}((V,V', \textnormal{eval}),(W,W', \textnormal{eval}))$ and $v \otimes w' \in V \otimes_{\mathbbm{k}} W'$ as
	\[
	\langle (f,g), v \otimes w' \rangle = w'(f(v)) = g(w')(v).
	\]
By definition $g = f^*$, the precomposition with $f$. Moreover, $f: V \rightarrow W$ is continuous, since $w' \circ f = g(w') \in V'$ is continuous for all $w' \in W'$ (cf. the universal property \eqref{weakuniversal}). Thus the maps $(f,g) \mapsto f$  and $\textnormal{id}_{V \otimes W'}$ define an isomorphism
\[
(\textnormal{Hom}_{\chu}((V,V', \textnormal{eval}),(W,W', \textnormal{eval})), V \otimes_{\mathbbm{k}} W', \langle -,- \rangle) \cong (L(V,W), V \otimes_{\mathbbm{k}} W', \langle -,- \rangle_L)
\]
of pairs, where $L(V,W)$ denotes the space of continuous functions from $V$ to $W$, and the pairing $\langle f, v \otimes w' \rangle_L$ is given as $w'(f(v))$. It follows
		\begin{align*}
		\underline{\textnormal{Hom}}_{w}(V,W) 
		&= (R \circ \underline{\textnormal{Hom}} \circ (T_w^{\textnormal{opp}(1)} \times T_w))(V, W) 
		\cong R(L(V,W), V \otimes_{\mathbbm{k}} W', \langle -,- \rangle_L) \\
		&= (L(V,W), \sigma(L(V,W), V \otimes_{\mathbbm{k}} W')).
	\end{align*}
This means that $\underline{\textnormal{Hom}}_{w}(V,W) \cong L(V,W)$ is equipped with the weakest topology, such that the induced injection $V \otimes_{\mathbbm{k}} W' \hookrightarrow L(V,W)^{\star}$ defined as  $v \otimes w' \mapsto (f \mapsto w'(f(v)))$ restricts to an isomorphism 
\[ V \otimes_{\mathbbm{k}} W' \cong L(V,W)'. \]
 The second claim follows analogously by replacing $R$ with $L$.
\end{proof}

\begin{lemma}
\label{dualitytvs}
The functors $D_m: \TVS_m \rightarrow \TVS_m^{\textnormal{opp}(1)}$ and $D_w: \TVS_w \rightarrow \TVS_w^{\textnormal{opp}(1)}$ satisfy
	\[
	D_{w}(V) = (V', \sigma(V', V)) = (V')_{\sigma} \quad \textnormal{and} \quad D_{m}(V) = (V', \tau (V', V)) = (V')_{\tau}.
	\] 
	Moreover, it holds $D_w^2 \cong \textnormal{id}_{\TVSw}$ and $D_m^2 \cong \textnormal{id}_{\TVSm}$.
\end{lemma}
\begin{proof}
	Let $V \in \TVSw$, then
	\begin{align*}
		D_{w}(V) &= (R^{\textnormal{opp}(1)} \circ (-)^{\star} \circ T_w)(V) = (R^{\textnormal{opp}(1)} \circ (-)^{\star})(V, V', \textnormal{eval}) \\
		&= R^{\textnormal{opp}(1)}(V', V, \textnormal{eval}) = (V', \sigma(V', V)) = (V')_{\sigma},
	\end{align*}
	i.e. $V'$ equipped with the weak-$^\star$ topology. The second claim follows analogously by replacing $R^{\textnormal{opp}(1)}$ with $L^{\textnormal{opp}(1)}$. 
	From the definition of the weak topology it follows 
\begin{align*}
	&D_w^2(V) = (((V')_{\sigma})')_{\sigma} \cong V_{\sigma} \cong V.
	\end{align*}
	Again, using the strongest topology instead of the weakest topology, one shows the remaining part.
\end{proof}

\begin{lemma}
\label{tensor1tvs}
The tensor products $\otimes^1_w: \TVS_w \times \TVS_w \rightarrow \TVS_w$ and $\otimes^1_m: \TVS_m \times \TVS_m \rightarrow \TVS_m$ satisfy
	\begin{align*}
		V \otimes_w^1 W &\cong (V \otimes_{\mathbbm{k}} W, \sigma(V \otimes_{\mathbbm{k}} W, L(V, (W')_{\sigma}))), \quad V,W \in \TVSw, \\
		V \otimes_m^1 W &\cong (V \otimes_{\mathbbm{k}} W, \tau(V \otimes_{\mathbbm{k}} W, L(V, (W')_{\tau}))), \quad V,W \in \TVSm.
	\end{align*} 
\end{lemma}
\begin{proof}
Using \autoref{internalhomtvs} and \autoref{dualitytvs} we compute
	\begin{align*}
		 V \otimes_w^1 W &= D_w(\underline{\textnormal{Hom}}_w(V, D_w(W))) 
	\cong D_w(\underline{\textnormal{Hom}}_w(V, (W')_{\sigma}) \\
	&\cong D_w \bigg( L(V, (W')_{\sigma}), \sigma \big (L(V, (W')_{\sigma}), V \otimes_{\mathbbm{k}}  ((W')_{\sigma})' \big) \bigg) \\
	&\cong D_w \bigg( L(V, (W')_{\sigma}), \sigma \big (L(V, (W')_{\sigma}), V \otimes_{\mathbbm{k}}  W \big) \bigg) \\
	&\cong \bigg( L(V, (W')_{\sigma}), \sigma \big (L(V, (W')_{\sigma}), V \otimes_{\mathbbm{k}}  W \big) \bigg)'_{\sigma}
	\end{align*}
	This shows that $V \otimes_w^1 W$ is isomorphic to the vector space $V \otimes_{\mathbbm{k}} W$ with the weakest topology such that the injection $L(V, (W')_{\sigma}) \hookrightarrow (V \otimes_{\mathbbm{k}} W)^{\star}$, $f \mapsto (v \otimes w \mapsto f(v)(w))$ restricts to an isomorphism
	\[
	L(V, (W')_{\sigma}) \cong (V \otimes_w^1 W)'.
	\]
	The second claim follows analogously by replacing $R$ with $L$. 
\end{proof}

\begin{lemma}
\label{tensor2tvs}
	The tensor products $\otimes^2_w: \TVS_w \times \TVS_w \rightarrow \TVS_w$ and $\otimes^2_m: \TVS_m \times \TVS_m \rightarrow \TVS_m$ satisfy
	\[
	V \otimes_w^2 W \cong (L(W'_{\sigma}, V), \sigma(L(W'_{\sigma}, V), W' \otimes_{\mathbbm{k}} V')), \quad V,W \in \TVSw;
	\]
	\[
		V \otimes_m^2 W \cong (L(W'_{\tau}, V), \sigma(L(W'_{\tau}, V), W' \otimes_{\mathbbm{k}} V')), \quad V,W \in \TVSm.
	\]
\end{lemma}
\begin{proof}
	Using \autoref{dualitytvs} and \autoref{tensor1tvs} we compute for $V,W \in \TVSw$,
	\begin{align*}
		&V \otimes_w^2 W = (W'_{\sigma} \otimes_w^1 V'_{\sigma})'_{\sigma} 
		\cong (W' \otimes_{\mathbbm{k}} V', \sigma(W' \otimes_{\mathbbm{k}} V', L(W'_{\sigma}, (V'_{\sigma})'_{\sigma})))'_{\sigma} \\
		&\cong (W' \otimes_{\mathbbm{k}} V', \sigma(W' \otimes_{\mathbbm{k}} V', L(W'_{\sigma}, V)))'_{\sigma}
		\cong (L(W'_{\sigma}, V), \sigma(L(W'_{\sigma}, V), W' \otimes_{\mathbbm{k}} V')).
	\end{align*}
	Analogously one validates the claim for the Mackey topology case.
\end{proof}

\begin{lemma}
	The two tensor products on $\TVS_w$ and $\TVS_m$, respectively, are related in the sense that 
	\[
(V \otimes^2_w W)'_{\sigma} \cong (W')_{\sigma} \otimes_w^1 (V')_{\sigma},  \quad V,W \in \TVSw;
\]
\[
(V \otimes^2_m W)'_{\tau} \cong (W')_{\tau} \otimes_m^1 (V')_{\tau},  \quad V,W \in \TVSm.
\]
\end{lemma}
\begin{proof}
	Follows immediately from the definitions.
\end{proof}

As a last step we would like to relate our two tensor products on $\TVS_w$ and $\TVS_m$, respectively, to the \textit{projective} and \textit{injective} tensor products (cf. \cite{treves1967topological}). Explicitly, the $\pi$-topology, or projective topology, on the projective tensor product $V \otimes_{\mathbbm{k}} W$ for $V,W \in \TVS$ is the strongest locally convex topology, such that the canonical bilinear map $V \times W \rightarrow V \otimes_{\mathbbm{k}} W$ is continuous, where $V \times W$ is given the product topology. Equipped with the $\pi$-topology, the space $V \otimes_{\mathbbm{k}} W$ is denoted by $V \otimes_{\pi} W$. Usefully, the projective topology enjoys the following universal property (cf. \cite[Prop. 43.4]{treves1967topological}):
\begin{quote}
	Let $V, W \in \TVS$. Then $V \otimes_{\pi} W$ equips $V \otimes_{\mathbbm{k}} W$ with the only topology such that for every $U \in \TVS$ the isomorphism of bilinear maps $V \times W \rightarrow U$ and linear maps $V \otimes_{\mathbbm{k}} W \rightarrow U$, provided by the universal property of the tensor product $V \otimes_{\mathbbm{k}} W$, restricts to an isomorphism of (jointly) continuous bilinear maps $V \times W \rightarrow U$ and continuous linear maps $V \otimes_{\pi} W \rightarrow U$.
\end{quote}

Our next result shows that a similar statement holds for the tensor products $\otimes^1_w$ and $\otimes^1_m$, when we restrict us to $U = \mathbbm{k}$ and separately continuous bilinear maps, instead of jointly continuous bilinear maps. Note that jointly continuous implies separately continuous, but not vice versa. Explicitly, the following is \textit{not} the universal property of the projective tensor product.

\begin{lemma}
\label{tvstensoruniversal}
The tensor products $\otimes^1_w: \TVS_w \times \TVS_w \rightarrow \TVS_w$ and $\otimes^1_m: \TVS_m \times \TVS_m \rightarrow \TVS_m$ satisfy the following property:
\begin{quote}
	Let $V, W \in \TVSw$. Then $V \otimes^1_w W$ ($V \otimes^1_m W$) equips $V \otimes_{\mathbbm{k}} W$ with the weakest (strongest) topology such that the isomorphism of bilinear maps $V \times W \rightarrow \mathbbm{k}$ and linear maps $V \otimes_{\mathbbm{k}} W \rightarrow \mathbbm{k}$, provided by the universal property of the tensor product $V \otimes_{\mathbbm{k}} W$, restricts to an isomorphism of separately continuous bilinear maps $V \times W \rightarrow \mathbbm{k}$ and continuous linear maps $V \otimes^1_w W \rightarrow \mathbbm{k}$ ($V \otimes^1_m W \rightarrow \mathbbm{k}$).
\end{quote}
\end{lemma}
\begin{proof}
We use that $R(R^{-1}(V) \otimes_{\chu} R^{-1}(W)) \cong V \otimes^1_w W$.
	We have $R^{-1}(V) = T_w(V) = (V, V', \textnormal{eval})$ for any $V \in \TVSw$. Thus we find
	\begin{align*}
		R^{-1}(V) \otimes_{\chu} R^{-1}(W) 
	&= (\underline{\textnormal{Hom}}((V, V',\textnormal{eval}), (W, W',\textnormal{eval})^{\star}))^{\star} \\
	&= (\underline{\textnormal{Hom}}( (V, V',\textnormal{eval}), (W', W,\textnormal{eval}) ) )^{\star}\\
	&= (\textnormal{Hom}_{\chu}((V, V',\textnormal{eval}), (W', W,\textnormal{eval})), V \otimes_{\mathbbm{k}} W, \langle -,- \rangle)^{\star}\\
	&= (V \otimes_{\mathbbm{k}} W, \textnormal{Hom}_{\chu}((V, V',\textnormal{eval}), (W', W,\textnormal{eval})),\langle -,- \rangle)
	\end{align*}
	with pairing 
	\[
	\langle v \otimes w, (f, g) \rangle = f(v)(w) = g(w)(v)
	\]
	for $v \in V, w\in W$ and $f: V \rightarrow W', \ g: W \rightarrow V'$. 
	By construction every $(f, g) \in \textnormal{Hom}_{\chu}((V, V',\textnormal{eval}), (W', W,\textnormal{eval}))$ induces a separately continuous bilinearform $T_{(f, g)}: V \times W \rightarrow \mathbbm{k}$ given by
	\[
	T_{(f, g)}(v,w) = f(v)(w) = g(w)(v).
	\]
	In fact $\textnormal{Hom}_{\chu}((V, V',\textnormal{eval}), (W', W,\textnormal{eval}))$ is \textit{isomorphic} to the the space $\mathcal{B}(V, W)$ of separately continuous bilinearforms $V \times W \rightarrow \mathbbm{k}$.
	Let $B(V, W)$ be the space of bilinear maps $V \times W \rightarrow \mathbbm{k}$. The universal property of the tensor product yields a linear map of vector spaces 
	\begin{equation}
	\label{univproptensorproduct}
		\varphi: B(V, W) \rightarrow (V \otimes_{\mathbbm{k}} W)^{\star}.
	\end{equation}
	Let $\iota: \mathcal{B}(V, W) \hookrightarrow B(V, W)$ be the inclusion map. The pairing of the pair 
	\begin{equation}
			\label{Rminus1}
	R^{-1}(V) \otimes_{\chu} R^{-1}(W) \cong (V \otimes_{\mathbbm{k}} W, \mathcal{B}(V, W), \langle -, - \rangle)
	\end{equation}
	is given by $
	\langle v \otimes w, f \rangle = (\varphi \circ \iota)(f)(v \otimes w)$.
	This induces an injection  
	\begin{equation}
	\label{partlyphi}
		\varphi \circ \iota: \mathcal{B}(V, W) \rightarrow (V \otimes_{\mathbbm{k}} W)^{\star}. 
	\end{equation}
	From \eqref{Rminus1} and the definition of $R$ it follows
	\[
	V \otimes^1_w W \cong R(R^{-1}(V) \otimes_{\chu} R^{-1}(W)) \cong (V \otimes_{\mathbbm{k}} W, \sigma(V \otimes_{\mathbbm{k}} W, \mathcal{B}(V, W))).
	\] 	
	By the very definition, $V \otimes^1_w W$ equips $V \otimes_{\mathbbm{k}} W$ with the weakest topology such that \eqref{partlyphi}, and thus $\varphi$, restricts to an isomorphism
	\[
	\mathcal{B}(V,W) \cong (V \otimes^1_w W)'.
	\]
	Analogously one shows the claim for the strongest topology.
\end{proof}

Aside from the jointly/separately continuous aspect, one of main differences between the tensor products $\otimes_w^1, \otimes_m^1$ and the projective tensor product $\otimes_{\pi}$ is that the property in \autoref{tvstensoruniversal} only holds for $\mathbbm{k} \in \TVS$, while it is true for every $U \in \TVS$ in the case of the projective tensor product.

To overcome this problem, it seems appropriate to generalise \autoref{defpair} and \autoref{dualtop} and then proceed in an analogous way as we did in this section. Let us sketch shortly the main idea; it is beyond the scope of this thesis to develop the whole theory.

\begin{definition}
Let $U \in \Vect$ be a $\mathbbm{k}$-vector space. 
An $U$-\textit{pair} is a triple $(V, W, \langle - , - \rangle)$ of $\mathbbm{k}$-vector spaces $V, W$, together with a linear map $\langle -, - \rangle: V \otimes_{\mathbbm{k}} W \rightarrow U$ called \textit{pairing}. A \textit{morphism of} $U$-\textit{pairs} $(V_1, W_1, \langle -,- \rangle_1)$ and $(V_2, W_2, \langle -,- \rangle_2)$ is a tuple $(f,g)$ of linear maps $f: V_1 \rightarrow V_2$ and $g: W_2 \rightarrow W_1$, such that
	\begin{equation}
		\langle f(v), w \rangle_2 =\ \langle v, g(w) \rangle_1 \in U
	\end{equation}
	for all $v \in V_1$ and $w \in W_2$. 
\end{definition}

\begin{definition}
	Let $U \in \TVS$ be a topological vector space.
	Let $(V,W,\langle - , - \rangle)$ be a separated and extensional $U$-pair. A \textit{dual topology} on $V$ w.r.t. to this $U$-pair is a topology $\mathcal{T}$, such that $(V, \mathcal{T}) \in \TVS$ and such that the injective map $W \rightarrow \textnormal{Hom}_{\mathbbm{k}}(V,U)$ induces an isomorphism
	\[
	W \simeq \lbrace f: (V, \mathcal{T}) \rightarrow U\  \textnormal{continuous} \rbrace \subseteq \textnormal{Hom}_{\mathbbm{k}}(V,U).
	\]
\end{definition}

\chapter{Appendix}
\section{Proofs}

\subsection*{Proof of \autoref{dualsuniquebicat}}

\begin{lemma*}
	Let $(\mathcal{C}, \otimes, \unit)$ be a strict monoidal bicategory and $A,B,B' \in \mathcal{C}$. Assume that $A \dashv B$ and $A \dashv B'$ witnessed by $\varepsilon_1, \eta_1$ and $\varepsilon_1', \eta_1'$. Then there exist $1$-morphisms $\alpha: B \rightarrow B'$ and $\beta: B' \rightarrow B$ as well as invertible $2$-morphisms $\beta \circ \alpha \cong \textnormal{id}_{B}$ and $\alpha \circ \beta \cong \textnormal{id}_{B'}$. The $1$-morphism $\alpha$ preserves evaluation and coevaluation morphisms in the sense that there exist invertible $2$-morphisms $\varepsilon_1' \circ (\textnormal{id}_{A} \otimes \alpha) \cong \varepsilon_1$ and $(\alpha \otimes \textnormal{id}_A) \circ \eta_1 \cong \eta_1'$. Furthermore, $\alpha$ is unique in the sense that for every other $1$-morphism $\tilde{\alpha}: B \rightarrow B'$ that preserves evaluation and coevaluation morphisms in this way, there exists an invertible $2$-morphism $\alpha \cong \tilde{\alpha}$. 
\end{lemma*}
\begin{proof}
	We denote $A,B,B'$ by red, blue and green, respectively. For the pair $A \dashv B$ let be given $1$-morphisms $\varepsilon_1 \equiv \includegraphics[height=0.5cm]{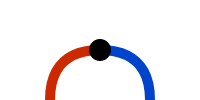}$, $\eta_1 \equiv \includegraphics[height=0.5cm]{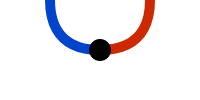}$ with $2$-morphisms $\varepsilon_2, \eta_2$ as in $\eqref{zigzagidentities2}$. Similarly, for the pair $A \dashv B'$ let be given $1$-morphisms $\varepsilon_1' \equiv \includegraphics[height=0.5cm]{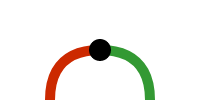}$,$\eta_1' \equiv \includegraphics[height=0.5cm]{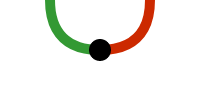}$ as well as $2$-morphisms $\varepsilon_2', \eta_2'$. Define $1$-morphisms $\alpha := \includegraphics[height=1cm]{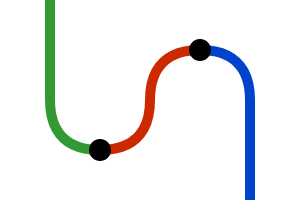}$ and $\beta := \includegraphics[height=1cm]{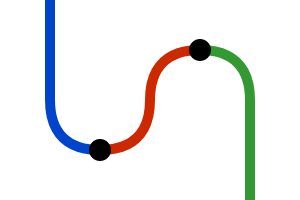}$. The following composition of $2$-morpshisms provides an invertible $2$-morphism $\beta \circ \alpha \cong \textnormal{id}_{B}$:
	\begin{align*}
		 \includegraphics[height=2cm]{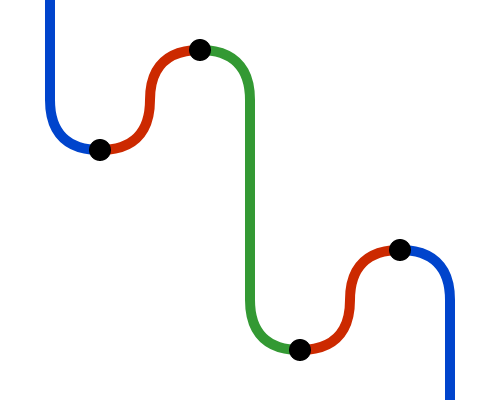} 
		 \overset{\eqref{intertwiner}}{\cong} \includegraphics[height=2cm]{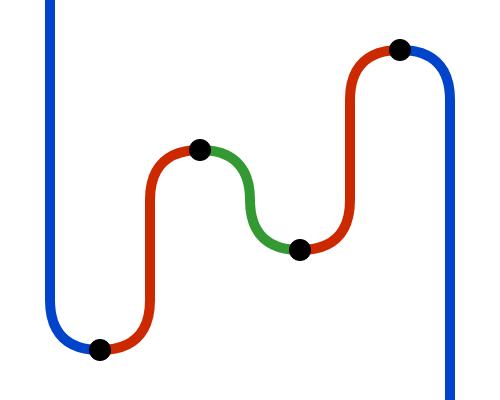} 
		 \overset{(\eta_2')^{-1}}{\cong} \includegraphics[height=1cm]{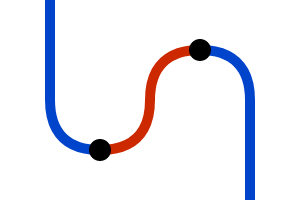} 
		 \overset{\varepsilon_2}{\cong} \includegraphics[height=0.5cm]{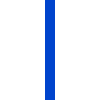}.
	\end{align*}
	Similarly one shows the existence of an invertible $2$-morphism $\alpha \circ \beta \cong \textnormal{id}_{B'}$. Furthermore, the $1$-morphism $\alpha$ preserves units and counits in the sense that there exist invertible compositions of $2$-morphisms
	\begin{align*}
		\varepsilon_1' \circ (\textnormal{id}_{A} \otimes \alpha) \equiv \includegraphics[height=1.5cm]{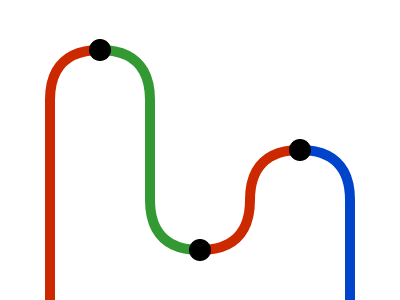} 
		\overset{\eqref{intertwiner}}{\cong} \includegraphics[height=1.5cm]{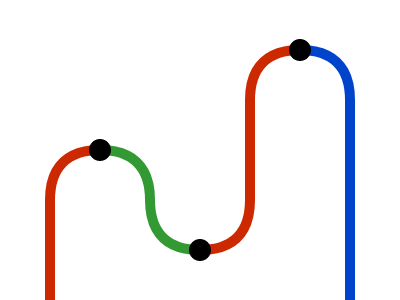}
		\overset{(\eta_2')^{-1}}{\cong} \includegraphics[height=0.5cm]{pics/dual/e1}
		\equiv \varepsilon_1,
		\\ 
		(\alpha \otimes \textnormal{id}_A) \circ \eta_1 \equiv \includegraphics[height=1.5cm]{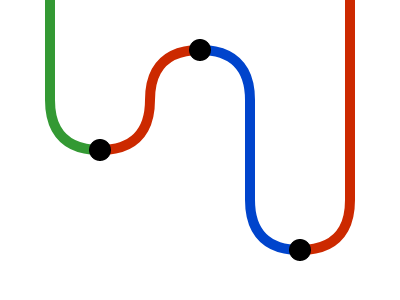} \overset{\eqref{intertwiner}}{\cong} \includegraphics[height=1.5cm]{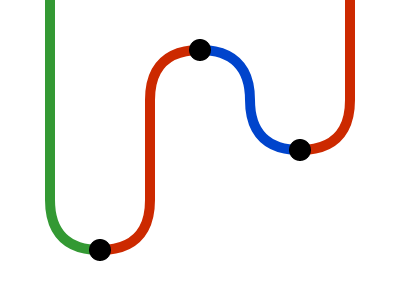} 
		\overset{\eta_2^{-1}}{\cong} \includegraphics[height=0.5cm]{pics/dual/n1b} 
		\equiv \eta_1'.
	\end{align*}
	Assume that there exists another $1$-morphism $\tilde{\alpha} \equiv \includegraphics[height=0.5cm]{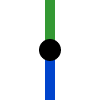} :  B \rightarrow B'$ together with invertible $2$-morphisms
\begin{equation}
\label{tildealphapreserveseps}
	\varepsilon_1' \circ (\textnormal{id}_{A} \otimes \tilde{\alpha}) \equiv \includegraphics[height=1cm]{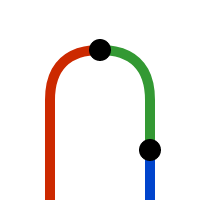} \cong \includegraphics[height=0.5cm]{pics/dual/e1} \equiv \varepsilon_1,
\end{equation}
\begin{equation}
\label{tildealphapreserveseta}
	(\tilde{\alpha} \otimes \textnormal{id}_A) \circ \eta_1 \equiv \includegraphics[height=1cm]{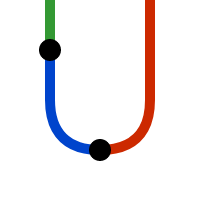}  \cong \includegraphics[height=0.5cm]{pics/dual/n1b} \equiv \eta_1'.
\end{equation}
Then we can define the following invertible $2$-morphisms between $\alpha$ and $\tilde{\alpha}$:
\begin{align*}
	\alpha \equiv \includegraphics[height=1cm]{pics/dual/alphadef} 
	\overset{\eqref{tildealphapreserveseps}}{\cong} \includegraphics[height=1.5cm]{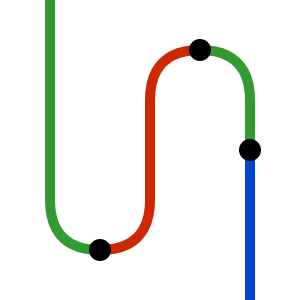}
\overset{\eqref{intertwiner}}{\cong}  \includegraphics[height=1.5cm]{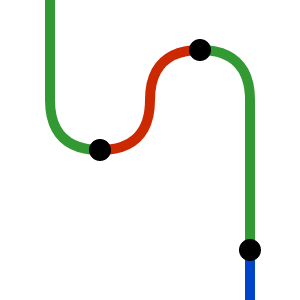}
	\overset{\varepsilon_2'}{\cong} \includegraphics[height=0.5cm]{pics/dual/3/tildealpha}
	\equiv \tilde{\alpha},
\end{align*}
\begin{align*}
	\alpha \equiv \includegraphics[height=1cm]{pics/dual/alphadef}
	\overset{\eqref{tildealphapreserveseta}}{\cong} \includegraphics[height=1.5cm]{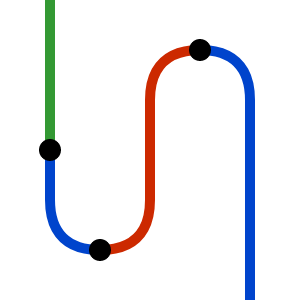}
\overset{\eqref{intertwiner}}{\cong}  \includegraphics[height=1.5cm]{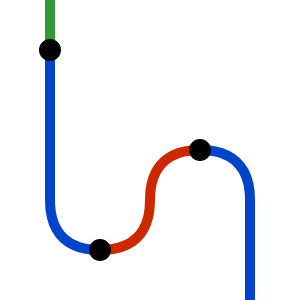}
	\overset{\varepsilon_2}{\cong} \includegraphics[height=0.5cm]{pics/dual/3/tildealpha}
	\equiv \tilde{\alpha}.
\end{align*}
\end{proof}

\subsection*{Proof of \autoref{lemmafrobpseudo}}
\begin{lemma*}
Let $(A,\ \mu \equiv \includegraphics[height=0.5cm]{pics/proof1/mu}: A \otimes A \rightarrow A,\ \eta \equiv \includegraphics[height=0.5cm]{pics/proof1/eta}: \unit \rightarrow A)$ be a pseudomonoid in a monoidal bicategory $(\mathcal{C}, \otimes, \unit)$. The following are equivalent:
	\begin{enumerate}
		\item There exists a morphism $\varepsilon \equiv \includegraphics[height=0.5cm]{pics/proof1/vareps}:A \rightarrow \unit$, such that $\sigma \equiv \includegraphics[height=1cm]{pics/proof1/muvareps}: A \otimes A \rightarrow \unit$ is an evaluation for a biexact pairing $A \dashv A$.
		\item There exists a morphism $\sigma \equiv \includegraphics[height=0.5cm]{pics/proof1/sigma}: A \otimes A \rightarrow \unit$, that is an evaluation for a biexact pairing $A \dashv A$ and an invertible $2$-cell
		\begin{equation*}
			\includegraphics[height=1cm]{pics/proof1/frobenius2} \cong \includegraphics[height=1cm]{pics/proof1/frobenius1}
		\end{equation*}
		\item There exists a pseudocomonoid structure $(A,\ \delta \equiv\includegraphics[height=0.5cm]{pics/proof1/delta}: A \rightarrow A \otimes A,\ \varepsilon \equiv\includegraphics[height=0.5cm]{pics/proof1/vareps}: A \rightarrow \unit)$ on $A$ and two invertible $2$-cells
			\begin{equation*}
				\includegraphics[height=1cm]{pics/proof1/frobzigzag1} \cong \includegraphics[height=1cm]{pics/proof1/deltacircmu} \cong \includegraphics[height=1cm]{pics/proof1/frobzigzag2}
			\end{equation*}
	\end{enumerate}
\end{lemma*}	
\begin{proof}
			$(1) \Rightarrow (2)$: We follow the arguments of \cite[Prop. 3.2]{street2004frobenius}. Define $\sigma := \includegraphics[height=1cm]{pics/proof1/muvareps}$, then it remains to construct an invertible $2$-cell \eqref{form}. By the definition of a pseudomonoid there exists an invertible $2$-morphism 
			$\alpha: \includegraphics[height=1cm]{pics/proof1/assoc1} \cong \includegraphics[height=1cm]{pics/proof1/assoc2}$ as in \eqref{pseudomonoassoc}. Consider the invertible $2$-cell 
			$\textnormal{id}_{\varepsilon} \alpha:  \includegraphics[height=1.5cm]{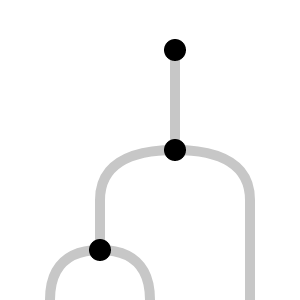} \cong \includegraphics[height=1.5cm]{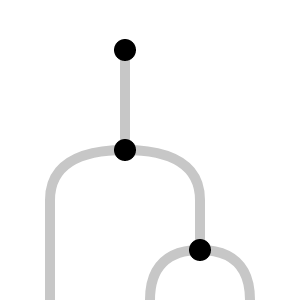}.
			$ Recalling $\sigma \equiv \includegraphics[height=1cm]{pics/proof1/muvareps}$ shows that the definition is appropriate.
			
		$(2) \Rightarrow (1):$ It suffices to construct a morphism $\varepsilon \equiv$ \includegraphics[height=0.5cm]{pics/proof1/vareps}, such that \includegraphics[height=1cm]{pics/proof1/muvareps} $\cong$ \includegraphics[height=0.5cm]{pics/proof1/sigma}. Let $\varepsilon := \includegraphics[height=1cm]{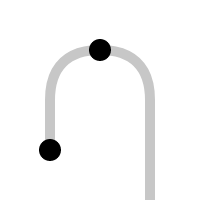}$. Then one computes
		\begin{align*}
			\includegraphics[height=1cm]{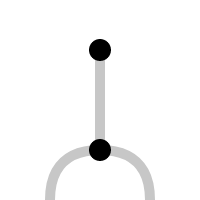} 
			= \includegraphics[height=1.5cm]{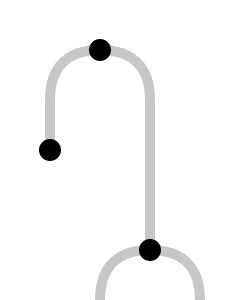} 
			\overset{\eqref{intertwiner}}{\cong} \includegraphics[height=1.5cm]{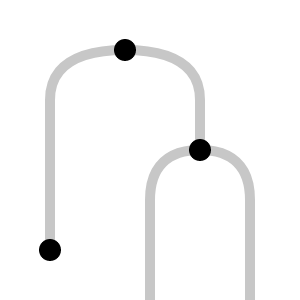} 
			\overset{\eqref{form}}{\cong} \includegraphics[height=1.5cm]{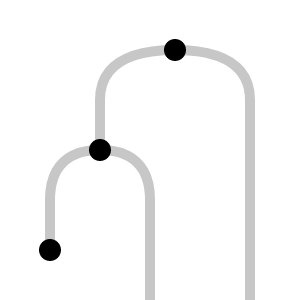} 
			\overset{\eqref{pseudomonounit}}{\cong}\includegraphics[height=0.5cm]{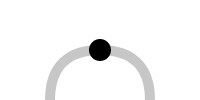}.
		\end{align*}
		\item $(3) \Rightarrow (2):$ This is a straightforward generalisation of the monoidal category case \cite{fuchs2009frobenius}. Define $\sigma :=  \includegraphics[height=0.5cm]{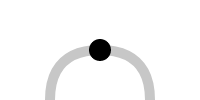} :=\includegraphics[height=1cm]{pics/proof1/muvareps}$. As in $(1) \Rightarrow (2)$ consider the invertible $2$-cell
			$\textnormal{id}_{\varepsilon} \alpha
			$ as a witness of \eqref{form}. The coevaluation for the biexact pairing $A \dashv A$ is given by $\includegraphics[height=0.5cm]{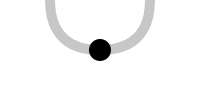} := \includegraphics[height=1cm]{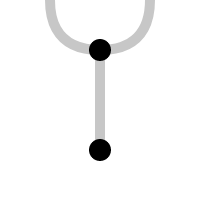}$. Indeed, we find the invertible $2$-cells		
		\begin{equation}
			\varepsilon_2:= \includegraphics[height=1cm]{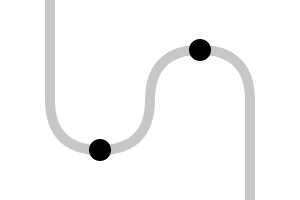} 
		= \includegraphics[height=2cm]{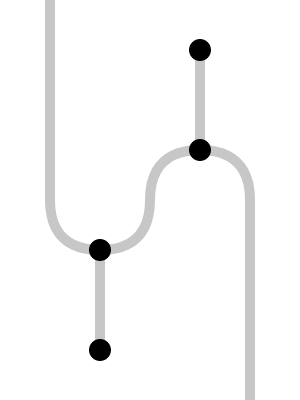} 
		\overset{\eqref{pseudocomonoid2cell}}{\cong} \includegraphics[height=2cm]{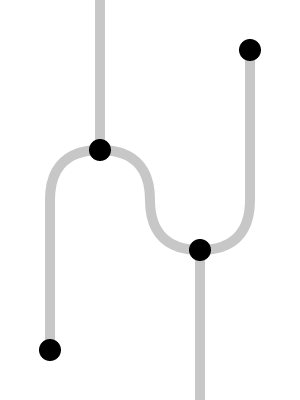} 
		\overset{\eqref{intertwiner}}{\cong} \includegraphics[height=2cm]{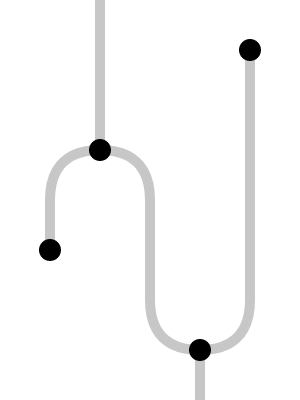} 
		\overset{\eqref{pseudomonounit}}{\cong} \includegraphics[height=1cm]{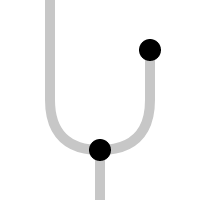}
		\overset{\eqref{pseudomonounit}}{\cong} \includegraphics[height=1cm, width=0.5cm]{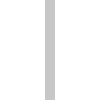},
		\end{equation}
		
		\begin{equation}
	\eta_2:= 	\includegraphics[height=1cm, width=0.5cm]{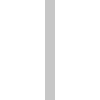}
		\cong \includegraphics[height=1cm]{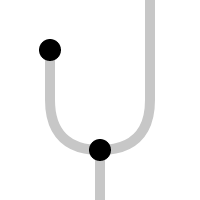}
			\overset{\eqref{pseudomonounit}}{\cong} \includegraphics[height=2cm]{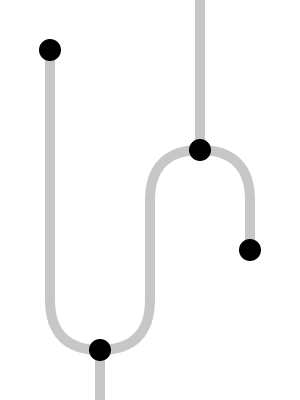}
		\overset{\eqref{intertwiner}}{\cong}  \includegraphics[height=2cm]{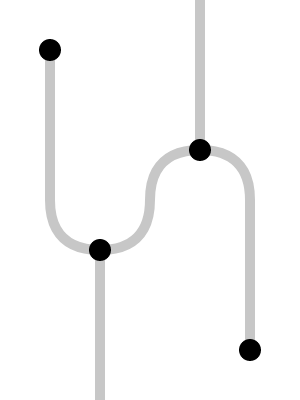} 
			\overset{\eqref{pseudocomonoid2cell}}{\cong} \includegraphics[height=2	cm]{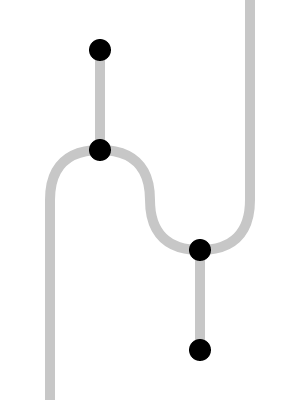} 
	= \includegraphics[height=1cm]{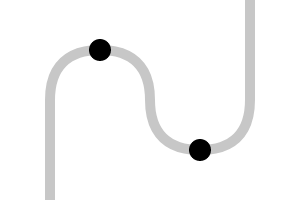} .
	\end{equation}		
	$(2) \Rightarrow (3):$ Denote the coevaluation of the biexact pairing $A \dashv A$ as  \includegraphics[height=0.5cm]{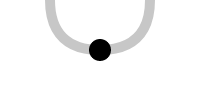}. Let $\varepsilon = \includegraphics[height=0.5cm]{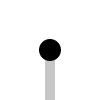}:= \includegraphics[height=1cm]{pics/proof1/varepsdef}$ and $\delta = \includegraphics[height=0.5cm]{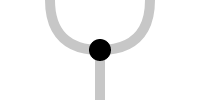} := \includegraphics[height=1cm]{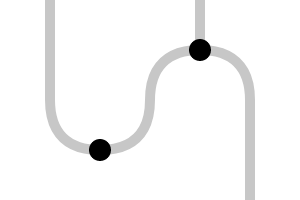} $. To obtain a pseudocomonoid structure $(A, \delta, \varepsilon)$ it remains to define appropriate invertible $2$-cells.
	Define one of the invertible $2$-cells witnessing the counitality as the composition
	\begin{equation*}
		\includegraphics[height=1cm]{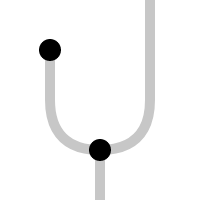}
		= \includegraphics[height=2cm]{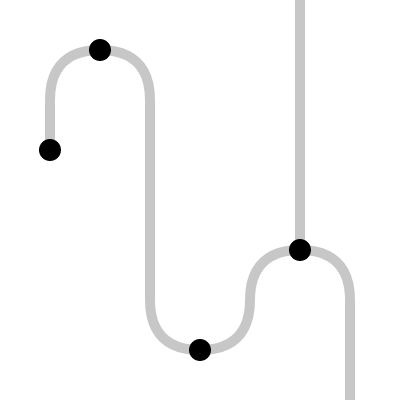}
		\overset{\eqref{intertwiner}}{\cong} \includegraphics[height=2cm]{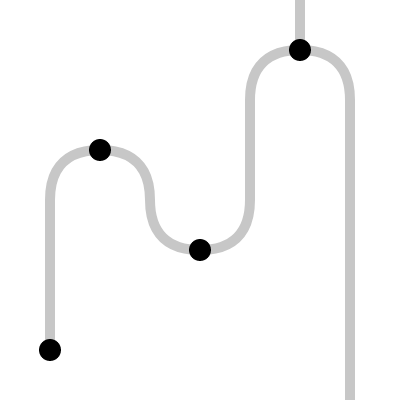}
		\overset{\eqref{zigzagidentities2}}{\cong} \includegraphics[height=1cm]{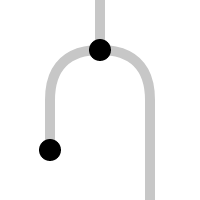}
		\overset{\eqref{pseudomonounit}}{\cong} \includegraphics[height=1cm, width=0.5cm]{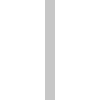},
	\end{equation*}
	and the second one analogously. Regarding the coassociativity, first note that we can find invertible $2$-morphisms 
	\begin{equation}
	\label{multiplrewrite}
		\includegraphics[height=0.5cm]{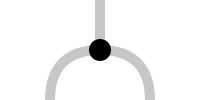}
	\overset{\eqref{zigzagidentities2}}{\cong} \includegraphics[height=1.5cm]{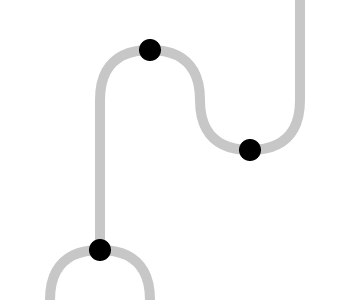}
	\overset{\eqref{intertwiner}}{\cong} \includegraphics[height=1.5cm]{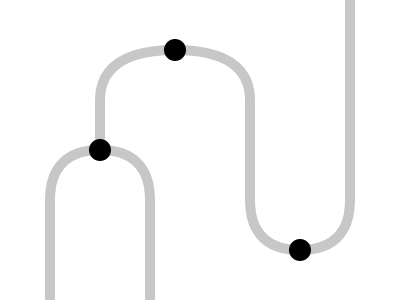}
	\overset{\eqref{form}}{\cong} \includegraphics[height=1.5cm]{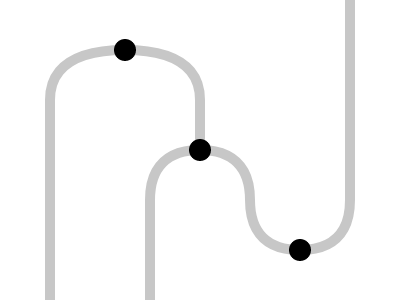}.
	\end{equation}
	Therefore, the comultiplication can be rewritten as
	\begin{align}
	\label{comultrewrite}
	\begin{split}
		\includegraphics[height=0.5cm]{pics/proof3/comult} 
		&= \includegraphics[height=1cm]{pics/proof3/comultdef}
		\overset{\eqref{multiplrewrite}}{\cong} \includegraphics[height=2cm]{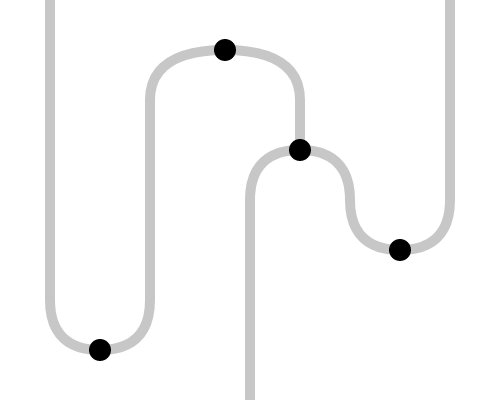}
		\overset{\eqref{intertwiner}}{\cong} \includegraphics[height=2cm]{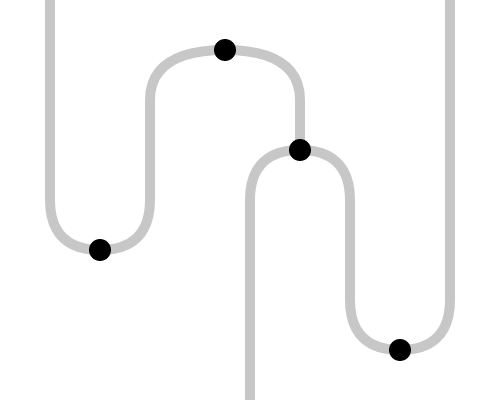}
		\overset{\eqref{intertwiner}}{\cong} \includegraphics[height=2cm]{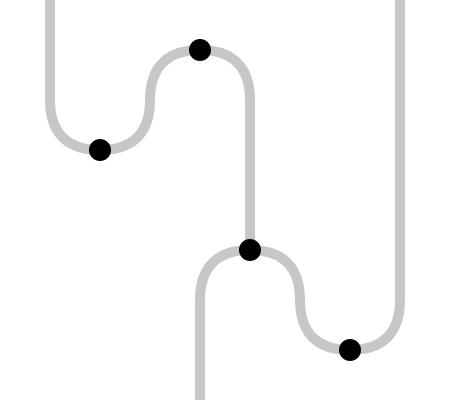}
		\overset{\eqref{zigzagidentities2}}{\cong} \includegraphics[height=1cm]{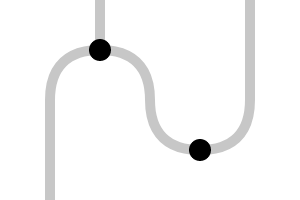}.
	\end{split}	
	\end{align}
	Finally, the coassociativity is witnessed by the composition of invertible $2$-cells
	\begin{align*}
		\includegraphics[height=1cm]{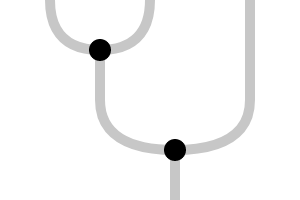}
		&\overset{\eqref{comultrewrite}}{\cong} \includegraphics[height=2cm]{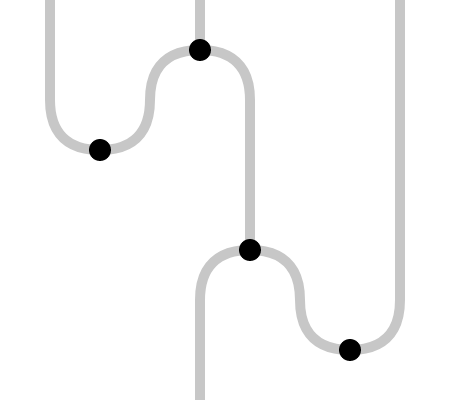}
		\overset{\eqref{intertwiner}}{\cong} \includegraphics[height=2cm]{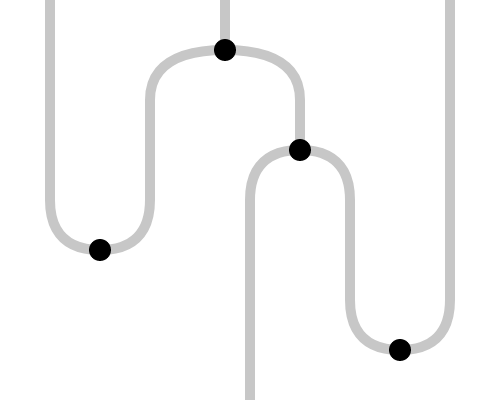}
		\overset{\eqref{pseudomonoassoc}}{\cong} \includegraphics[height=2cm]{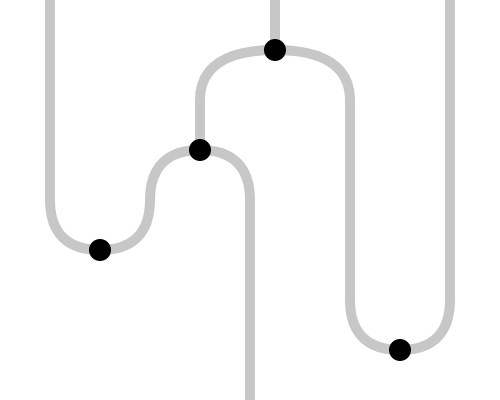} \\
		&= \includegraphics[height=1.5cm]{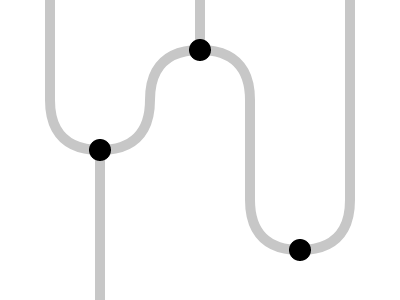}
		\overset{\eqref{intertwiner}}{\cong} \includegraphics[height=1.5cm]{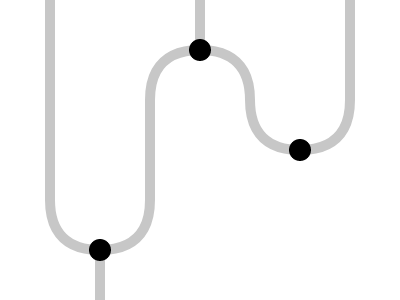}
		\overset{\eqref{comultrewrite}}{\cong} \includegraphics[height=1cm]{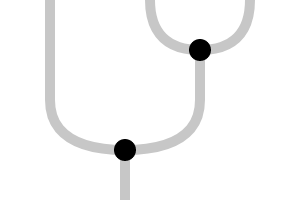}.		
	\end{align*}
The $2$-cells $\eqref{pseudocomonoid2cell}$ are given by the compositions
\begin{equation*}
	\includegraphics[height=1cm]{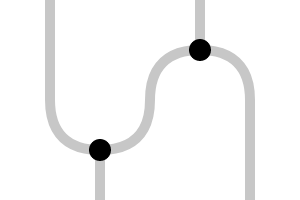}
= 	\includegraphics[height=1.5cm]{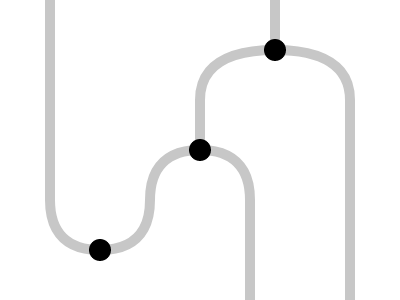}
\overset{\eqref{pseudomonoassoc}}{\cong} 	\includegraphics[height=1.5cm]{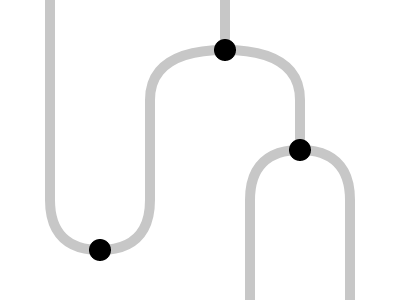}
\overset{\eqref{intertwiner}}{\cong} 	\includegraphics[height=1.5cm]{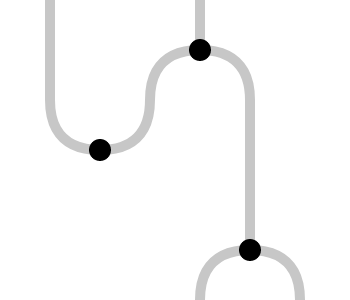}
=	\includegraphics[height=1cm]{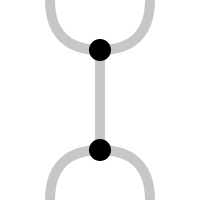},
\end{equation*}	
\begin{equation*}
	\includegraphics[height=1cm]{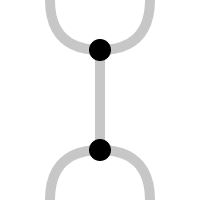}
	\overset{\eqref{comultrewrite}}{\cong} \includegraphics[height=1.5cm]{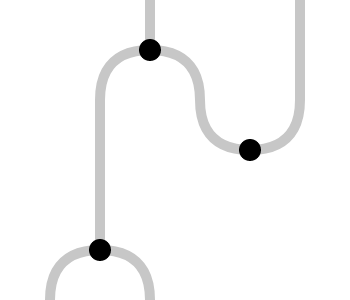}
	\overset{\eqref{intertwiner}}{\cong} \includegraphics[height=1.5cm]{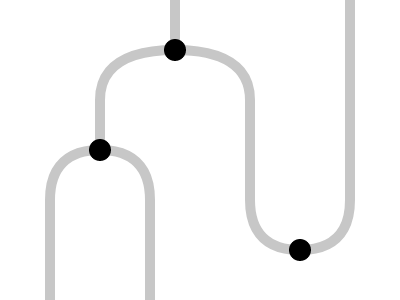}
	\overset{\eqref{pseudomonoassoc}}{\cong} \includegraphics[height=1.5cm]{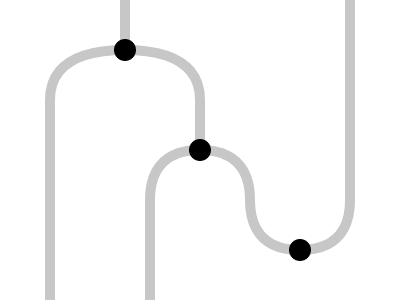}
	\overset{\eqref{comultrewrite}}{\cong} \includegraphics[height=1cm]{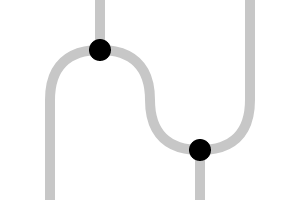}.
\end{equation*}
\end{proof}

\subsection*{Proof of \autoref{gammamonoidal}}

\begin{lemma*}
	Let $\mathcal{X}$ be a finite linear category and $F,G \in \mathcal{L}ex(\mathcal{X}, \mathcal{X})$ be two left exact functors. Then 
	\begin{equation}
		\Gamma^{\textnormal{rl}}(F) \circ \Gamma^{\textnormal{rl}}(G) \cong \Gamma^{\textnormal{rl}}(\Gamma^{\textnormal{rl}}(F) \circ G).
	\end{equation}
	On the other hand, let $F,G \in \mathcal{R}ex(\mathcal{X}, \mathcal{X})$ be two right exact functors. Then 
		\begin{equation}
		\Gamma^{\textnormal{lr}}(F) \circ \Gamma^{\textnormal{lr}}(G) \cong \Gamma^{\textnormal{lr}}(\Gamma^{\textnormal{lr}}(F) \circ G).
	\end{equation}
\end{lemma*}
\begin{proof}
	For $y \in \mathcal{X}$ one computes 
	\begin{align*}
		\Gamma^{\textnormal{rl}}(F) \circ \Gamma^{\textnormal{rl}}(G) (y) 
		&\overset{\eqref{gammaexplicitend}}{=} 
		\Gamma^{\textnormal{rl}}(F) \big( \int^{x \in \mathcal{X}}\ _{\mathcal{X}} \langle y, x \rangle ^* \otimes G(x) \big) \\
		&\overset{\eqref{gammaexplicitend}}{=}  \int^{x' \in \mathcal{X}}\ _{\mathcal{X}} \big \langle \int^{x \in \mathcal{X}}\ _{\mathcal{X}} \langle y, x \rangle ^* \otimes G(x), x' \big \rangle^* \otimes F(x') \\
		&\overset{\eqref{coendcontineq}}{\cong} \int^{x' \in \mathcal{X}} \big (\int_{x \in \mathcal{X}}\ _{\mathcal{X}} \big \langle  _{\mathcal{X}} \langle y, x \rangle ^* \otimes G(x), x' \big \rangle \big)^* \otimes F(x') \\
		&\cong \int^{x' \in \mathcal{X}} \int^{x \in \mathcal{X}}\ _{\mathcal{X}} \big \langle  _{\mathcal{X}} \langle y, x \rangle ^* \otimes G(x), x' \big \rangle^*  \otimes F(x') \\
&\overset{\eqref{finitemodulecat}}{\cong} \int^{x' \in \mathcal{X}} \int^{x \in \mathcal{X}}\ _{\mathcal{X}} \langle y, x \rangle ^* \otimes\ _{\mathcal{X}} \langle  G(x), x' \rangle^*  \otimes F(x') \\
&\overset{\textnormal{Fubini}}{\cong} \int^{x \in \mathcal{X}} \int^{x' \in \mathcal{X}}\ _{\mathcal{X}} \langle y, x \rangle ^* \otimes\ _{\mathcal{X}} \langle  G(x), x' \rangle^*  \otimes F(x') \\
&\cong \int^{x \in \mathcal{X}}\ _{\mathcal{X}} \langle y, x \rangle ^* \otimes\ \int^{x' \in \mathcal{X}}\ _{\mathcal{X}} \langle  G(x), x' \rangle^*  \otimes F(x') \\
&\overset{\eqref{gammaexplicitend}}{=}  \int^{x \in \mathcal{X}}\ _{\mathcal{X}} \langle y, x \rangle ^* \otimes\ \Gamma^{\textnormal{rl}}(F)(G(x)) \overset{\eqref{gammaexplicitend}}{=}  \Gamma^{\textnormal{rl}}(\Gamma^{\textnormal{rl}}(F)\circ G)(y)
	\end{align*}
which proves the first claim. For the second claim we compute similarly
		\begin{align*}
		\Gamma^{\textnormal{lr}}(F) \circ \Gamma^{\textnormal{lr}}(G) (y) 
		&\overset{\eqref{gammaexplicitend}}{=} \Gamma^{\textnormal{lr}}(F) \big( \int_{x \in \mathcal{X}}\ _{\mathcal{X}} \langle x, y \rangle \otimes G(x) \big) \\
		&\overset{\eqref{gammaexplicitend}}{=} \int_{x' \in \mathcal{X}}\ _{\mathcal{X}} \big \langle x', \int_{x \in \mathcal{X}}\ _{\mathcal{X}} \langle x,y \rangle \otimes G(x) \big \rangle \otimes F(x') \\
		&\overset{\eqref{endcontineq}}{\cong} \int_{x' \in \mathcal{X}} \int_{x \in \mathcal{X}}\ _{\mathcal{X}} \big \langle x',  _{\mathcal{X}} \langle x,y \rangle \otimes G(x) \big \rangle \otimes F(x') \\
		&\overset{\eqref{finitemodulecat}}{\cong} \int_{x' \in \mathcal{X}} \int_{x \in \mathcal{X}}\ _{\mathcal{X}} \langle x,y \rangle \otimes\ _{\mathcal{X}} \big \langle x', G(x) \big \rangle \otimes F(x') \\
		&\overset{\textnormal{Fubini}}{\cong} \int_{x \in \mathcal{X}} \int_{x' \in \mathcal{X}}\ _{\mathcal{X}} \langle x,y \rangle \otimes\ _{\mathcal{X}} \big \langle x', G(x) \big \rangle \otimes F(x') \\
		&\cong \int_{x \in \mathcal{X}}\ _{\mathcal{X}} \langle x,y \rangle \otimes\ \int_{x' \in \mathcal{X}}\  _{\mathcal{X}} \big \langle x', G(x) \big \rangle \otimes F(x') \\
		&\overset{\eqref{gammaexplicitend}}{=} \int_{x \in \mathcal{X}}\ _{\mathcal{X}} \langle x,y \rangle \otimes\ \Gamma^{\textnormal{lr}}(F)(G(x))\overset{\eqref{gammaexplicitend}}{=} \Gamma^{\textnormal{lr}}(\Gamma^{\textnormal{lr}}(F)\circ G)(y).
	\end{align*}
\end{proof}

\chapter*{Notation}
\newcolumntype{Y}{>{\centering\arraybackslash}X}

\begin{tabularx}{\textwidth}{| Y | Y | Y |}
\hline
$\mathcal{C}^{\textnormal{opp}(1)}$ & Dual category with inverted morphisms & Section 2.1 \\
\hline
$F^{\textnormal{opp}(1)}$ & Dual functor & Section 2.1 \\
\hline
$\mathcal{C}^{\textnormal{opp}(0)}$ & Opposite category with twisted tensor product & \autoref{oppositecategory} \\
\hline
$\mathcal{C}^{\textnormal{opp}(0,1)}$ & Defined as $(\mathcal{C}^{\textnormal{opp}(0)})^{\textnormal{opp}(1)} \cong (\mathcal{C}^{\textnormal{opp}(1)})^{\textnormal{opp}(0)}$ & Section 2.1 \\
\hline
$_{\mathcal{C}} \langle -,- \rangle$ & Defined as $\textnormal{Hom}_{\mathcal{C}}(-,-)$ & / \\
\hline
$\Cat$ & Category of small categories &/ \\
\hline
$\Set$ & Category of sets &/ \\
\hline
$\Vect, \vect$ & Category of (f.d.) $\mathbbm{k}$-vector spaces &/ \\
\hline
$A\Repfd , A\modl$ & Category of f.d. left $A$-modules & Sec. 4.1.1 \\
\hline
$\underline{\textnormal{Hom}}, \underline{\textnormal{Hom}}^r, \underline{\textnormal{Hom}}^l$ & Internal homs & \autoref{internalhomdef} \\
\hline
$\Vcat$ &Bicat. of $\mathcal{V}$-enriched cat. and $\mathcal{V}$-functors & \autoref{vcat} \\
\hline
$\Vmod$ &Bicat. of $\mathcal{V}$-enriched cat. and modules & \autoref{vmod} \\
\hline
$\Prof$ & Bicat. of profunctors & \autoref{vmod} \\
\hline
$\mathcal{L}ex(\mathcal{C}, \mathcal{D}), \mathcal{R}ex(\mathcal{C}, \mathcal{D})$ & Category of left (right) exact functors $\mathcal{C} \rightarrow \mathcal{D}$ & Section 4.1.1 \\
\hline
$\int_{c \in \mathcal{C}} F(\bar{c}, c) \in \mathcal{D}$ & End of a functor $F: \mathcal{C}^{\textnormal{opp}(1)} \times \mathcal{C} \rightarrow \Set$. & \autoref{defend} \\
\hline
$\int^{c \in \mathcal{C}} F(\bar{c}, c) \in \mathcal{D}$ & Coend of a functor $F: \mathcal{C}^{\textnormal{opp}(1)} \times \mathcal{C} \rightarrow \Set$. & \autoref{defend} \\
\hline
$\textnormal{N}^{\textnormal{r}}_A, \textnormal{N}^{\textnormal{l}}_A$ & Classical Nakayama functors & Section 4.1.3 \\
\hline
$\textnormal{N}^{\textnormal{r}}_{\mathcal{X}}, \textnormal{N}^{\textnormal{l}}_{\mathcal{X}}$ &Nakayama functors of finite linear category $\mathcal{X}$ & \autoref{generalnakayamafunctor} \\
\hline
\end{tabularx}
\pagestyle{plain}

\begin{tabularx}{\textwidth}{| Y | Y | Y |}
\hline
$\QF$ & Category of quadratic forms $G \rightarrow \mathbbm{k}^{\times}$ & \autoref{classicalquadraticformdef} \\
\hline
	$H^{3}_{\textnormal{ab}}(G, \mathbbm{k}^{\times})$ & Third abelian cohomology group & \autoref{abeliangroupcohomology} \\
\hline
$\GVect$ & Category of $G$-graded vector spaces over $\mathbbm{k}$ & \autoref{gradedvecdef} \\
\hline
$\Gvect$ & Category of $G$-graded f.d. vector spaces over $\mathbbm{k}$ & \autoref{gradedvecdef} \\
\hline
$\RMS$ & Ribbon monoidal structures on $\Gvectc$ & \autoref{classicalribbonmonoidalstructures} \\
\hline
$(-)^{g_0}$ & Duality functor on $\Gvect$ w.r.t. $g_0 \in G$ & \autoref{generaldualitygraded} \\
\hline
$\otimes_{g_0}$ & Tensor product on $\Gvect$ w.r.t. $g_0 \in G$ & \autoref{generaltensordefdef} \\
\hline
$\WQF$ & Weak quadratic forms $G \rightarrow \mathbbm{k}^{\times}$ &\autoref{weakquadraticformdef} \\
\hline
$\WRQF$ & Weak representable quadratic forms $G \rightarrow \mathbbm{k}^{\times}$ & \autoref{weakrepresentableqf}  \\
\hline
$\WSQF$ & Weak symmetric quadratic forms $G \rightarrow \mathbbm{k}^{\times}$ & \autoref{weaksymmqf} \\
\hline
$\WRMS$ & Weak ribbon monoidal structures on $\Gvectc$ & \autoref{weakribbonmonoidal} \\
\hline
$\chu$ & Category of separated and extensional pairs & \autoref{chudef} \\
\hline
$\TVS$ & Cat. of Hausd. loc. convex topol. vect. spaces &\autoref{tvsdef} \\
\hline
$V'$ & Topol. dual of contin. functionals $V \rightarrow \mathbbm{k}$ & \autoref{weakstrongdef} \\
\hline
$\sigma(V, W, \langle -,- \rangle)$ & Weakest dual top. on $V$ w.r.t $(V, W, \langle - , - \rangle)$ & \autoref{weakstrongdef} \\
\hline
$\tau(V, W, \langle -, - \rangle)$ & Strongest dual top. on $V$ w.r.t. $(V, W, \langle - , - \rangle)$ & \autoref{weakstrongdef} \\
\hline
$\sigma(V, V')$ & Weak topology on $V$ & \autoref{weaklymackeydef} \\
\hline
$\sigma(V', V)$ &Weak$^\star$ topology on $V'$ &\autoref{weaklymackeydef} \\
\hline
$\tau(V, V')$ & Mackey topology on $V$ & \autoref{weaklymackeydef} \\
\hline
$\TVSm$ & Category of Mackey spaces &\autoref{mackeyweaklycatdef} \\
\hline
$\TVSw$ & Category of weakly topologized spaces &\autoref{mackeyweaklycatdef}\\
\hline
\end{tabularx}

\thispagestyle{plain}

\bibliography{verzeichnis1}
\bibliographystyle{alpha}

\newpage
\thispagestyle{empty}
\vspace*{4em}
\noindent \Large{\textbf{Erklärung}}
\vspace*{1em}\\
\noindent \normalsize Die vorliegende Arbeit habe ich selbständig verfasst und keine anderen als die angegebenen Hilfsmittel - insbesondere keine im Quellenverzeichnis nicht benannten Internet-Quellen - benutzt. Die Arbeit habe ich vorher nicht in einem anderen Prüfungsverfahren eingereicht. Die eingereichte schriftliche Fassung entspricht genau der auf dem elektronischen Speichermedium.
\vspace*{1em}\\
\noindent Hamburg, den 13. September 2018,
\vspace*{4em}\\
Stefan Zetzsche
\end{document}